\documentclass[a4paper,11pt]{article}

\pdfoutput=1

\usepackage[utf8]{inputenc}
\usepackage[T1]{fontenc}
\usepackage[UKenglish]{babel}
\usepackage{amsmath,amsthm,amssymb,amsfonts}
\usepackage[a4paper,margin=1in]{geometry}
\usepackage{graphicx}
\usepackage{url}
\usepackage[hidelinks]{hyperref}
\usepackage{mathtools, nccmath}
\usepackage{tikz}
\usepackage{tikzit}
\usepackage{stmaryrd}
\usepackage{physics}
\usepackage{float}

\tikzstyle{doubled}=[line width=1.5pt] % set the line width for all doubled (quantum) maps/wires

\tikzstyle{dot}=[inner sep=0mm,minimum width=2mm,minimum height=2mm,draw,shape=circle]  
\tikzstyle{ddot}=[inner sep=0mm, doubled, minimum width=2.5mm,minimum height=2.5mm,draw,shape=circle]

\tikzstyle{pdot}=[inner sep=0mm, doubled, minimum width=2.5mm,minimum height=2.5mm,shape=circle]
\tikzstyle{phase dimensions}=[minimum size=6mm,font=\footnotesize,inner sep=0.2mm,outer sep=-2mm]

\tikzstyle{phase dot}=[pdot,phase dimensions]
\tikzstyle{wphase dot}=[dot, phase dimensions]

\tikzstyle{hadamard}=[fill=white,draw,inner sep=0.6mm,font=\footnotesize,minimum height=6mm,minimum width=8mm]

\tikzstyle{anti} = [fill=white,draw,inner sep=0.6mm,font=\footnotesize,minimum height=3mm,minimum width=3mm]

\tikzstyle{triang}=[regular polygon,regular polygon sides=3,draw,scale=0.75,inner sep=-0.75pt,minimum width=9mm,fill=white,regular polygon rotate=180]
\tikzstyle{triang_lesssep}=[regular polygon,regular polygon sides=3,draw,scale=0.75,inner sep=-4pt,minimum width=9mm,fill=white,regular polygon rotate=180, text depth=4mm]
\tikzstyle{triangdag}=[regular polygon,regular polygon sides=3,draw,scale=0.75,inner sep=-0.5pt,minimum width=9mm,fill=white]

\makeatletter
\newcommand{\boxshape}[3]{%
\pgfdeclareshape{#1}{
\inheritsavedanchors[from=rectangle] % this is nearly a rectangle
\inheritanchorborder[from=rectangle]
\inheritanchor[from=rectangle]{center}
\inheritanchor[from=rectangle]{north}
\inheritanchor[from=rectangle]{south}
\inheritanchor[from=rectangle]{west}
\inheritanchor[from=rectangle]{east}
\backgroundpath{% this is new
\southwest \pgf@xa=\pgf@x \pgf@ya=\pgf@y
\northeast \pgf@xb=\pgf@x \pgf@yb=\pgf@y

\@tempdima=#2
\@tempdimb=#3

\pgfpathmoveto{\pgfpoint{\pgf@xa - 5pt + \@tempdima}{\pgf@ya}}
\pgfpathlineto{\pgfpoint{\pgf@xa - 5pt - \@tempdima}{\pgf@yb}}
\pgfpathlineto{\pgfpoint{\pgf@xb + 5pt + \@tempdimb}{\pgf@yb}}
\pgfpathlineto{\pgfpoint{\pgf@xb + 5pt - \@tempdimb}{\pgf@ya}}
\pgfpathlineto{\pgfpoint{\pgf@xa - 5pt + \@tempdima}{\pgf@ya}}
\pgfpathclose
}
}}

\boxshape{NEbox}{0pt}{5pt}
\boxshape{SEbox}{0pt}{-5pt}
\boxshape{NWbox}{5pt}{0pt}
\boxshape{SWbox}{-5pt}{0pt}
\boxshape{EBox}{-3pt}{3pt}
\boxshape{WBox}{3pt}{-3pt}
\makeatother

\tikzstyle{map}=[draw,shape=NEbox,inner sep=2pt,minimum height=6mm,fill=white]
\tikzstyle{mapdag}=[draw,shape=SEbox,inner sep=2pt,minimum height=6mm,fill=white]
\tikzstyle{maptrans}=[draw,shape=SWbox,inner sep=2pt,minimum height=6mm,fill=white]
\tikzstyle{mapconj}=[draw,shape=NWbox,inner sep=2pt,minimum height=6mm,fill=white]

\tikzstyle{dmap}=[draw,doubled,shape=NEbox,inner sep=2pt,minimum height=6mm,fill=white]
\tikzstyle{dmapdag}=[draw,doubled,shape=SEbox,inner sep=2pt,minimum height=6mm,fill=white]
\tikzstyle{dmaptrans}=[draw,doubled,shape=SWbox,inner sep=2pt,minimum height=6mm,fill=white]
\tikzstyle{dmapconj}=[draw,doubled,shape=NWbox,inner sep=2pt,minimum height=6mm,fill=white]

\makeatletter
\pgfdeclareshape{cornerpoint}{
\inheritsavedanchors[from=rectangle] % this is nearly a rectangle
\inheritanchorborder[from=rectangle]
\inheritanchor[from=rectangle]{center}
\inheritanchor[from=rectangle]{north}
\inheritanchor[from=rectangle]{south}
\inheritanchor[from=rectangle]{west}
\inheritanchor[from=rectangle]{east}
\backgroundpath{% this is new
\southwest \pgf@xa=\pgf@x \pgf@ya=\pgf@y
\northeast \pgf@xb=\pgf@x \pgf@yb=\pgf@y

\pgfmathsetmacro{\pgf@shorten@left}{\pgfkeysvalueof{/tikz/shorten left}}
\pgfmathsetmacro{\pgf@shorten@right}{\pgfkeysvalueof{/tikz/shorten right}}

\pgfpathmoveto{\pgfpoint{0.5 * (\pgf@xa + \pgf@xb)}{\pgf@ya - 5pt}}
\pgfpathlineto{\pgfpoint{\pgf@xa - 8pt + \pgf@shorten@left}{\pgf@yb - 1.5 * \pgf@shorten@left}}
\pgfpathlineto{\pgfpoint{\pgf@xa - 8pt + \pgf@shorten@left}{\pgf@yb}}
\pgfpathlineto{\pgfpoint{\pgf@xb + 8pt - \pgf@shorten@right}{\pgf@yb}}
\pgfpathlineto{\pgfpoint{\pgf@xb + 8pt - \pgf@shorten@right}{\pgf@yb - 1.5 * \pgf@shorten@right}}
\pgfpathclose
}
}

\pgfdeclareshape{cornercopoint}{
\inheritsavedanchors[from=rectangle] % this is nearly a rectangle
\inheritanchorborder[from=rectangle]
\inheritanchor[from=rectangle]{center}
\inheritanchor[from=rectangle]{north}
\inheritanchor[from=rectangle]{south}
\inheritanchor[from=rectangle]{west}
\inheritanchor[from=rectangle]{east}
\backgroundpath{% this is new
\southwest \pgf@xa=\pgf@x \pgf@ya=\pgf@y
\northeast \pgf@xb=\pgf@x \pgf@yb=\pgf@y

\pgfmathsetmacro{\pgf@shorten@left}{\pgfkeysvalueof{/tikz/shorten left}}
\pgfmathsetmacro{\pgf@shorten@right}{\pgfkeysvalueof{/tikz/shorten right}}

\pgfpathmoveto{\pgfpoint{0.5 * (\pgf@xa + \pgf@xb)}{\pgf@yb + 5pt}}
\pgfpathlineto{\pgfpoint{\pgf@xa - 8pt + \pgf@shorten@left}{\pgf@ya + 1.5 * \pgf@shorten@left}}
\pgfpathlineto{\pgfpoint{\pgf@xa - 8pt + \pgf@shorten@left}{\pgf@ya}}
\pgfpathlineto{\pgfpoint{\pgf@xb + 8pt - \pgf@shorten@right}{\pgf@ya}}
\pgfpathlineto{\pgfpoint{\pgf@xb + 8pt - \pgf@shorten@right}{\pgf@ya + 1.5 * \pgf@shorten@right}}
\pgfpathclose
}
}

\makeatother

\pgfkeyssetvalue{/tikz/shorten left}{0pt}
\pgfkeyssetvalue{/tikz/shorten right}{0pt}

\tikzstyle{kpoint common}=[draw,fill=white,inner sep=1pt,minimum height=4mm]
\tikzstyle{kpoint}=[shape=cornerpoint,shorten left=5pt,kpoint common]
\tikzstyle{kpoint adjoint}=[shape=cornercopoint,shorten left=5pt,kpoint common]
\tikzstyle{kpoint conjugate}=[shape=cornerpoint,shorten right=5pt,kpoint common]
\tikzstyle{kpoint transpose}=[shape=cornercopoint,shorten right=5pt,kpoint common]

\tikzstyle{kpointdag}=[kpoint adjoint]
\tikzstyle{kpointadj}=[kpoint adjoint]
\tikzstyle{kpointconj}=[kpoint conjugate]
\tikzstyle{kpointtrans}=[kpoint transpose]

\tikzstyle{big kpoint}=[kpoint, minimum width=1.0 cm, minimum height=2mm, inner sep=4pt, text depth=1.5mm]

 \tikzstyle{upground}=[circuit ee IEC,thick,ground,rotate=90,scale=1.5]
 \tikzstyle{downground}=[circuit ee IEC,thick,ground,rotate=-90,scale=1.5]

\tikzstyle{smallcirc}=[circle,fill=white,draw=black]
\tikzstyle{plain}=[-,draw=black,line width=2.000]
\tikzstyle{process}=[rectangle,fill=white,draw=black]
\tikzstyle{none}=[inner sep=0pt]
\tikzstyle{discarding}=[fill=white, draw=black, shape=circle, style=upground]
\tikzstyle{smalldiscarding}=[fill=white, draw=black, style=upground, scale=0.5]
\tikzstyle{backdiscard}=[fill=white, draw=black, shape=circle, style=downground, scale=0.5]
\tikzstyle{smallbackdiscard}=[fill=white, draw=black, shape=circle, style=downground, scale=0.5]
\tikzstyle{state}=[fill=white, draw=black, style=triang, tikzit shape=rectangle]
\tikzstyle{kstate}=[fill=white, draw=black, style=kpoint, tikzit shape=rectangle]
\tikzstyle{kstateconj}=[fill=white, draw=black, style=kpoint conjugate, tikzit shape=rectangle]
\tikzstyle{kstateBIG}=[fill=white, draw=black, style=big kpoint, tikzit shape=rectangle]
\tikzstyle{effect}=[fill=white, draw=black, style=triangdag]
\tikzstyle{keffect}=[fill=white, draw=black, style=kpoint adjoint]
\tikzstyle{keffectconj}=[fill=white, draw=black, style=kpoint transpose]
\tikzstyle{morphdag}=[style=mapdag]
\tikzstyle{morph}=[style=hadamard]
\tikzstyle{WIDEmorph}=[style=hadamard, minimum width=14mm]
\tikzstyle{morphtrans}=[style=maptrans]
\tikzstyle{morphconj}=[style=mapconj]
\tikzstyle{CPMmorph}=[style=dmap]
\tikzstyle{CPMmorphconj}=[style=dmapconj]
\tikzstyle{CPMmorphdag}=[style=dmapdag]
\tikzstyle{CPMmorphtrans}=[style=dmaptrans]
\tikzstyle{CPMstate}=[fill=white, draw=black, style=triang, doubled]
\tikzstyle{CPMstateBIG}=[fill=white, draw=black, style={triang_lesssep}, doubled]
\tikzstyle{CPMkstate}=[fill=white, draw=black, style=kpoint, tikzit shape=rectangle, doubled]
\tikzstyle{CPMkstateconj}=[fill=white, draw=black, style=kpoint conjugate, tikzit shape=rectangle, doubled]
\tikzstyle{CPMkstateBIG}=[fill=white, draw=black, style=big kpoint, tikzit shape=rectangle, doubled]
\tikzstyle{CPMkeffect}=[fill=white, draw=black, style=kpoint adjoint, doubled]
\tikzstyle{CPMkeffectconj}=[fill=white, draw=black, style=kpoint transpose, doubled]
\tikzstyle{UHfB}=[fill=white, draw=black, style=triangdag, doubled, inner sep=-2pt]
\tikzstyle{leak}=[style=tinypoint, regular polygon rotate=-90]
\tikzstyle{leakfill}=[style=tinypoint, regular polygon rotate=-90, fill=black]
\tikzstyle{Z}=[style=dot, fill=green]
\tikzstyle{X}=[style=dot, fill=red]
\tikzstyle{black_dot}=[style=dot, fill=black]
\tikzstyle{white_dot}=[style=dot, fill=white]
\tikzstyle{qblack_dot}=[style=ddot, fill=black]
\tikzstyle{qwhite_dot}=[style=ddot, fill=white]
\tikzstyle{whitephase}=[style=wphase dot, fill=white]
\tikzstyle{qredphase}=[style=phase dot, fill=red]
\tikzstyle{qgreenphase}=[style=phase dot, fill=green]
\tikzstyle{had}=[style=hadamard, doubled]
\tikzstyle{box}=[style=hadamard]
\tikzstyle{classhad}=[style=hadamard]
\tikzstyle{antipode}=[style=anti]

\tikzstyle{dottededge}=[-, dotted]
\tikzstyle{double edge}=[-, style=doubled, draw=black, tikzit draw={rgb,255: red,18; green,168; blue,191}]
\tikzstyle{new edge style 0}=[<-]
\tikzstyle{new edge style 1}=[-, draw={rgb,255: red,223; green,66; blue,126}, fill={rgb,255: red,223; green,66; blue,126}]
\tikzstyle{new edge style 2}=[-, draw={rgb,255: red,14; green,188; blue,83}]
\tikzstyle{new edge style 3}=[<-, draw={rgb,255: red,223; green,66; blue,126}]
\tikzstyle{new edge style 4}=[<-, draw={rgb,255: red,0; green,128; blue,128}]
\tikzstyle{new edge style 5}=[-, draw={rgb,255: red,214; green,110; blue,62}]
\tikzstyle{new edge style 6}=[-, draw={rgb,255: red,174; green,20; blue,174}]
\tikzstyle{new edge style 7}=[-]

\usetikzlibrary{calc, positioning, shapes.geometric}
\usetikzlibrary{
	arrows,
	shapes,
	decorations,
	intersections,
	backgrounds,
	positioning,
	circuits.ee.IEC
	}

\newcommand{\tikzfigscale}[2]{\scalebox{#1}{\input{#2.tikz}}}

\newcommand{\cat}{\mathbf}

\newcommand{\putt}[1]{\mathbin{\uparrow_{#1}}}
\newcommand{\get}{\raisebox{-0.01cm}{\includegraphics[scale=0.02]{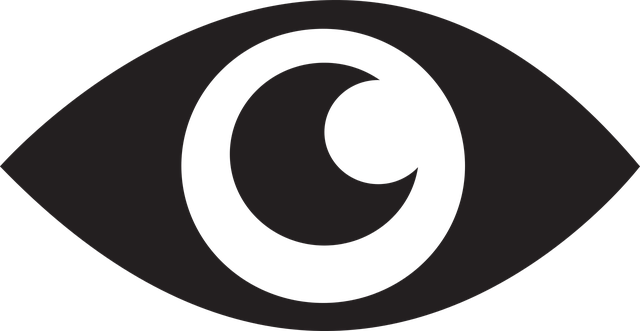}}}
\newcommand{\mix}[1]{\sim_{#1}}

\newcommand{\wfrob}{
\mathbin{\begin{tikzpicture}
	\begin{pgfonlayer}{nodelayer}
		\node [style={white_dot}, scale=0.6] (0) at (0, 0) {};
		\node [style=none] (1) at (0, 0.2) {};
		\node [style=none] (2) at (-0.2, -0.2) {};
		\node [style=none] (3) at (0.2, -0.2) {};
	\end{pgfonlayer}
	\begin{pgfonlayer}{edgelayer}
		\draw (1.center) to (0);
		\draw [in=90, out=-150] (0) to (2.center);
		\draw [in=-270, out=-30] (0) to (3.center);
	\end{pgfonlayer}
\end{tikzpicture}}}

\theoremstyle{definition}
\newtheorem{defn}{Definition}
\theoremstyle{plain}
\newtheorem{prop}{Proposition}
\theoremstyle{plain}
\newtheorem{corr}{Corollary}
\theoremstyle{definition}
\newtheorem{example}{Example}

\title{Categories of tagged lenses}
\author{Matthew Wilson\\
Universit\'e Paris-Saclay, CNRS, ENS Paris-Saclay, Inria,\\
Laboratoire M\'ethodes Formelles,\\
91190 Gif-sur-Yvette, France\\
\texttt{matthew.wilson@centralesupelec.fr}}
\date{}

\begin{document}

\maketitle

\begin{abstract}
To better understand the landscape of well behaviour laws for lenses on strict cartesian symmetric monoidal categories, we define \textit{tagged lenses}. We prove that tagged lenses form a strict symmetric monoidal category, equipped with a strict symmetric monoidal functor into the category of \textbf{putget} lenses. We then observe compositional entailment for the other two lens laws, identifying change-dependence and first-last dependence of tags as sufficient conditions for \textbf{getput} and \textbf{putput} respectively.     % As a result we enlarge the hierarchy of compositional lens classes, to include those with side-effects which satisfy particular design laws.
\end{abstract}

\section{Introduction}
Lenses in their various forms can be used to model update and read-out procedures for substructures of global data structures \cite{foster_combinators_nodate, bancilhon_update_1981}. When equipped with so-called  "very-well-behaviour" (vwb) laws, set-based lenses \cite{foster_combinators_nodate} can be interpreted as coming with a guarantee that they perform modification and examination without noise or side-effect. 

Recent works aimed to give meaning to the relaxation of vwb laws by examining dependencies between different relaxed law sets \cite{fischer_lens, nakano, lens_putput}, in which natural examples appear as modifications to very-well-behaved lenses by adding primitive auxiliary tagging mechanisms. This exploration of the interpretation of subsets of the lens laws, whilst enlightening, is however left at the heuristic level without precise formalisation.

In this paper we aim to develop a structural way to understand the meaning of relaxed lens laws by introducing a simple formal notion of auxiliary counter termed a \textit{tag}. Tags are a restricted class of side-effects which can be assigned behaviour laws which entail well-behaviour laws on the lenses from which they are constructed. There are already effective approaches \cite{abou_monads, monadic_combinator, monads_entangled_notions, monadic_composing} to introducing general side effects into lenses, the motivation for the definition of tagged lenses is not to propose an alternative to these more general approaches, but to give a class of lenses which is just broad enough for formal demonstration of the dependency between side-effect behaviour and well-behaviour laws. Our contributions to the theory of set-based, and more generally cartesian monoidal, lenses are:
\begin{itemize}
    \item tagged vwb-lenses form a strict symmetric monoidal category;
    \item the induced lens construction defines a strict symmetric monoidal functor into the category of putget lenses; and
    \item the three standard lens laws arise compositionally: putget for every tagged lens, getput under change-dependence, and putput under first-last dependence.
\end{itemize}
In doing so, we provide a simple formalisation of the interaction between auxiliary counters and lens laws, with the aim of exploring in future work whether this connection extends to more general kinds of lenses with side-effects, to lenses in non-cartesian categories, and to more general optic constructions.

\section{Preliminary Material}
Lenses will be introduced as set-based, that is, as functions between sets. The graphical circuit-based language for cartesian monoidal categories is then introduced. Throughout this paper we will present many definitions and proofs in both standard algebraic language and the graphical language for cartesian monoidal categories. This is with the aim of being accessible across disciplines.
\subsection{Lenses}
Very well behaved lenses \cite{foster_combinators_nodate} consist of a pair $(p:S\times V \rightarrow S,g:S \rightarrow V)$ of functions, with $S$ interpreted as global data structure and $V$ as some substructure. The ``put" $p:S\times V \rightarrow S$ uses each value $v$ of the substructure $V$ to define an update $p_{v}:S \rightarrow S$ on the global structure $S$. The ``get" $g:S \rightarrow V$ extracts from each value $s$ of the global structure $S$ a value $v$ of the substructure $V$. A strict way to enforce that $p,g$ behave like update and read-out is to impose equations between terms that can be built from them.
\begin{defn}[Very-Well-Behaved Lens]
A lens from set $V$ to set $S$ is a pair of functions $(g:S\rightarrow V,p:S \times V \rightarrow S)$ \cite{foster_combinators_nodate}. A lens is further \emph{very-well-behaved} (\emph{vwb}) if it satisfies the following:
\begin{itemize}
    \item \textbf{putput}: $p(p(s,v_1),v_2) = p(s,v_2)$
    \item \textbf{putget}: $g(p(s,v)) = v$
    \item \textbf{getput}: $p(s,g(s)) = s$
\end{itemize}
\end{defn}
Lenses are termed well-behaved (wb) if they satisfy only the last two laws and will here be termed putget (pg) if they satisfy only the \textbf{putget} law. Set-based lenses can be composed in sequence and in parallel to form a symmetric monoidal category \cite{Riley2018CategoriesOO}. We will use the following running examples, all mined or adapted from \cite{fischer_lens, nakano}, noting the lens laws that they satisfy and the intuitive reasons for satisfaction of those laws which this paper aims to abstract and formalise.
\begin{example}[Basic vwb-lens]
\begin{align*}
    S &= U \times V \\
    \texttt{put}((u,v),v') &= (u,v') \\
    g(u,v) &= v
\end{align*}
\end{example}
It is easy to check that this pair defines a vwb-lens.
\begin{example}[PutCount]
\begin{align*}
    S &= \mathbf{N} \times V \\
    \texttt{PutCount}((i,v),v') &= (i+1,v') \\
    g(i,v) &= v
\end{align*}
\end{example}
Note that \texttt{PutCount} satisfies \textbf{putget} because the get does not see the counter. 
\begin{example}[PutCountView]
\begin{align*}
    S &= \Big(\prod_{w \in V} \mathbf{N}\Big) \times V \\
    \texttt{PutCountView}(((i_w)_{w \in V},v),v') &= ((j_w)_{w \in V},v') \\
    g((i_w)_{w \in V},v) &= v
\end{align*}
where $j_{v'} = i_{v'} + 1$ and $j_w = i_w$ for $w \neq v'$.
\end{example}
Note that \texttt{PutCount} satisfies \textbf{putget} because the get does not see the counter.
\begin{example}[PutCountChanges]
\[
\texttt{PutCountChanges}((i,v),v')=
\begin{cases}
(i+1,v') & v \neq v',\\
(i,v') & v = v',
\end{cases}
\qquad
g(i,v)=v.
\]
\end{example}
Note that \texttt{PutCount} satisfies \textbf{putget} because the get does not see the counter, but furthermore, is satisfies \textbf{getput} because only non-trivial changes are recorded by the counter. 
\begin{example}
The lens \texttt{PutCountViewChanges} on $S = \big(\prod_{w \in V} \mathbf{N}\big) \times V$ is obtained from \texttt{PutCountView} by recording only genuine changes in the view:
\[
((i_w)_{w \in V},v),v' \mapsto ((j_w)_{w \in V},v').
\]
When $v \neq v'$, we set $j_{v'} = i_{v'} + 1$ and $j_w = i_w$ for $w \neq v'$. When $v = v'$, we set $j_w = i_w$ for all $w$. Again $g((i_w)_{w \in V},v)=v$.
\end{example}
Note that \texttt{PutCount} satisfies \textbf{putget} because the get does not see the counter, but furthermore, is satisfies \textbf{getput} because only non-trivial changes are recorded by the counter. 
\begin{example}[PutFlag]
\[
\texttt{PutFlag}((b,v),v')=(1,v'),
\qquad
g(b,v)=v.
\]
\end{example}
Note that \texttt{PutCount} satisfies \textbf{putget} because the get does not see the counter, but furthermore, is satisfies \textbf{putput} because the flag only records whether an update has occurred, it does not count the number of updates. 
\begin{example}[PutScaled]
\[
\texttt{PutScaled}((x,y),v)=(v,vy/x).
\]
\end{example}
All of the counting examples satisfy putget because they are induced by tagged lenses. The change-sensitive examples satisfy getput because they record only genuine changes, while \texttt{PutFlag} satisfies putput because its tag is first-last depending. The vwb laws for \texttt{PutScaled} can likewise be understood as consequences of its tag behaviour.
\subsection{String Diagrams for Cartesian Monoidal Categories}
A monoidal category is a category equipped with a notion of parallel composition $\otimes$ of objects and of morphisms, the definitions of categories and monoidal categories can be found here \cite{maclane:71}. The category of functions between sets can be viewed as a monoidal category in which the cartesian product $\times$ takes the place of $\otimes$. Any term of morphisms in a symmetric monoidal category may be expressed as a ``string diagram" in which each object (set) $A$ is represented by a wire, each morphism (function) $f:A \rightarrow B$ is represented by a box with input wire $A$ and output wire $B$, sequential composition $f \circ g$ is expressed by routing the output wire of $g$ to the input wire of $f$ and parallel composition of morphisms is expressed by placing boxes next to each-other.
\[\begin{tikzpicture}[scale=0.7, {every node/.style}={scale=0.7}]
	\begin{pgfonlayer}{nodelayer}
		\node [style=none] (0) at (-3.75, 1.5) {};
		\node [style=none] (1) at (-3.75, 0.5) {};
		\node [style=none] (2) at (-3.25, -1) {$A$};
		\node [style=none] (3) at (-4.25, 0.5) {};
		\node [style=none] (4) at (-3.25, 0.5) {};
		\node [style=none] (5) at (-3.25, -0.5) {};
		\node [style=none] (6) at (-4.25, -0.5) {};
		\node [style=none] (7) at (-3.75, 0) {$f$};
		\node [style=none] (8) at (-3.75, -0.5) {};
		\node [style=none] (9) at (-3.75, -1.25) {};
		\node [style=none] (10) at (-3.25, 1) {$B$};
		\node [style=none] (11) at (-1.75, 1.5) {};
		\node [style=none] (12) at (-1.75, 0.5) {};
		\node [style=none] (13) at (-1.25, -1) {$C$};
		\node [style=none] (14) at (-2.25, 0.5) {};
		\node [style=none] (15) at (-1.25, 0.5) {};
		\node [style=none] (16) at (-1.25, -0.5) {};
		\node [style=none] (17) at (-2.25, -0.5) {};
		\node [style=none] (18) at (-1.75, 0) {$g$};
		\node [style=none] (19) at (-1.75, -0.5) {};
		\node [style=none] (20) at (-1.75, -1.25) {};
		\node [style=none] (21) at (-1.25, 1) {$D$};
		\node [style=none] (24) at (2.5, -2) {$A$};
		\node [style=none] (25) at (1.5, -0.5) {};
		\node [style=none] (26) at (2.5, -0.5) {};
		\node [style=none] (27) at (2.5, -1.5) {};
		\node [style=none] (28) at (1.5, -1.5) {};
		\node [style=none] (29) at (2, -1) {$f$};
		\node [style=none] (30) at (2, -1.5) {};
		\node [style=none] (31) at (2, -2.25) {};
		\node [style=none] (32) at (2.5, 0) {$B$};
		\node [style=none] (33) at (2, 2.5) {};
		\node [style=none] (34) at (2, 1.5) {};
		\node [style=none] (36) at (1.5, 1.5) {};
		\node [style=none] (37) at (2.5, 1.5) {};
		\node [style=none] (38) at (2.5, 0.5) {};
		\node [style=none] (39) at (1.5, 0.5) {};
		\node [style=none] (40) at (2, 1) {$g$};
		\node [style=none] (41) at (2, 0.5) {};
		\node [style=none] (42) at (2, -0.5) {};
		\node [style=none] (43) at (2.5, 2) {$C$};
	\end{pgfonlayer}
	\begin{pgfonlayer}{edgelayer}
		\draw (1.center) to (0.center);
		\draw [style=none] (3.center) to (4.center);
		\draw [style=none] (4.center) to (5.center);
		\draw [style=none] (5.center) to (6.center);
		\draw [style=none] (6.center) to (3.center);
		\draw (9.center) to (8.center);
		\draw (12.center) to (11.center);
		\draw [style=none] (14.center) to (15.center);
		\draw [style=none] (15.center) to (16.center);
		\draw [style=none] (16.center) to (17.center);
		\draw [style=none] (17.center) to (14.center);
		\draw (20.center) to (19.center);
		\draw [style=none] (25.center) to (26.center);
		\draw [style=none] (26.center) to (27.center);
		\draw [style=none] (27.center) to (28.center);
		\draw [style=none] (28.center) to (25.center);
		\draw (31.center) to (30.center);
		\draw (34.center) to (33.center);
		\draw [style=none] (36.center) to (37.center);
		\draw [style=none] (37.center) to (38.center);
		\draw [style=none] (38.center) to (39.center);
		\draw [style=none] (39.center) to (36.center);
		\draw (42.center) to (41.center);
	\end{pgfonlayer}
\end{tikzpicture}\]
The category $\mathbf{Set}$ in which lenses were originally formulated belongs to the special class of \textit{cartesian} monoidal categories, which are monoidal categories equipped with a few additional features \cite{heunen2019categories}
\begin{itemize}
    \item a family of copying maps $c_A: A \rightarrow A \times A$ for each object $A$, expressed by dots, and such that the following holds \[\begin{tikzpicture}[scale=0.7, {every node/.style}={scale=0.7}]
	\begin{pgfonlayer}{nodelayer}
		\node [style={white_dot}] (55) at (-2.25, -0.75) {};
		\node [style=none] (56) at (-3, 0) {};
		\node [style=none] (57) at (-1.5, 0) {};
		\node [style=none] (58) at (-2.25, -1.5) {};
		\node [style={white_dot}] (59) at (0.5, -0.75) {};
		\node [style=none] (60) at (0, 0) {};
		\node [style=none] (61) at (1, 0) {};
		\node [style=none] (62) at (0.5, -1.5) {};
		\node [style={white_dot}] (63) at (1, -0.75) {};
		\node [style=none] (64) at (0.5, 0) {};
		\node [style=none] (65) at (1.5, 0) {};
		\node [style=none] (66) at (1, -1.5) {};
		\node [style=none] (67) at (-0.75, -0.75) {=};
		\node [style=none] (68) at (-2.25, -2) {$A \otimes B$};
		\node [style=none] (69) at (-3, 0.5) {$A \otimes B$};
		\node [style=none] (70) at (-1.5, 0.5) {$A \otimes B$};
		\node [style=none] (71) at (0.5, -2) {$A$};
		\node [style=none] (72) at (0, 0.5) {$A$};
		\node [style=none] (73) at (1, 0.5) {$A$};
		\node [style=none] (74) at (1, -2) {$B$};
		\node [style=none] (75) at (0.5, 0.5) {$B$};
		\node [style=none] (76) at (1.5, 0.5) {$B$};
	\end{pgfonlayer}
	\begin{pgfonlayer}{edgelayer}
		\draw (58.center) to (55);
		\draw [bend left] (55) to (56.center);
		\draw [bend right] (55) to (57.center);
		\draw (62.center) to (59);
		\draw [bend left] (59) to (60.center);
		\draw [bend right] (59) to (61.center);
		\draw (66.center) to (63);
		\draw [bend left] (63) to (64.center);
		\draw [bend right] (63) to (65.center);
	\end{pgfonlayer}
\end{tikzpicture}\]
    \item a deletion map $d:A \rightarrow I$ for each object $A$ \[\begin{tikzpicture}[scale=0.7, {every node/.style}={scale=0.7}]
	\begin{pgfonlayer}{nodelayer}
		\node [style={white_dot}] (77) at (-1.75, 0.5) {};
		\node [style=none] (78) at (-1.75, -0.5) {};
		\node [style=none] (79) at (-1.75, -1) {$A\otimes B$};
		\node [style={white_dot}] (80) at (0.25, 0.5) {};
		\node [style=none] (81) at (0.25, -0.5) {};
		\node [style=none] (82) at (0.25, -1) {$A$};
		\node [style={white_dot}] (83) at (0.75, 0.5) {};
		\node [style=none] (84) at (0.75, -0.5) {};
		\node [style=none] (86) at (-0.75, 0) {$=$};
		\node [style=none] (87) at (0.75, -1) {$B$};
	\end{pgfonlayer}
	\begin{pgfonlayer}{edgelayer}
		\draw (78.center) to (77);
		\draw (81.center) to (80);
		\draw (84.center) to (83);
	\end{pgfonlayer}
\end{tikzpicture}\]
\end{itemize}
the copying and deleting map are additionally required to satisfy the following conditions:
\[\begin{tikzpicture}[scale=0.7, {every node/.style}={scale=0.7}]
	\begin{pgfonlayer}{nodelayer}
		\node [style={white_dot}] (77) at (-2.25, 0) {};
		\node [style=none] (78) at (-2.25, -0.75) {};
		\node [style=none] (86) at (-1, 0) {$=$};
		\node [style=none] (93) at (-2.75, 0.75) {};
		\node [style={white_dot}] (94) at (-1.75, 0.75) {};
		\node [style=none] (105) at (0, 0.5) {};
		\node [style=none] (109) at (0, -0.5) {};
		\node [style={white_dot}] (110) at (2.25, 0) {};
		\node [style=none] (111) at (2.25, -0.75) {};
		\node [style=none] (112) at (1, 0) {$=$};
		\node [style=none] (113) at (2.75, 0.75) {};
		\node [style={white_dot}] (114) at (1.75, 0.75) {};
	\end{pgfonlayer}
	\begin{pgfonlayer}{edgelayer}
		\draw (78.center) to (77);
		\draw [bend left] (94) to (77);
		\draw [bend right] (93.center) to (77);
		\draw (105.center) to (109.center);
		\draw (111.center) to (110);
		\draw [bend right] (114) to (110);
		\draw [bend left] (113.center) to (110);
	\end{pgfonlayer}
\end{tikzpicture}\]
and every morphism $f:A \rightarrow B$ is required to be a comonoid homomorphism from $c_A$ to $c_B$, meaning the following:
\[\begin{tikzpicture}[scale=0.7, {every node/.style}={scale=0.7}]
	\begin{pgfonlayer}{nodelayer}
		\node [style={white_dot}] (77) at (-2.25, 1) {};
		\node [style=box] (78) at (-2.25, 0) {$f$};
		\node [style=none] (86) at (-1, 0) {$=$};
		\node [style=none] (87) at (-2.25, -1) {};
		\node [style={white_dot}] (88) at (0, 0.5) {};
		\node [style=none] (89) at (0, -0.5) {};
	\end{pgfonlayer}
	\begin{pgfonlayer}{edgelayer}
		\draw (78) to (77);
		\draw (87.center) to (78);
		\draw (89.center) to (88);
	\end{pgfonlayer}
\end{tikzpicture} \quad \quad \quad \quad \begin{tikzpicture}[scale=0.7, {every node/.style}={scale=0.7}]
	\begin{pgfonlayer}{nodelayer}
		\node [style={white_dot}] (77) at (-2.5, 0.5) {};
		\node [style=box] (78) at (-2.5, -0.25) {$f$};
		\node [style=none] (86) at (-1, 0) {$=$};
		\node [style=none] (87) at (-2.5, -1) {};
		\node [style=none] (93) at (-3, 1.25) {};
		\node [style=none] (94) at (-2, 1.25) {};
		\node [style={white_dot}] (97) at (0.5, -0.25) {};
		\node [style=none] (98) at (0.5, -1) {};
		\node [style=none] (105) at (0, 1.25) {};
		\node [style=box] (106) at (0, 0.5) {$f$};
		\node [style=none] (109) at (1, 1.25) {};
		\node [style=box] (110) at (1, 0.5) {$f$};
	\end{pgfonlayer}
	\begin{pgfonlayer}{edgelayer}
		\draw (78) to (77);
		\draw (87.center) to (78);
		\draw [bend left] (94.center) to (77);
		\draw [bend right] (93.center) to (77);
		\draw (98.center) to (97);
		\draw (106) to (105.center);
		\draw (110) to (109.center);
		\draw [bend right] (106) to (97);
		\draw [bend left] (110) to (97);
	\end{pgfonlayer}
\end{tikzpicture}\]
In any cartesian monoidal category the lens laws read \cite{Riley2018CategoriesOO}
\begin{equation*}
    \tikzfigscale{0.7}{figs/prop1/lawfullens}
\end{equation*}
%furthermore example (1) is easily generalised to arbitrary Cartesian monoidal categories via the following definitions. 
%\[\tikzfig{tagfigs/simple_put} \quad \quad \quad \quad \tikzfig{tagfigs/simple_get}\]
%and as a quick exercise in using the graphical language note that the getput law follows immediately by \[ \tikzfig{tagfigs/simple_lens}\] as do the other two lens laws. 
%The graphical language for Cartesian monoidal categories is strong enough to prove that this basic example satisfies all three of the vwb-lens laws. 
The notions of $\mathbf{putget}$, $\mathbf{well behaved}$, and $\mathbf{vwb}$-lenses each define symmetric monoidal sub-categories of $\mathbf{L}[\mathcal{C}]$, the symmetric monoidal category of Lenses over $\mathcal{C}$ in which
\begin{itemize}
    \item Objects are the objects of $\mathcal{C}$
    \item Morphisms from $V$ to $V'$ are the lenses from $V$ to $V'$
\end{itemize}
and for which sequential composition \[(\putt{},g):=(\putt{}_2,g_2) \circ (\putt{}_1,g_1)\] is defined by:
\[\begin{tikzpicture}[scale=0.7, {every node/.style}={scale=0.7}]
	\begin{pgfonlayer}{nodelayer}
		\node [style=none] (0) at (-2.75, -0.5) {};
		\node [style=WIDEmorph] (1) at (-2.25, 0.5) {$\putt{}$};
		\node [style=none] (2) at (-2.75, 0.75) {};
		\node [style=none] (3) at (-2.75, 0.25) {};
		\node [style=none] (4) at (-1.75, 0.5) {};
		\node [style=none] (5) at (-1.75, -0.5) {};
		\node [style=none] (6) at (-2.75, 1.5) {};
		\node [style=none] (7) at (-0.75, 0.25) {$=$};
		\node [style=WIDEmorph] (8) at (1.75, 0.5) {$\putt{}_1$};
		\node [style=none] (9) at (1.25, 0.25) {};
		\node [style=morph] (10) at (1.25, -0.5) {$g_2$};
		\node [style=WIDEmorph] (11) at (0.75, 1.75) {$\putt{}_2$};
		\node [style=none] (12) at (0.25, 1.5) {};
		\node [style=none] (13) at (1.25, 1.75) {};
		\node [style=none] (14) at (2.25, -1.25) {};
		\node [style={white_dot}] (15) at (0.5, -1.25) {};
		\node [style=none] (16) at (0.5, -2) {};
		\node [style=none] (17) at (0.25, 2) {};
		\node [style=none] (18) at (0.25, 2.75) {};
		\node [style=none] (19) at (1.25, 0.75) {};
		\node [style=none] (20) at (2.25, 0.5) {};
	\end{pgfonlayer}
	\begin{pgfonlayer}{edgelayer}
		\draw [in=90, out=-90] (4.center) to (5.center);
		\draw (2.center) to (6.center);
		\draw [in=-90, out=90, looseness=1.25] (0.center) to (3.center);
		\draw [in=90, out=-90] (9.center) to (10);
		\draw [in=135, out=-90, looseness=0.75] (12.center) to (15);
		\draw (15) to (16.center);
		\draw [in=-90, out=0] (15) to (10);
		\draw (17.center) to (18.center);
		\draw [in=90, out=-90, looseness=1.25] (13.center) to (19.center);
		\draw [in=90, out=-90] (20.center) to (14.center);
	\end{pgfonlayer}
\end{tikzpicture} \quad \quad \quad \quad \begin{tikzpicture}[scale=0.7, {every node/.style}={scale=0.7}]
	\begin{pgfonlayer}{nodelayer}
		\node [style=morph] (0) at (-1.25, 0) {$g$};
		\node [style=none] (1) at (-1.25, 0.75) {};
		\node [style=none] (2) at (-1.25, -0.75) {};
		\node [style=none] (3) at (-0.25, 0) {$=$};
		\node [style=morph] (4) at (0.75, 0.5) {$g_1$};
		\node [style=none] (5) at (0.75, 1.25) {};
		\node [style=morph] (6) at (0.75, -0.5) {$g_2$};
		\node [style=none] (7) at (0.75, -1.5) {};
	\end{pgfonlayer}
	\begin{pgfonlayer}{edgelayer}
		\draw (1.center) to (0);
		\draw (2.center) to (0);
		\draw (5.center) to (4);
		\draw (4) to (6);
		\draw (7.center) to (6);
	\end{pgfonlayer}
\end{tikzpicture}\]
and for which parallel composition is defined by:
\[\begin{tikzpicture}[scale=0.7, {every node/.style}={scale=0.7}]
	\begin{pgfonlayer}{nodelayer}
		\node [style=none] (187) at (-3, -0.75) {};
		\node [style=WIDEmorph] (188) at (-2.5, 0.25) {$\putt{}$};
		\node [style=none] (189) at (-3, 0.5) {};
		\node [style=none] (190) at (-3, 0) {};
		\node [style=none] (191) at (-2, 0.25) {};
		\node [style=none] (192) at (-2, -0.75) {};
		\node [style=none] (193) at (-3, 1.25) {};
		\node [style=none] (194) at (-1, 0) {$=$};
		\node [style=none] (199) at (-0.25, -1) {};
		\node [style=WIDEmorph] (200) at (0.25, 0.5) {$\putt{}_1$};
		\node [style=none] (201) at (-0.25, 0.75) {};
		\node [style=none] (202) at (-0.25, 0.25) {};
		\node [style=none] (203) at (0.75, 0.5) {};
		\node [style=none] (204) at (2, -1) {};
		\node [style=none] (205) at (-0.25, 1.5) {};
		\node [style=none] (206) at (0.75, -1) {};
		\node [style=WIDEmorph] (207) at (2.25, 0.5) {$\putt{}_2$};
		\node [style=none] (208) at (1.75, 0.75) {};
		\node [style=none] (209) at (1.75, 0.25) {};
		\node [style=none] (210) at (2.75, 0.5) {};
		\node [style=none] (211) at (2.75, -1) {};
		\node [style=none] (212) at (1.75, 1.5) {};
	\end{pgfonlayer}
	\begin{pgfonlayer}{edgelayer}
		\draw [in=90, out=-90] (191.center) to (192.center);
		\draw (189.center) to (193.center);
		\draw [in=-90, out=90, looseness=1.25] (187.center) to (190.center);
		\draw [in=90, out=-90] (203.center) to (204.center);
		\draw (201.center) to (205.center);
		\draw [in=-90, out=90, looseness=1.25] (199.center) to (202.center);
		\draw [in=90, out=-90] (210.center) to (211.center);
		\draw (208.center) to (212.center);
		\draw [in=-90, out=90, looseness=1.25] (206.center) to (209.center);
	\end{pgfonlayer}
\end{tikzpicture} \quad \quad \quad \quad \begin{tikzpicture}[scale=0.7, {every node/.style}={scale=0.7}]
	\begin{pgfonlayer}{nodelayer}
		\node [style=morph] (130) at (-2, 0) {$g$};
		\node [style=none] (131) at (-2, 0.75) {};
		\node [style=none] (132) at (-2, -0.75) {};
		\node [style=none] (184) at (0, 0.75) {};
		\node [style=morph] (185) at (0, 0) {$g_1$};
		\node [style=none] (186) at (0, -0.75) {};
		\node [style=none] (232) at (-1, 0) {$=$};
		\node [style=morph] (233) at (1, 0) {$g_2$};
		\node [style=none] (234) at (1, 0.75) {};
		\node [style=none] (236) at (1, -0.75) {};
	\end{pgfonlayer}
	\begin{pgfonlayer}{edgelayer}
		\draw (131.center) to (130);
		\draw (132.center) to (130);
		\draw (186.center) to (185);
		\draw (234.center) to (233);
		\draw (184.center) to (185);
		\draw (233) to (236.center);
	\end{pgfonlayer}
\end{tikzpicture}\]
The subcategories of $\mathbf{putget}$, $\mathbf{well-behaved}$, and $\mathbf{vwb}$ lenses will be denoted $\mathbf{pgL}[\mathcal{C}]$, $\mathbf{wbL}[\mathcal{C}]$, and $\mathbf{vwbL}[\mathcal{C}]$ respectively.

%Having lost the putput law, we would not know for example, that there really is a view part of the scource which depends only on the most recent update - even though we secretly know there is! The earlier views matter, but only in so much as they influence the tags. Of course one can imagine lenses which drop the putput law for reasons other than that they apply some tag.
\section{Tagged vwb-Lenses}
Our goal is to show that each lens law is entailed in a compositional way by a design law for side-effects; the aim of this first section is to introduce a limited class of side-effects broad enough to demonstrate this point. We model the notion of a \textit{tagging} side-effect, which we abbreviate to \textit{tag}, by abstracting the notion of a family of morphisms indexed by the possible old and new values of the view during update. In the category $\mathbf{Set}$ this is captured by a three-input function $T:S \times V \times V \rightarrow S$, parameterised by two copies of the view system. Graphically such a function is denoted by:
\[\begin{tikzpicture}[scale=0.7, {every node/.style}={scale=0.7}]
	\begin{pgfonlayer}{nodelayer}
		\node [style=WIDEmorph] (0) at (-0.75, 0) {$T$};
		\node [style=none] (1) at (-1.25, 0.25) {};
		\node [style=none] (2) at (-1.25, -0.25) {};
		\node [style=none] (3) at (-0.5, 0) {};
		\node [style=none] (4) at (-0.25, 0) {};
		\node [style=none] (5) at (-0.5, -1) {};
		\node [style=none] (6) at (-0.25, -1) {};
		\node [style=none] (7) at (-1.25, 1) {};
		\node [style=none] (8) at (-1.25, -1) {};
		\node [style=none] (9) at (-3.25, 1.25) {\texttt{System}};
		\node [style=none] (10) at (2, 0.75) {\texttt{Updated View}};
		\node [style=none] (11) at (1.25, -2) {\texttt{Prior View}};
		\node [style=none] (12) at (-3.25, 1) {};
		\node [style=none] (13) at (-1.5, 1) {};
		\node [style=none] (14) at (-1.5, -1) {};
		\node [style=none] (15) at (-0.5, -1.25) {};
		\node [style=none] (16) at (0, -0.75) {};
		\node [style=none] (17) at (1.25, -1.75) {};
		\node [style=none] (18) at (2, 0.5) {};
	\end{pgfonlayer}
	\begin{pgfonlayer}{edgelayer}
		\draw (2.center) to (8.center);
		\draw (3.center) to (5.center);
		\draw (4.center) to (6.center);
		\draw (1.center) to (7.center);
		\draw [style=new edge style 0, in=315, out=-165] (14.center) to (12.center);
		\draw [style=new edge style 0] (13.center) to (12.center);
		\draw [style=new edge style 4, in=195, out=0] (16.center) to (18.center);
		\draw [style=new edge style 3, in=165, out=-90, looseness=0.75] (15.center) to (17.center);
	\end{pgfonlayer}
\end{tikzpicture}\]
This is a family of processes in the sense that for each prior view $p$ and future view $f$ there exists a tagging process $T_{pf}: = T(-,p,f):S \rightarrow S$ given by inserting $p$ and $f$ into $T$. %, written graphically this is: %\[\tikzfig{tagfigs/tagmember}\]
We expect such a family of tags to satisfy one basic condition: \textit{each tag in the family should have no effect on the view property of the system}. Each tag should commute with $\putt{}$ and have no effect on the get $g$. These conditions can be expressed algebraically by
\begin{align*}
    & \putt{}(T(s,v_1,v_2),v_3) = T(\putt{}(s,v_3),v_1,v_2) \\
    & g(T(s,v_1,v_2)) = g(s)
\end{align*}
%\[\tikzfig{tagfigs/tagmembercon1}\]
%    \item Tags should not track the order of events only some metric for the number of occurrences - All members of the family of tags can be commuted amongst themselves. \[\tikzfig{tagfigs/tagmembercon2}\]

When written in point-free notation the above condition leads to the following definition for a tag $T$ with respect to a lens $(\putt{},g)$.
\begin{defn}
For any pair of objects $V$,$S$ in a cartesian monoidal category $\mathcal{C}$ a tuple $(\putt{},g,T:S \times V \times V \rightarrow S)$ is a \textit{tagged vwb-lens} from $V$ to $S$ if $(\putt{},g)$ is a vwb-lens from $V$ to $S$ and the morphisms $\putt{},g,T$ together satisfy the following conditions:
\[\begin{tikzpicture}[scale=0.7, {every node/.style}={scale=0.7}]
	\begin{pgfonlayer}{nodelayer}
		\node [style=WIDEmorph] (0) at (-7.25, -0.5) {$T$};
		\node [style=none] (1) at (-7.75, -0.25) {};
		\node [style=none] (2) at (-7.75, -0.75) {};
		\node [style=none] (3) at (-7, -0.5) {};
		\node [style=none] (4) at (-6.75, -0.5) {};
		\node [style=none] (5) at (-7, -1.5) {};
		\node [style=none] (6) at (-6.75, -1.5) {};
		\node [style=none] (8) at (-7.75, -1.5) {};
		\node [style=WIDEmorph] (9) at (-7, 1) {$\putt{}$};
		\node [style=none] (10) at (-7.5, 1.25) {};
		\node [style=none] (11) at (-7.5, 0.75) {};
		\node [style=none] (13) at (-6.5, 1) {};
		\node [style=none] (15) at (-6, -1.5) {};
		\node [style=none] (16) at (-7.5, 2) {};
		\node [style=none] (17) at (-5.5, 0) {$=$};
		\node [style=WIDEmorph] (18) at (-3.75, 1) {$T$};
		\node [style=none] (19) at (-4.75, 0) {};
		\node [style=none] (20) at (-4.75, -0.5) {};
		\node [style=none] (21) at (-3.5, 1) {};
		\node [style=none] (22) at (-3.25, 1) {};
		\node [style=none] (23) at (-3, -1.5) {};
		\node [style=none] (24) at (-2.75, -1.5) {};
		\node [style=none] (25) at (-4.5, -1.5) {};
		\node [style=WIDEmorph] (26) at (-4.25, -0.25) {$\putt{}$};
		\node [style=none] (27) at (-4.25, 1.25) {};
		\node [style=none] (28) at (-4.25, 0.75) {};
		\node [style=none] (29) at (-3.75, -0.25) {};
		\node [style=none] (30) at (-2.25, -1.5) {};
		\node [style=none] (31) at (-4.25, 2) {};
		\node [style=WIDEmorph] (32) at (-0.25, 0) {$T$};
		\node [style=none] (33) at (-0.75, 0.25) {};
		\node [style=none] (34) at (-0.75, -0.25) {};
		\node [style=none] (35) at (0, 0) {};
		\node [style=none] (36) at (0.25, 0) {};
		\node [style=none] (37) at (0, -1.25) {};
		\node [style=none] (38) at (0.5, -1.25) {};
		\node [style=none] (39) at (-0.75, -1.25) {};
		\node [style=morph] (40) at (-0.75, 1) {$g$};
		\node [style=none] (41) at (-0.75, 1.75) {};
		\node [style=none] (42) at (1, 0) {$=$};
		\node [style=none] (44) at (2, -0.5) {};
		\node [style={white_dot}] (46) at (2.75, 0.5) {};
		\node [style={white_dot}] (47) at (3.25, 0.5) {};
		\node [style=none] (48) at (2.75, -0.5) {};
		\node [style=none] (49) at (3.25, -0.5) {};
		\node [style=morph] (51) at (2, 0.25) {$g$};
		\node [style=none] (52) at (2, 1) {};
	\end{pgfonlayer}
	\begin{pgfonlayer}{edgelayer}
		\draw (2.center) to (8.center);
		\draw (3.center) to (5.center);
		\draw (4.center) to (6.center);
		\draw [in=90, out=-90] (13.center) to (15.center);
		\draw (10.center) to (16.center);
		\draw [in=-90, out=90, looseness=1.25] (1.center) to (11.center);
		\draw [in=90, out=-90] (20.center) to (25.center);
		\draw [in=90, out=-90] (21.center) to (23.center);
		\draw [in=90, out=-90] (22.center) to (24.center);
		\draw [in=90, out=-90] (29.center) to (30.center);
		\draw (27.center) to (31.center);
		\draw [in=-90, out=90, looseness=1.25] (19.center) to (28.center);
		\draw (34.center) to (39.center);
		\draw (35.center) to (37.center);
		\draw [in=90, out=-90] (36.center) to (38.center);
		\draw [in=-90, out=90, looseness=1.25] (33.center) to (40);
		\draw (41.center) to (40);
		\draw (46) to (48.center);
		\draw (47) to (49.center);
		\draw [in=-90, out=90, looseness=1.25] (44.center) to (51);
		\draw (52.center) to (51);
	\end{pgfonlayer}
\end{tikzpicture}\]
\end{defn}
Only the former condition is actually required; the latter is entailed by the first. We prove this in both algebraic and graphical notation. First graphically:
\[\begin{tikzpicture}[scale=0.7, {every node/.style}={scale=0.7}]
	\begin{pgfonlayer}{nodelayer}
		\node [style=WIDEmorph] (32) at (-0.25, 0) {$T$};
		\node [style=none] (33) at (-0.75, 0.25) {};
		\node [style=none] (34) at (-0.75, -0.25) {};
		\node [style=none] (35) at (0, 0) {};
		\node [style=none] (36) at (0.25, 0) {};
		\node [style=none] (37) at (0, -1.25) {};
		\node [style=none] (38) at (0.5, -1.25) {};
		\node [style=none] (39) at (-0.75, -1.25) {};
		\node [style=morph] (40) at (-0.75, 1) {$g$};
		\node [style=none] (41) at (-0.75, 1.75) {};
		\node [style=none] (42) at (1, 0) {$=$};
		\node [style=none] (44) at (17, -0.75) {};
		\node [style={white_dot}] (46) at (17.75, 0.25) {};
		\node [style={white_dot}] (47) at (18.25, 0.25) {};
		\node [style=none] (48) at (17.75, -0.75) {};
		\node [style=none] (49) at (18.25, -0.75) {};
		\node [style=morph] (51) at (17, 0) {$g$};
		\node [style=none] (52) at (17, 0.75) {};
		\node [style=WIDEmorph] (53) at (2.75, 1.25) {$T$};
		\node [style=none] (54) at (2.25, 1.5) {};
		\node [style=none] (55) at (2.25, 1) {};
		\node [style=none] (56) at (3, 1.25) {};
		\node [style=none] (57) at (3.25, 1.25) {};
		\node [style=none] (58) at (3.75, -3.25) {};
		\node [style=none] (59) at (4.25, -3.25) {};
		\node [style=none] (60) at (2, -0.5) {};
		\node [style=morph] (61) at (2.25, 2.25) {$g$};
		\node [style=none] (62) at (2.25, 3) {};
		\node [style=none] (63) at (4.75, 0) {$=$};
		\node [style=none] (64) at (2, -1) {};
		\node [style=WIDEmorph] (65) at (2.5, -0.75) {$\putt{}$};
		\node [style=none] (66) at (3, -0.75) {};
		\node [style=morph] (67) at (3, -1.75) {$g$};
		\node [style={white_dot}] (68) at (2.5, -2.5) {};
		\node [style=none] (69) at (2.5, -3.25) {};
		\node [style=WIDEmorph] (70) at (6.25, -0.75) {$T$};
		\node [style=none] (71) at (6, 1.5) {};
		\node [style=none] (72) at (6, 1) {};
		\node [style=none] (73) at (6.5, -0.75) {};
		\node [style=none] (74) at (6.75, -0.75) {};
		\node [style=none] (75) at (7.5, -3.25) {};
		\node [style=none] (76) at (8, -3.25) {};
		\node [style=none] (77) at (5.75, -0.5) {};
		\node [style=morph] (78) at (6, 2.25) {$g$};
		\node [style=none] (79) at (6, 3) {};
		\node [style=none] (80) at (5.75, -1) {};
		\node [style=WIDEmorph] (81) at (6.5, 1.25) {$\putt{}$};
		\node [style=none] (82) at (7, 1.25) {};
		\node [style=morph] (83) at (7.5, -0.75) {$g$};
		\node [style={white_dot}] (84) at (6.25, -2.5) {};
		\node [style=none] (85) at (6.25, -3.25) {};
		\node [style=none] (86) at (8.5, 0) {$=$};
		\node [style=WIDEmorph] (87) at (10, -0.75) {$T$};
		\node [style=none] (90) at (10.25, -0.75) {};
		\node [style=none] (91) at (10.5, -0.75) {};
		\node [style=none] (92) at (11.25, -3.25) {};
		\node [style=none] (93) at (11.75, -3.25) {};
		\node [style=none] (94) at (9.5, -0.5) {};
		\node [style=none] (97) at (9.5, -1) {};
		\node [style=none] (99) at (10, 2.75) {};
		\node [style=morph] (100) at (11.25, -0.75) {$g$};
		\node [style={white_dot}] (101) at (10, -2.5) {};
		\node [style=none] (102) at (10, -3.25) {};
		\node [style={white_dot}] (103) at (9.5, 1) {};
		\node [style=none] (104) at (12.25, 0) {$=$};
		\node [style={white_dot}] (106) at (13.75, -0.25) {};
		\node [style={white_dot}] (107) at (14.25, -0.25) {};
		\node [style=none] (108) at (15, -3.25) {};
		\node [style=none] (109) at (15.5, -3.25) {};
		\node [style={white_dot}] (111) at (13.25, -0.75) {};
		\node [style=none] (112) at (13.75, 2.75) {};
		\node [style=morph] (113) at (15, -0.75) {$g$};
		\node [style={white_dot}] (114) at (13.75, -2.5) {};
		\node [style=none] (115) at (13.75, -3.25) {};
		\node [style=none] (116) at (16, 0) {$=$};
	\end{pgfonlayer}
	\begin{pgfonlayer}{edgelayer}
		\draw (34.center) to (39.center);
		\draw (35.center) to (37.center);
		\draw [in=90, out=-90] (36.center) to (38.center);
		\draw [in=-90, out=90, looseness=1.25] (33.center) to (40);
		\draw (41.center) to (40);
		\draw (46) to (48.center);
		\draw (47) to (49.center);
		\draw [in=-90, out=90, looseness=1.25] (44.center) to (51);
		\draw (52.center) to (51);
		\draw (55.center) to (60.center);
		\draw [in=90, out=-90, looseness=0.75] (56.center) to (58.center);
		\draw [in=90, out=-90] (57.center) to (59.center);
		\draw [in=-90, out=90, looseness=1.25] (54.center) to (61);
		\draw (62.center) to (61);
		\draw (67) to (66.center);
		\draw [bend left=45] (67) to (68);
		\draw [in=-90, out=135] (68) to (64.center);
		\draw (68) to (69.center);
		\draw (72.center) to (77.center);
		\draw [in=90, out=-90] (73.center) to (75.center);
		\draw [in=90, out=-90] (74.center) to (76.center);
		\draw [in=-90, out=90, looseness=1.25] (71.center) to (78);
		\draw (79.center) to (78);
		\draw [in=-90, out=90] (83) to (82.center);
		\draw [bend left=45] (83) to (84);
		\draw [in=-90, out=135] (84) to (80.center);
		\draw (84) to (85.center);
		\draw [in=90, out=-90] (90.center) to (92.center);
		\draw [in=90, out=-90] (91.center) to (93.center);
		\draw [in=-90, out=90] (100) to (99.center);
		\draw [bend left=45] (100) to (101);
		\draw [in=-90, out=135] (101) to (97.center);
		\draw (101) to (102.center);
		\draw (103) to (94.center);
		\draw [in=90, out=-90] (106) to (108.center);
		\draw [in=90, out=-90] (107) to (109.center);
		\draw [in=-90, out=90] (113) to (112.center);
		\draw [bend left=45] (113) to (114);
		\draw [in=-90, out=135] (114) to (111);
		\draw (114) to (115.center);
	\end{pgfonlayer}
\end{tikzpicture}\]
then algebraically:
\begin{align*}
    g(T(s,v_1,v_2)) & = g(T(\putt{}(s,g(s)),v_1,v_2)) \\
    & = g(\putt{}(T(s,v_1,v_2),g(s))) \\
    & = g(s).
\end{align*}
Every tagged lens $\phi := (\putt{},g,T)$ induces a lens $L(\phi)$ which takes two copies of the view, uses one of them to update via the vwb lens, and uses the other to trigger the tag corresponding to the transition between the old and new view values. This updating procedure is captured precisely by the following definition.
\begin{defn}[Induced Lens]
For every tagged lens $(\putt{},g,T)$ the induced lens \[(P,g):=L((\putt{},g,T))\] is defined by \[\begin{tikzpicture}[scale=0.7, {every node/.style}={scale=0.7}]
	\begin{pgfonlayer}{nodelayer}
		\node [style=none] (1) at (-2.75, -0.75) {};
		\node [style=WIDEmorph] (9) at (-2.25, 0.25) {$P$};
		\node [style=none] (10) at (-2.75, 0.5) {};
		\node [style=none] (11) at (-2.75, 0) {};
		\node [style=none] (13) at (-1.75, 0.25) {};
		\node [style=none] (15) at (-1.75, -0.75) {};
		\node [style=none] (16) at (-2.75, 1.25) {};
		\node [style=none] (17) at (-0.75, 0) {$=$};
		\node [style=WIDEmorph] (18) at (0.75, 0.75) {$T$};
		\node [style=none] (21) at (1, 0.75) {};
		\node [style=none] (22) at (1.25, 0.75) {};
		\node [style=morph] (23) at (1, -0.5) {$g$};
		\node [style=WIDEmorph] (26) at (1.5, 2) {$\putt{}$};
		\node [style=none] (27) at (0.25, 1) {};
		\node [style=none] (28) at (0.25, 0.5) {};
		\node [style={white_dot}] (30) at (2.25, -0.75) {};
		\node [style=none] (31) at (1, 1.75) {};
		\node [style=none] (53) at (2, 2) {};
		\node [style=none] (54) at (2.25, -2.25) {};
		\node [style={white_dot}] (55) at (0.5, -1.25) {};
		\node [style=none] (56) at (0.5, -2.25) {};
		\node [style=none] (57) at (1, 3) {};
		\node [style=none] (58) at (1, 2.25) {};
	\end{pgfonlayer}
	\begin{pgfonlayer}{edgelayer}
		\draw [in=90, out=-90] (13.center) to (15.center);
		\draw (10.center) to (16.center);
		\draw [in=-90, out=90, looseness=1.25] (1.center) to (11.center);
		\draw [in=90, out=-90] (21.center) to (23);
		\draw [in=270, out=90] (27.center) to (31.center);
		\draw [in=30, out=-90, looseness=0.50] (53.center) to (30);
		\draw [in=150, out=-75, looseness=0.75] (22.center) to (30);
		\draw [in=120, out=-90, looseness=0.75] (28.center) to (55);
		\draw (55) to (56.center);
		\draw [in=-90, out=0] (55) to (23);
		\draw (57.center) to (58.center);
		\draw (30) to (54.center);
	\end{pgfonlayer}
\end{tikzpicture}\]
\end{defn}
In algebraic notation this reads \[p(s,v) := \putt{}(T(s,g(s),v),v),\] so the induced put first extracts the old view $g(s)$, feeds it into the past input of $T$, copies the new view value $v$ into the future input of $T$, and then applies the underlying vwb update. The running examples above admit tagged presentations.
\begin{example}[Tagged PutCountChanges]
The lens $(\texttt{PutCountChanges},g)$ is induced by the tagged vwb lens with
\[
\putt{}((i,v),v')=(i,v'),
\qquad
\texttt{CountChanges}((i,v),v',v'')=
\begin{cases}
(i+1,v) & v' \neq v'',\\
(i,v) & v' = v''.
\end{cases}
\]
\end{example}
\begin{example}[Tagged PutFlag]
The lens $(\texttt{PutFlag},g)$ is induced by
\[
\putt{}((b,v),v')=(b,v'),
\qquad
\texttt{Flag}((b,v),v',v'')=(1,v).
\]
\end{example}
\begin{example}[Tagged PutScaled]
The lens $(\texttt{PutScaled},g)$ is induced by
\[
\putt{}((x,y),v)=(v,y),
\qquad
\texttt{Scale}((x,y),v,v')=(x,v'y/v).
\]
\end{example}
Note that these induced lenses all satisfy \textbf{putget}, as a direct consequence of the tagged lens axioms.
\begin{prop}
Every lens induced by a tagged lens satisfies $PutGet$
\end{prop}
\begin{proof}
\[\begin{tikzpicture}[scale=0.7, {every node/.style}={scale=0.7}]
	\begin{pgfonlayer}{nodelayer}
		\node [style=none] (1) at (-8.25, -0.75) {};
		\node [style=WIDEmorph] (9) at (-7.75, 0.25) {$P$};
		\node [style=none] (10) at (-8.25, 0.5) {};
		\node [style=none] (11) at (-8.25, 0) {};
		\node [style=none] (13) at (-7.25, 0.25) {};
		\node [style=none] (15) at (-7.25, -0.75) {};
		\node [style=morph] (16) at (-8.25, 1.25) {$g$};
		\node [style=none] (17) at (-6.25, 0) {$=$};
		\node [style=WIDEmorph] (18) at (-4.75, 0.25) {$T$};
		\node [style=none] (21) at (-4.5, 0.25) {};
		\node [style=none] (22) at (-4.25, 0.25) {};
		\node [style=morph] (23) at (-4.5, -1) {$g$};
		\node [style=WIDEmorph] (26) at (-4, 1.5) {$\putt{}$};
		\node [style=none] (27) at (-5.25, 0.5) {};
		\node [style=none] (28) at (-5.25, 0) {};
		\node [style={white_dot}] (30) at (-3.25, -1.25) {};
		\node [style=none] (31) at (-4.5, 1.25) {};
		\node [style=none] (53) at (-3.5, 1.5) {};
		\node [style=none] (54) at (-3.25, -2.25) {};
		\node [style={white_dot}] (55) at (-5, -1.75) {};
		\node [style=none] (56) at (-5, -2.5) {};
		\node [style=none] (59) at (-8.25, 2) {};
		\node [style=none] (60) at (-4.5, 1.75) {};
		\node [style=morph] (61) at (-4.5, 2.5) {$g$};
		\node [style=none] (62) at (-4.5, 3.25) {};
		\node [style=none] (79) at (-2.25, 0) {$=$};
		\node [style=WIDEmorph] (80) at (-0.75, 0.75) {$T$};
		\node [style=none] (81) at (-0.5, 0.75) {};
		\node [style=none] (82) at (-0.25, 0.75) {};
		\node [style=morph] (83) at (-0.5, -0.5) {$g$};
		\node [style=none] (86) at (-1.25, 0.5) {};
		\node [style={white_dot}] (87) at (0.75, -0.75) {};
		\node [style=none] (89) at (0.5, 0.75) {};
		\node [style=none] (90) at (0.75, -1.75) {};
		\node [style={white_dot}] (91) at (-1, -1.25) {};
		\node [style=none] (92) at (-1, -2) {};
		\node [style=none] (93) at (-1.25, 1) {};
		\node [style={white_dot}] (94) at (-1.25, 2) {};
		\node [style=none] (95) at (1.5, 0) {$=$};
		\node [style={white_dot}] (97) at (3.25, 1) {};
		\node [style={white_dot}] (98) at (4, 1) {};
		\node [style=morph] (99) at (3.25, 0) {$g$};
		\node [style={white_dot}] (100) at (2.5, 1) {};
		\node [style={white_dot}] (101) at (4.5, -0.25) {};
		\node [style=none] (102) at (5, 1) {};
		\node [style=none] (103) at (4.5, -1.25) {};
		\node [style={white_dot}] (104) at (2.75, -0.75) {};
		\node [style=none] (105) at (2.75, -1.5) {};
		\node [style=none] (106) at (6, 0) {$=$};
		\node [style=none] (112) at (7.75, 0.5) {};
		\node [style=none] (113) at (7.75, -0.75) {};
		\node [style={white_dot}] (114) at (7, 0.5) {};
		\node [style=none] (115) at (7, -0.75) {};
	\end{pgfonlayer}
	\begin{pgfonlayer}{edgelayer}
		\draw [in=90, out=-90] (13.center) to (15.center);
		\draw (10.center) to (16);
		\draw [in=-90, out=90, looseness=1.25] (1.center) to (11.center);
		\draw [in=90, out=-90] (21.center) to (23);
		\draw [in=270, out=90] (27.center) to (31.center);
		\draw [in=30, out=-90, looseness=0.50] (53.center) to (30);
		\draw [in=150, out=-90] (22.center) to (30);
		\draw [in=135, out=-90, looseness=0.75] (28.center) to (55);
		\draw (55) to (56.center);
		\draw [in=-90, out=0] (55) to (23);
		\draw (30) to (54.center);
		\draw (59.center) to (16);
		\draw (60.center) to (61);
		\draw (62.center) to (61);
		\draw [in=90, out=-90] (81.center) to (83);
		\draw [in=90, out=-90, looseness=1.50] (89.center) to (87);
		\draw [in=165, out=-90, looseness=0.75] (82.center) to (87);
		\draw [in=135, out=-90, looseness=0.75] (86.center) to (91);
		\draw (91) to (92.center);
		\draw [in=-90, out=0] (91) to (83);
		\draw (87) to (90.center);
		\draw [in=270, out=90] (93.center) to (94);
		\draw [in=90, out=-90] (97) to (99);
		\draw [in=0, out=-90] (102.center) to (101);
		\draw [in=165, out=-90, looseness=0.75] (98) to (101);
		\draw [in=135, out=-90, looseness=0.75] (100) to (104);
		\draw (104) to (105.center);
		\draw [in=-90, out=0] (104) to (99);
		\draw (101) to (103.center);
		\draw (114) to (115.center);
		\draw (113.center) to (112.center);
	\end{pgfonlayer}
\end{tikzpicture}\]
\end{proof}
The remaining lens laws will be consequences of additional design laws for tags.
%\begin{prop}[TGetPut]
%The lens induced by a tagged lens satisfies an equation on GetPut involving the tag, Expressing that retrieving and reinserting applies only a tag.
%\end{prop}
%\begin{proof}
%\[\tikzfig{tagfigs/taggp}\]
%\end{proof}
%In general one cannot expect any interesting Putput for for a lens induced by a tagged lens, however there is a subcategory of induced lenses which do satisfy PutPut laws, those induced by ``future-depending tags" which we introduce in section 2.3. 

\section{The Symmetric Monoidal Category of Tagged Lenses}
Our second step towards demonstrating a compositional entailment between tag laws and lens laws is to establish the relevant notion of compositionality: tagged lenses can be composed, and the induced lens construction preserves this composition. There are in principle many ways to compose tags; we choose the sequential composition for which
\[
L((\putt{}_2,g_2,T_2) \circ (\putt{}_1,g_1,T_1)) = L((\putt{}_2,g_2,T_2)) \circ L((\putt{}_1,g_1,T_1)).
\]
Categorically, this says that $L$ is functorial.
\begin{prop}[The Category $\cat{TL}(\mathcal{C})$]
For every cartesian monoidal category $\mathcal{C}$ a symmetric monoidal category of tagged lenses $\cat{TL}(\mathcal{C})$ can be constructed with
\begin{itemize}
    \item Objects given by the objects of $\mathcal{C}$
    \item Morphisms given by \[\mathbf{TL}[\mathcal{C}](S,S') := \{ \textrm{ Tagged lenses from $S$ to $S'$ }\}\]
\end{itemize}
along with the sequential composition of morphisms 
\begin{align}
    (\putt{}_1,T_1,g_1) & : V \rightarrow V' \\
    (\putt{}_2,T_2,g_2) & : V' \rightarrow V'' \\
(\putt{},T,g) = (\putt{}_2,T_2,g_2) \circ (\putt{}_1,T_1,g_1) & : V \rightarrow V'' 
\end{align}
defined by: \[\begin{tikzpicture}[scale=0.7, {every node/.style}={scale=0.7}]
	\begin{pgfonlayer}{nodelayer}
		\node [style=none] (1) at (-2.25, -0.5) {};
		\node [style=WIDEmorph] (9) at (-1.75, 0.5) {$T$};
		\node [style=none] (10) at (-2.25, 0.75) {};
		\node [style=none] (11) at (-2.25, 0.25) {};
		\node [style=none] (13) at (-1.25, 0.5) {};
		\node [style=none] (15) at (-1.25, -0.5) {};
		\node [style=none] (16) at (-2.25, 1.5) {};
		\node [style=none] (17) at (0, 0.25) {$=$};
		\node [style=WIDEmorph] (18) at (4, 4.5) {$T_1$};
		\node [style=none] (21) at (3.5, 4.25) {};
		\node [style=morph] (23) at (3.5, 3.5) {$g_2$};
		\node [style=WIDEmorph] (26) at (3, 5.75) {$\putt{}_2$};
		\node [style=none] (28) at (2.5, 5.5) {};
		\node [style=none] (53) at (3.5, 5.75) {};
		\node [style={white_dot}] (55) at (2.75, 2.75) {};
		\node [style=none] (56) at (1.25, 1.25) {};
		\node [style=none] (60) at (2.5, 6) {};
		\node [style=none] (61) at (2.5, 6.75) {};
		\node [style=none] (62) at (3.5, 4.75) {};
		\node [style=none] (63) at (4.25, 4.5) {};
		\node [style=none] (72) at (-1.5, 0.5) {};
		\node [style=none] (73) at (-1.5, -0.5) {};
		\node [style=none] (75) at (4.5, 4.5) {};
		\node [style=WIDEmorph] (76) at (1.75, 1.25) {$T_2$};
		\node [style=none] (77) at (1.25, 1) {};
		\node [style=none] (82) at (2, 1.25) {};
		\node [style=none] (83) at (3.25, 0.25) {};
		\node [style=none] (84) at (2.25, 1.25) {};
		\node [style=WIDEmorph] (85) at (3.75, 0.25) {$\putt{}_1$};
		\node [style=none] (86) at (3.25, 0) {};
		\node [style=none] (87) at (4.25, 0.25) {};
		\node [style=none] (89) at (3, -1) {};
		\node [style=WIDEmorph] (90) at (3.5, -1) {$T_1$};
		\node [style=none] (91) at (3, -1.25) {};
		\node [style=none] (94) at (3.75, -1) {};
		\node [style=none] (95) at (4, -1) {};
		\node [style=morph] (96) at (2, -2.25) {$g_2$};
		\node [style=morph] (97) at (2.75, -5) {$g_2$};
		\node [style=none] (98) at (2.75, -3.75) {};
		\node [style=WIDEmorph] (99) at (3.25, -3.75) {$\putt{}_1$};
		\node [style=none] (100) at (2.75, -4) {};
		\node [style=none] (101) at (3.75, -3.75) {};
		\node [style={white_dot}] (102) at (1.5, -4.25) {};
		\node [style={white_dot}] (103) at (2, -5.75) {};
		\node [style={white_dot}] (104) at (4.5, -5) {};
		\node [style={white_dot}] (105) at (6, -5) {};
		\node [style=none] (106) at (2, -6.75) {};
		\node [style=none] (107) at (4.5, -6.5) {};
		\node [style=none] (108) at (6, -6.5) {};
		\node [style=morph] (130) at (-8.25, -4) {$g$};
		\node [style=none] (131) at (-8.25, -3.25) {};
		\node [style=none] (132) at (-8.25, -4.75) {};
		\node [style=none] (133) at (-7.25, -4) {$=$};
		\node [style=morph] (134) at (-6.25, -3.5) {$g_1$};
		\node [style=none] (135) at (-6.25, -2.75) {};
		\node [style=morph] (136) at (-6.25, -4.5) {$g_2$};
		\node [style=none] (137) at (-6.25, -5.5) {};
		\node [style=none] (138) at (-9.75, 2) {};
		\node [style=WIDEmorph] (139) at (-9.25, 3) {$\putt{}$};
		\node [style=none] (140) at (-9.75, 3.25) {};
		\node [style=none] (141) at (-9.75, 2.75) {};
		\node [style=none] (142) at (-8.75, 3) {};
		\node [style=none] (143) at (-8.75, 2) {};
		\node [style=none] (144) at (-9.75, 4) {};
		\node [style=none] (145) at (-7.75, 2.75) {$=$};
		\node [style=WIDEmorph] (146) at (-5.25, 3) {$\putt{}_1$};
		\node [style=none] (147) at (-5.75, 2.75) {};
		\node [style=morph] (148) at (-5.75, 2) {$g_2$};
		\node [style=WIDEmorph] (149) at (-6.25, 4.25) {$\putt{}_2$};
		\node [style=none] (150) at (-6.75, 4) {};
		\node [style=none] (151) at (-5.75, 4.25) {};
		\node [style=none] (152) at (-4.75, 1.25) {};
		\node [style={white_dot}] (153) at (-6.5, 1.25) {};
		\node [style=none] (154) at (-6.5, 0.5) {};
		\node [style=none] (155) at (-6.75, 4.5) {};
		\node [style=none] (156) at (-6.75, 5.25) {};
		\node [style=none] (157) at (-5.75, 3.25) {};
		\node [style=none] (158) at (-4.75, 3) {};
		\node [style=none] (159) at (-3.25, 7.25) {};
		\node [style=none] (160) at (-11.25, 7.25) {};
		\node [style=none] (161) at (-11.25, 1.5) {};
		\node [style=none] (162) at (-3.25, 1.5) {};
		\node [style=none] (163) at (-3.25, -0.75) {};
		\node [style=none] (164) at (-11.25, -0.75) {};
		\node [style=none] (165) at (-11.25, -7.25) {};
		\node [style=none] (166) at (-3.25, -7.25) {};
		\node [style=none] (167) at (7.5, -7.25) {};
		\node [style=none] (168) at (7.25, 7.25) {};
	\end{pgfonlayer}
	\begin{pgfonlayer}{edgelayer}
		\draw [in=90, out=-90] (13.center) to (15.center);
		\draw (10.center) to (16.center);
		\draw [in=-90, out=90, looseness=1.25] (1.center) to (11.center);
		\draw [in=90, out=-90] (21.center) to (23);
		\draw [in=135, out=-90, looseness=0.75] (28.center) to (55);
		\draw [in=90, out=-90] (55) to (56.center);
		\draw [in=-90, out=0] (55) to (23);
		\draw (60.center) to (61.center);
		\draw [in=90, out=-90, looseness=1.25] (53.center) to (62.center);
		\draw [in=90, out=-90] (72.center) to (73.center);
		\draw [in=90, out=-90] (84.center) to (83.center);
		\draw [in=90, out=-90, looseness=1.25] (86.center) to (89.center);
		\draw [style=new edge style 2, in=30, out=-90, looseness=1.25] (63.center) to (104);
		\draw [style=new edge style 2, in=105, out=-75] (94.center) to (104);
		\draw [style=new edge style 2, in=135, out=-60] (101.center) to (104);
		\draw [style=new edge style 2] (104) to (107.center);
		\draw [style=new edge style 5] (105) to (108.center);
		\draw [in=135, out=-90, looseness=0.75] (77.center) to (102);
		\draw [in=-90, out=15] (102) to (96);
		\draw [in=90, out=-90, looseness=1.25] (91.center) to (96);
		\draw [in=90, out=-90] (82.center) to (98.center);
		\draw [in=150, out=-90] (102) to (103);
		\draw [in=-90, out=15, looseness=1.25] (103) to (97);
		\draw [in=270, out=90] (97) to (100.center);
		\draw (103) to (106.center);
		\draw [style=new edge style 5, in=135, out=-90, looseness=0.75] (95.center) to (105);
		\draw [style=new edge style 5, in=285, out=60] (105) to (75.center);
		\draw [style=new edge style 5, in=90, out=-90] (87.center) to (105);
		\draw (131.center) to (130);
		\draw (132.center) to (130);
		\draw (135.center) to (134);
		\draw (134) to (136);
		\draw (137.center) to (136);
		\draw [in=90, out=-90] (142.center) to (143.center);
		\draw (140.center) to (144.center);
		\draw [in=-90, out=90, looseness=1.25] (138.center) to (141.center);
		\draw [in=90, out=-90] (147.center) to (148);
		\draw [in=135, out=-90, looseness=0.75] (150.center) to (153);
		\draw (153) to (154.center);
		\draw [in=-90, out=0] (153) to (148);
		\draw (155.center) to (156.center);
		\draw [in=90, out=-90, looseness=1.25] (151.center) to (157.center);
		\draw [in=90, out=-90] (158.center) to (152.center);
		\draw [style=double edge] (168.center) to (167.center);
		\draw [style=double edge] (167.center) to (166.center);
		\draw [style=double edge] (166.center) to (163.center);
		\draw [style=double edge] (163.center) to (162.center);
		\draw [style=double edge] (162.center) to (159.center);
		\draw [style=double edge] (159.center) to (168.center);
		\draw [style=double edge] (159.center) to (160.center);
		\draw [style=double edge] (160.center) to (161.center);
		\draw [style=double edge] (164.center) to (163.center);
		\draw [style=double edge] (161.center) to (164.center);
		\draw [style=double edge] (164.center) to (165.center);
		\draw [style=double edge] (165.center) to (166.center);
	\end{pgfonlayer}
\end{tikzpicture}\]
and the parallel composition of morphisms
\begin{align}
    (\putt{}_1,T_1,g_1) & : V \rightarrow V' \\
    (\putt{}_2,T_2,g_2) & : W \rightarrow W' \\
(\putt{},T,g) = (\putt{}_2,T_2,g_2) \otimes (\putt{}_1,T_1,g_1) & : V \times W \rightarrow V' \times W' 
\end{align}
defined by: \[\begin{tikzpicture}[scale=0.7, {every node/.style}={scale=0.7}]
	\begin{pgfonlayer}{nodelayer}
		\node [style=none] (1) at (0.75, 2.75) {};
		\node [style=WIDEmorph] (9) at (1.25, 3.75) {$T$};
		\node [style=none] (10) at (0.75, 4) {};
		\node [style=none] (11) at (0.75, 3.5) {};
		\node [style=none] (13) at (1.75, 3.75) {};
		\node [style=none] (15) at (1.75, 2.75) {};
		\node [style=none] (16) at (0.75, 4.75) {};
		\node [style=none] (72) at (1.5, 3.75) {};
		\node [style=none] (73) at (1.5, 2.75) {};
		\node [style=morph] (130) at (-4, 3.75) {$g$};
		\node [style=none] (131) at (-4, 4.5) {};
		\node [style=none] (132) at (-4, 3) {};
		\node [style=none] (160) at (-12, 5.75) {};
		\node [style=none] (161) at (-12, 1.75) {};
		\node [style=none] (184) at (-2, 4.5) {};
		\node [style=morph] (185) at (-2, 3.75) {$g_1$};
		\node [style=none] (186) at (-2, 3) {};
		\node [style=none] (187) at (-11.25, 2.75) {};
		\node [style=WIDEmorph] (188) at (-10.75, 3.75) {$\putt{}$};
		\node [style=none] (189) at (-11.25, 4) {};
		\node [style=none] (190) at (-11.25, 3.5) {};
		\node [style=none] (191) at (-10.25, 3.75) {};
		\node [style=none] (192) at (-10.25, 2.75) {};
		\node [style=none] (193) at (-11.25, 4.75) {};
		\node [style=none] (194) at (-9.25, 3.5) {$=$};
		\node [style=none] (195) at (0, 5.75) {};
		\node [style=none] (196) at (-5, 5.75) {};
		\node [style=none] (197) at (-5, 1.75) {};
		\node [style=none] (198) at (0, 1.75) {};
		\node [style=none] (199) at (-8.5, 2.5) {};
		\node [style=WIDEmorph] (200) at (-8, 4) {$\putt{}_1$};
		\node [style=none] (201) at (-8.5, 4.25) {};
		\node [style=none] (202) at (-8.5, 3.75) {};
		\node [style=none] (203) at (-7.5, 4) {};
		\node [style=none] (204) at (-6.25, 2.5) {};
		\node [style=none] (205) at (-8.5, 5) {};
		\node [style=none] (206) at (-7.5, 2.5) {};
		\node [style=WIDEmorph] (207) at (-6, 4) {$\putt{}_2$};
		\node [style=none] (208) at (-6.5, 4.25) {};
		\node [style=none] (209) at (-6.5, 3.75) {};
		\node [style=none] (210) at (-5.5, 4) {};
		\node [style=none] (211) at (-5.5, 2.5) {};
		\node [style=none] (212) at (-6.5, 5) {};
		\node [style=none] (213) at (3.75, 2.25) {};
		\node [style=WIDEmorph] (214) at (4.25, 4.25) {$T_1$};
		\node [style=none] (215) at (3.75, 4.5) {};
		\node [style=none] (216) at (3.75, 4) {};
		\node [style=none] (217) at (4.75, 4.25) {};
		\node [style=none] (218) at (6.75, 2.25) {};
		\node [style=none] (219) at (3.75, 5.25) {};
		\node [style=none] (220) at (4.5, 4.25) {};
		\node [style=none] (221) at (5.5, 2.25) {};
		\node [style=none] (222) at (4.25, 2.25) {};
		\node [style=WIDEmorph] (223) at (6.25, 4.25) {$T_2$};
		\node [style=none] (224) at (5.75, 4.5) {};
		\node [style=none] (225) at (5.75, 4) {};
		\node [style=none] (226) at (6.75, 4.25) {};
		\node [style=none] (227) at (7, 2.25) {};
		\node [style=none] (228) at (5.75, 5.25) {};
		\node [style=none] (229) at (6.5, 4.25) {};
		\node [style=none] (230) at (5.75, 2.25) {};
		\node [style=none] (231) at (2.75, 3.5) {$=$};
		\node [style=none] (232) at (-3, 3.75) {$=$};
		\node [style=morph] (233) at (-1, 3.75) {$g_2$};
		\node [style=none] (234) at (-1, 4.5) {};
		\node [style=none] (236) at (-1, 3) {};
		\node [style=none] (274) at (7.5, 5.75) {};
		\node [style=none] (276) at (7.5, 1.75) {};
	\end{pgfonlayer}
	\begin{pgfonlayer}{edgelayer}
		\draw [in=90, out=-90] (13.center) to (15.center);
		\draw (10.center) to (16.center);
		\draw [in=-90, out=90, looseness=1.25] (1.center) to (11.center);
		\draw [in=90, out=-90] (72.center) to (73.center);
		\draw (131.center) to (130);
		\draw (132.center) to (130);
		\draw [style=double edge] (160.center) to (161.center);
		\draw (186.center) to (185);
		\draw [in=90, out=-90] (191.center) to (192.center);
		\draw (189.center) to (193.center);
		\draw [in=-90, out=90, looseness=1.25] (187.center) to (190.center);
		\draw [style=double edge] (198.center) to (195.center);
		\draw [style=double edge] (195.center) to (196.center);
		\draw [style=double edge] (196.center) to (197.center);
		\draw [style=double edge] (197.center) to (198.center);
		\draw [in=90, out=-90] (203.center) to (204.center);
		\draw (201.center) to (205.center);
		\draw [in=-90, out=90, looseness=1.25] (199.center) to (202.center);
		\draw [in=90, out=-90] (210.center) to (211.center);
		\draw (208.center) to (212.center);
		\draw [in=-90, out=90, looseness=1.25] (206.center) to (209.center);
		\draw [in=90, out=-90] (217.center) to (218.center);
		\draw (215.center) to (219.center);
		\draw [in=-90, out=90, looseness=1.25] (213.center) to (216.center);
		\draw [in=90, out=-90] (220.center) to (221.center);
		\draw [in=90, out=-90] (226.center) to (227.center);
		\draw (224.center) to (228.center);
		\draw [in=-90, out=90, looseness=1.25] (222.center) to (225.center);
		\draw [in=90, out=-90] (229.center) to (230.center);
		\draw (234.center) to (233);
		\draw (184.center) to (185);
		\draw (233) to (236.center);
		\draw [style=double edge] (160.center) to (196.center);
		\draw [style=double edge] (197.center) to (161.center);
		\draw [style=double edge] (274.center) to (276.center);
		\draw [style=double edge] (274.center) to (195.center);
		\draw [style=double edge] (276.center) to (198.center);
	\end{pgfonlayer}
\end{tikzpicture}\]
\end{prop}
\begin{proof}
Given in the appendix
\end{proof}
At this point the graphical calculus becomes more readable, and more manageable, than the corresponding algebraic expressions for composed tags. We give an intuitive description of tag composition in the appendix. The compatibility above can be summarised categorically as follows.
\begin{prop}
The induced lens construction defines a strict monoidal injective-on-objects functor \[\mathcal{L}:\cat{TL}[\mathcal{C}] \longrightarrow \cat{pgL}[\mathcal{C}]\] from the category $\cat{TL}[\mathcal{C}]$ of tagged lenses over $\mathcal{C}$ into the category $\cat{pgL}[\mathcal{C}]$ of putget lenses over $\mathcal{C}$.
\end{prop}
\begin{proof}
The functor is defined on objects by $\mathcal{L}(V) := V$ and on morphisms by producing the lens induced by the tagged lens $\mathcal{L}((\putt{},g,T)) := L((\putt{},g,T))$. The proof that this is indeed a strict monoidal functor is given in the appendix. 
\end{proof}
Since $\mathcal{L}$ is injective on objects, its image is a subcategory of $\cat{pgL}[\mathcal{C}]$, namely the subcategory of putget lenses induced by tagged lenses.

\section{Subclasses of Tagged Lenses}
Our next step towards compositional entailment of lens laws by tag design laws is to define two intuitive compositional design laws for tags, each preserved under sequential and parallel composition:
\begin{itemize}
    \item \textbf{Change-Depending}: A tag which only records a genuine change in the value of the view, trivial updates in which $p=f$ are not recorded by $T$: \[\begin{tikzpicture}[scale=0.7, {every node/.style}={scale=0.7}]
	\begin{pgfonlayer}{nodelayer}
		\node [style=WIDEmorph] (0) at (-1.75, 0.5) {$T$};
		\node [style=none] (1) at (-2.25, 0.75) {};
		\node [style=none] (2) at (-2.25, 0.25) {};
		\node [style=none] (3) at (-1.5, 0.5) {};
		\node [style=none] (4) at (-1.25, 0.5) {};
		\node [style=none] (8) at (-2.25, -1.5) {};
		\node [style=none] (11) at (-2.25, 1.75) {};
		\node [style=none] (17) at (-0.25, 0) {$=$};
		\node [style={white_dot}] (46) at (1.75, 0) {};
		\node [style={white_dot}] (47) at (-1.25, -0.75) {};
		\node [style=none] (48) at (1.75, -0.75) {};
		\node [style=none] (49) at (-1.25, -1.5) {};
		\node [style=none] (50) at (0.75, 0.75) {};
		\node [style=none] (51) at (0.75, -0.75) {};
	\end{pgfonlayer}
	\begin{pgfonlayer}{edgelayer}
		\draw (2.center) to (8.center);
		\draw [in=-90, out=90, looseness=1.25] (1.center) to (11.center);
		\draw (46) to (48.center);
		\draw (47) to (49.center);
		\draw [in=150, out=-75] (3.center) to (47);
		\draw [in=30, out=-90, looseness=0.75] (4.center) to (47);
		\draw (50.center) to (51.center);
	\end{pgfonlayer}
\end{tikzpicture}\] in algebraic language \[T(s,v,v) = id\]
    \item \textbf{First-Last-Depending}: A tag which records only the initial and final values of the view after any sequence of updates, ignoring any intermediate values of the view: \[\begin{tikzpicture}[scale=0.7, {every node/.style}={scale=0.7}]
	\begin{pgfonlayer}{nodelayer}
		\node [style=WIDEmorph] (52) at (-2.5, 0.25) {$T$};
		\node [style=none] (53) at (-3, 0.5) {};
		\node [style=none] (54) at (-3, 0) {};
		\node [style=none] (55) at (-2.25, 0.25) {};
		\node [style=none] (56) at (-2, 0.25) {};
		\node [style=none] (57) at (-3, -1.75) {};
		\node [style=none] (58) at (-2, 1.25) {};
		\node [style=none] (59) at (0, 0) {$=$};
		\node [style={white_dot}] (61) at (-1.5, -1) {};
		\node [style=none] (63) at (-1.5, -1.75) {};
		\node [style=none] (67) at (-2.25, -1.75) {};
		\node [style=WIDEmorph] (68) at (-1.5, 1.5) {$T$};
		\node [style=none] (69) at (-2, 1.75) {};
		\node [style=none] (71) at (-1.25, 1.5) {};
		\node [style=none] (72) at (-1, 1.5) {};
		\node [style=none] (73) at (-2, 2.75) {};
		\node [style=none] (76) at (-0.75, -1.75) {};
		\node [style=WIDEmorph] (77) at (1.5, 1) {$T$};
		\node [style=none] (78) at (1, 1.25) {};
		\node [style=none] (79) at (1, 0.75) {};
		\node [style=none] (80) at (1.75, 1) {};
		\node [style=none] (82) at (1, -1.25) {};
		\node [style=none] (83) at (1, 2) {};
		\node [style={white_dot}] (84) at (2.5, -0.5) {};
		\node [style=none] (85) at (2.5, -1.25) {};
		\node [style=none] (86) at (1.75, -1.25) {};
		\node [style=none] (90) at (2, 1) {};
		\node [style=none] (92) at (3.25, -1.25) {};
	\end{pgfonlayer}
	\begin{pgfonlayer}{edgelayer}
		\draw (54.center) to (57.center);
		\draw [in=270, out=90, looseness=1.25] (53.center) to (58.center);
		\draw (61) to (63.center);
		\draw [in=150, out=-90] (56.center) to (61);
		\draw [in=-90, out=90, looseness=1.25] (69.center) to (73.center);
		\draw [in=270, out=45, looseness=0.75] (61) to (71.center);
		\draw [in=-90, out=90] (76.center) to (72.center);
		\draw [in=270, out=90] (67.center) to (55.center);
		\draw (79.center) to (82.center);
		\draw [in=270, out=90, looseness=1.25] (78.center) to (83.center);
		\draw [in=90, out=-90] (84) to (85.center);
		\draw [in=-90, out=90] (92.center) to (90.center);
		\draw [in=270, out=90] (86.center) to (80.center);
	\end{pgfonlayer}
\end{tikzpicture}\] in algebraic language \[T(T(s,v,v'),v',v'') = T(s,v,v'')\]
\end{itemize}
Both change-depending and first-last-depending tags are preserved by lens composition. We present the proof for change-depending tags in the main text.
\begin{prop}
Change-depending tagged lenses define a symmetric monoidal sub-category $\mathbf{cdepTL}[\mathcal{C}]$ of $\cat{TL}(\mathcal{C})$ 
\end{prop}
\begin{proof}
\[\input{tagfigs/cdeptla.tikz}\]
\[\begin{tikzpicture}[scale=0.5, {every node/.style}={scale=0.5}]
	\begin{pgfonlayer}{nodelayer}
		\node [style=none] (116) at (10.5, 1.75) {$=$};
		\node [style={white_dot}] (117) at (12.5, 1.75) {};
		\node [style=none] (119) at (12.5, 1) {};
		\node [style=none] (121) at (11.5, 2.5) {};
		\node [style=none] (122) at (11.5, 1) {};
		\node [style=none] (239) at (-10.25, 1.75) {$=$};
		\node [style=morph] (240) at (-6.75, 4.5) {$g_2$};
		\node [style=WIDEmorph] (241) at (-7.25, 6.75) {$\putt{}_2$};
		\node [style=none] (242) at (-7.75, 6.5) {};
		\node [style=none] (243) at (-6.75, 6.75) {};
		\node [style={white_dot}] (244) at (-7.5, 3.75) {};
		\node [style=none] (245) at (-9, 2.25) {};
		\node [style=none] (246) at (-7.75, 7) {};
		\node [style=none] (247) at (-7.75, 7.75) {};
		\node [style=WIDEmorph] (248) at (-8.5, 2.25) {$T_2$};
		\node [style=none] (249) at (-9, 2) {};
		\node [style=none] (250) at (-8.25, 2.25) {};
		\node [style=none] (251) at (-8.25, 0.5) {};
		\node [style=none] (252) at (-8, 2.25) {};
		\node [style=WIDEmorph] (253) at (-7.75, 0.5) {$\putt{}_1$};
		\node [style=none] (254) at (-7.25, 0.5) {};
		\node [style=none] (255) at (-8.25, 0.25) {};
		\node [style=none] (258) at (-6.5, 0.5) {};
		\node [style=WIDEmorph] (259) at (-6, 0.5) {$\putt{}_1$};
		\node [style=none] (260) at (-6.5, 0.25) {};
		\node [style=none] (261) at (-5.5, 0.5) {};
		\node [style={white_dot}] (262) at (-7.75, -1.75) {};
		\node [style={white_dot}] (263) at (-6.25, -1.75) {};
		\node [style=none] (264) at (-6.25, -4) {};
		\node [style={white_dot}] (265) at (-8.5, -3.75) {};
		\node [style=none] (266) at (-8.5, -4.75) {};
		\node [style=morph] (267) at (-7.75, -2.75) {$g_2$};
		\node [style=none] (268) at (-4.5, 1.75) {$=$};
		\node [style=morph] (269) at (-2.25, 4) {$g_2$};
		\node [style=WIDEmorph] (270) at (-2.75, 6.25) {$\putt{}_2$};
		\node [style=none] (271) at (-3.25, 6) {};
		\node [style=none] (272) at (-2.25, 6.25) {};
		\node [style={white_dot}] (273) at (-3, 3.25) {};
		\node [style=none] (274) at (-3.25, 2) {};
		\node [style=none] (275) at (-3.25, 6.5) {};
		\node [style=none] (276) at (-3.25, 7.25) {};
		\node [style=WIDEmorph] (277) at (-2.75, 2) {$T_2$};
		\node [style=none] (278) at (-3.25, 1.75) {};
		\node [style=none] (279) at (-2.5, 2) {};
		\node [style=none] (280) at (-2.25, -0.5) {};
		\node [style=none] (281) at (-2.25, 2) {};
		\node [style=WIDEmorph] (282) at (-1.75, -0.5) {$\putt{}_1$};
		\node [style=none] (283) at (-1.25, -0.5) {};
		\node [style=none] (284) at (-2.25, -0.75) {};
		\node [style=none] (291) at (-1.25, -4.25) {};
		\node [style={white_dot}] (292) at (-2.75, -3) {};
		\node [style=none] (293) at (-2.75, -4) {};
		\node [style=morph] (294) at (-2.25, -1.75) {$g_2$};
		\node [style={white_dot}] (295) at (-2.25, 0.5) {};
		\node [style=none] (296) at (-0.25, 1.75) {$=$};
		\node [style=morph] (297) at (1.75, 3.25) {$g_2$};
		\node [style=WIDEmorph] (298) at (1.25, 5.5) {$\putt{}_2$};
		\node [style=none] (299) at (0.75, 5.25) {};
		\node [style=none] (300) at (1.75, 5.5) {};
		\node [style={white_dot}] (301) at (1, 2.5) {};
		\node [style=none] (303) at (0.75, 5.75) {};
		\node [style=none] (304) at (0.75, 6.5) {};
		\node [style=none] (308) at (1.75, 0.5) {};
		\node [style=WIDEmorph] (310) at (2.25, 0.5) {$\putt{}_1$};
		\node [style=none] (311) at (2.75, 0.5) {};
		\node [style=none] (312) at (1.75, 0.25) {};
		\node [style=none] (313) at (2.75, -3.25) {};
		\node [style={white_dot}] (314) at (1.25, -2) {};
		\node [style=none] (315) at (1.25, -3) {};
		\node [style=morph] (316) at (1.75, -0.75) {$g_2$};
		\node [style={white_dot}] (317) at (1.75, 1.5) {};
		\node [style=none] (318) at (3.75, 1.75) {$=$};
		\node [style=morph] (319) at (5.75, 3.25) {$g_2$};
		\node [style=WIDEmorph] (320) at (5.25, 5.5) {$\putt{}_2$};
		\node [style=none] (321) at (4.75, 5.25) {};
		\node [style=none] (322) at (5.75, 5.5) {};
		\node [style={white_dot}] (323) at (5, 2.5) {};
		\node [style=none] (324) at (4.75, 5.75) {};
		\node [style=none] (325) at (4.75, 6.5) {};
		\node [style=none] (326) at (6.5, -3) {};
		\node [style={white_dot}] (329) at (5.75, 0.25) {};
		\node [style={white_dot}] (331) at (5.25, -2) {};
		\node [style=none] (332) at (5.25, -3) {};
		\node [style=morph] (333) at (5.75, -0.75) {$g_2$};
		\node [style={white_dot}] (334) at (6.5, -2) {};
		\node [style=none] (335) at (7.25, 1.75) {$=$};
		\node [style=morph] (336) at (9.25, 1) {$g_2$};
		\node [style=WIDEmorph] (337) at (8.75, 3.25) {$\putt{}_2$};
		\node [style=none] (338) at (8.25, 3) {};
		\node [style=none] (339) at (9.25, 3.25) {};
		\node [style={white_dot}] (340) at (8.5, 0.25) {};
		\node [style=none] (341) at (8.25, 3.5) {};
		\node [style=none] (342) at (8.25, 4.25) {};
		\node [style=none] (343) at (9.75, -0.75) {};
		\node [style=none] (346) at (8.5, -0.75) {};
		\node [style={white_dot}] (348) at (9.75, 0.25) {};
	\end{pgfonlayer}
	\begin{pgfonlayer}{edgelayer}
		\draw [style=new edge style 2] (117) to (119.center);
		\draw (121.center) to (122.center);
		\draw [in=135, out=-90, looseness=0.75] (242.center) to (244);
		\draw [in=90, out=-90] (244) to (245.center);
		\draw [in=-90, out=0] (244) to (240);
		\draw (246.center) to (247.center);
		\draw [in=90, out=-90] (252.center) to (251.center);
		\draw [in=90, out=-90] (250.center) to (258.center);
		\draw [style=new edge style 2] (263) to (264.center);
		\draw (243.center) to (240);
		\draw (266.center) to (265);
		\draw [in=270, out=150, looseness=0.50] (265) to (249.center);
		\draw [in=15, out=-90, looseness=1.25] (260.center) to (262);
		\draw [in=255, out=120, looseness=1.25] (262) to (255.center);
		\draw [in=90, out=-90] (262) to (267);
		\draw [in=15, out=-90] (267) to (265);
		\draw [style=new edge style 2, in=180, out=-90, looseness=0.50] (254.center) to (263);
		\draw [style=new edge style 2, in=270, out=0, looseness=0.50] (263) to (261.center);
		\draw [in=135, out=-90, looseness=0.75] (271.center) to (273);
		\draw [in=90, out=-90] (273) to (274.center);
		\draw [in=-90, out=0] (273) to (269);
		\draw (275.center) to (276.center);
		\draw (272.center) to (269);
		\draw (293.center) to (292);
		\draw [in=270, out=150, looseness=0.50] (292) to (278.center);
		\draw [in=15, out=-90] (294) to (292);
		\draw (284.center) to (294);
		\draw [style=new edge style 2] (283.center) to (291.center);
		\draw [in=45, out=-90, looseness=1.25] (281.center) to (295);
		\draw [in=285, out=135, looseness=1.25] (295) to (279.center);
		\draw (295) to (280.center);
		\draw [in=135, out=-90, looseness=0.75] (299.center) to (301);
		\draw [in=-90, out=0] (301) to (297);
		\draw (303.center) to (304.center);
		\draw (300.center) to (297);
		\draw (315.center) to (314);
		\draw [in=15, out=-90] (316) to (314);
		\draw (312.center) to (316);
		\draw [style=new edge style 2] (311.center) to (313.center);
		\draw (317) to (308.center);
		\draw [in=150, out=-90, looseness=0.50] (301) to (314);
		\draw [in=135, out=-90, looseness=0.75] (321.center) to (323);
		\draw [in=-90, out=0] (323) to (319);
		\draw (324.center) to (325.center);
		\draw (322.center) to (319);
		\draw (332.center) to (331);
		\draw [in=15, out=-90] (333) to (331);
		\draw (329) to (333);
		\draw [style=new edge style 2] (334) to (326.center);
		\draw [in=150, out=-90, looseness=0.50] (323) to (331);
		\draw [in=135, out=-90, looseness=0.75] (338.center) to (340);
		\draw [in=-90, out=0] (340) to (336);
		\draw (341.center) to (342.center);
		\draw (339.center) to (336);
		\draw [style=new edge style 2] (348) to (343.center);
		\draw (346.center) to (340);
	\end{pgfonlayer}
\end{tikzpicture}\]
\end{proof}
\begin{prop}
First-Last-Depending tagged lenses define a symmetric monoidal sub-category $\mathbf{fldepTL}[\mathcal{C}]$ of $\cat{TL}(\mathcal{C})$
\end{prop}
\begin{proof}
Given in the Appendix
\end{proof}

\section{Lens Laws Entailed by Tag Laws}
We now observe the entailment between compositional classes of tagged lenses and compositional classes of standard lenses, showing that each subcategory of tagged lenses defined by a design law maps into a corresponding subcategory of standard lenses. Since every tagged lens already induces a putget lens, the first new consequence is that change-depending tags force getput.
\begin{prop}[GetPut]
Let $(\putt{},g,T)$ be a tagged lens, if $T$ is change-depending then $L((\putt{},g,T))$ satisfies $\mathbf{getput}$. 
\end{prop}
\begin{proof}
Given in the appendix.
\end{proof}
An example of a lens which is well-behaved as a consequence of having a change-depending tag is \texttt{PutCountChanges}.
\begin{example}[PutCountChanges]
Since \texttt{CountChanges} defined by
\begin{align*}
    \texttt{CountChanges}:: (i,v) \times v' \times v'' \\
    \mapsto (i+1,v) \textrm{ if } v' \neq v'' \\
    \mapsto (i,v) \textrm{ if } v' = v'' \\
\end{align*} is a change-depending tag, it follows that $(\texttt{PutCountChanges},g) = \mathcal{L}((\putt{},g,\texttt{CountChanges}))$ satisfies $\textbf{getput}$ and so is well-behaved.
\end{example}
Next, every first-last-depending tag induces a lens which satisfies putput.
\begin{prop}
Let $(\putt{},g,T)$ be a tagged lens, if $T$ is first-last-depending then $L((\putt{},g,T))$ satisfies $\mathbf{putput}$. 
\end{prop}
\begin{proof}
\[\input{tagfigs/flputputa.tikz}\]
\[\begin{tikzpicture}[scale=0.5, {every node/.style}={scale=0.5}]
	\begin{pgfonlayer}{nodelayer}
		\node [style=WIDEmorph] (172) at (-7.25, 0.75) {$T$};
		\node [style={white_dot}] (173) at (-5.5, 1) {};
		\node [style=none] (174) at (-6.75, 0.75) {};
		\node [style=WIDEmorph] (175) at (-5.25, 4) {$\putt{}$};
		\node [style=none] (177) at (-5.75, 3.75) {};
		\node [style={white_dot}] (178) at (-4, -0.75) {};
		\node [style=none] (180) at (-4.75, 4) {};
		\node [style=none] (181) at (-4, -2.25) {};
		\node [style=none] (182) at (-5.75, 5) {};
		\node [style=none] (183) at (-5.75, 4.25) {};
		\node [style=WIDEmorph] (184) at (-7, -2) {$T$};
		\node [style=none] (185) at (-6.75, -2) {};
		\node [style=none] (186) at (-6.5, -2) {};
		\node [style=morph] (187) at (-6.75, -3.25) {$g$};
		\node [style=none] (188) at (-7.5, -1.75) {};
		\node [style=none] (189) at (-7.5, -2.25) {};
		\node [style={white_dot}] (190) at (-5.5, -3.5) {};
		\node [style={white_dot}] (191) at (-5.75, -0.75) {};
		\node [style=none] (192) at (-5.5, -5) {};
		\node [style={white_dot}] (193) at (-7.25, -4) {};
		\node [style=none] (194) at (-7.25, -5) {};
		\node [style=none] (195) at (-7.75, 1) {};
		\node [style=none] (197) at (-7.75, 0.5) {};
		\node [style=none] (198) at (-7, 0.75) {};
		\node [style=none] (199) at (-8.75, 0.25) {$=$};
		\node [style=WIDEmorph] (200) at (-1.5, 1.5) {$T$};
		\node [style=none] (202) at (-1, 1.5) {};
		\node [style=WIDEmorph] (203) at (-1, 3.25) {$\putt{}$};
		\node [style=none] (204) at (-1.5, 3) {};
		\node [style={white_dot}] (205) at (0.75, 0.5) {};
		\node [style=none] (206) at (-0.5, 3.25) {};
		\node [style=none] (207) at (1.25, -1.5) {};
		\node [style=none] (208) at (-1.5, 4.25) {};
		\node [style=none] (209) at (-1.5, 3.5) {};
		\node [style=WIDEmorph] (210) at (-1.25, -1.25) {$T$};
		\node [style=none] (211) at (-1, -1.25) {};
		\node [style=none] (212) at (-0.75, -1.25) {};
		\node [style=morph] (213) at (-1, -2.5) {$g$};
		\node [style=none] (214) at (-1.75, -1) {};
		\node [style=none] (215) at (-1.75, -1.5) {};
		\node [style={white_dot}] (216) at (0.25, -2.75) {};
		\node [style=none] (218) at (0.25, -4.25) {};
		\node [style={white_dot}] (219) at (-1.5, -3.25) {};
		\node [style=none] (220) at (-1.5, -4.25) {};
		\node [style=none] (221) at (-2, 1.75) {};
		\node [style=none] (222) at (-2, 1.25) {};
		\node [style=none] (223) at (-1.25, 1.5) {};
		\node [style=none] (224) at (-3, 0.25) {$=$};
		\node [style=WIDEmorph] (225) at (3.75, 1) {$T$};
		\node [style=none] (226) at (4.25, 1) {};
		\node [style=WIDEmorph] (227) at (4.25, 2.75) {$\putt{}$};
		\node [style=none] (228) at (3.75, 2.5) {};
		\node [style={white_dot}] (229) at (6, 0) {};
		\node [style=none] (230) at (4.75, 2.75) {};
		\node [style=none] (231) at (6, -3) {};
		\node [style=none] (232) at (3.75, 3.75) {};
		\node [style=none] (233) at (3.75, 3) {};
		\node [style=none] (235) at (4, 1) {};
		\node [style=morph] (237) at (4, -1.25) {$g$};
		\node [style=none] (239) at (3.25, 0.75) {};
		\node [style={white_dot}] (240) at (5, -2) {};
		\node [style=none] (241) at (5, -3) {};
		\node [style={white_dot}] (242) at (3.5, -2) {};
		\node [style=none] (243) at (3.5, -3) {};
		\node [style=none] (244) at (3.25, 1.25) {};
		\node [style=none] (247) at (2.25, 0.25) {$=$};
		\node [style={white_dot}] (260) at (8.75, -0.25) {};
		\node [style=none] (261) at (8.75, -1.25) {};
		\node [style=none] (265) at (7.25, 0.25) {$=$};
		\node [style=WIDEmorph] (266) at (9, 1.25) {$P$};
		\node [style=none] (267) at (8.5, 1.5) {};
		\node [style=none] (268) at (9.5, 1.25) {};
		\node [style=none] (269) at (9.5, -1.25) {};
		\node [style=none] (270) at (8.5, 2.25) {};
		\node [style=none] (273) at (7.75, -1.25) {};
		\node [style=none] (277) at (8.5, 1) {};
	\end{pgfonlayer}
	\begin{pgfonlayer}{edgelayer}
		\draw [style=new edge style 5, in=30, out=-90, looseness=0.50] (180.center) to (178);
		\draw [style=new edge style 5, in=150, out=-75, looseness=0.75] (174.center) to (178);
		\draw (182.center) to (183.center);
		\draw [style=new edge style 5] (178) to (181.center);
		\draw [in=90, out=-90] (185.center) to (187);
		\draw [style=new edge style 2, in=30, out=-90, looseness=0.50] (191) to (190);
		\draw [style=new edge style 2, in=150, out=-75, looseness=0.75] (186.center) to (190);
		\draw [in=120, out=-90, looseness=0.75] (189.center) to (193);
		\draw (193) to (194.center);
		\draw [in=-90, out=0] (193) to (187);
		\draw [style=new edge style 2] (190) to (192.center);
		\draw [in=90, out=-90, looseness=1.25] (177.center) to (195.center);
		\draw [style=new edge style 2, in=150, out=-90] (198.center) to (191);
		\draw [in=90, out=-90] (197.center) to (188.center);
		\draw [style=new edge style 2, in=90, out=-90] (173) to (191);
		\draw [style=new edge style 5, in=30, out=-90, looseness=0.50] (206.center) to (205);
		\draw [style=new edge style 5, in=150, out=-75, looseness=0.75] (202.center) to (205);
		\draw (208.center) to (209.center);
		\draw [style=new edge style 5, in=90, out=-90] (205) to (207.center);
		\draw [in=90, out=-90] (211.center) to (213);
		\draw [style=new edge style 2, in=150, out=-75, looseness=0.75] (212.center) to (216);
		\draw [in=120, out=-90, looseness=0.75] (215.center) to (219);
		\draw (219) to (220.center);
		\draw [in=-90, out=0] (219) to (213);
		\draw [style=new edge style 2] (216) to (218.center);
		\draw [in=90, out=-90, looseness=1.25] (204.center) to (221.center);
		\draw [in=90, out=-90] (222.center) to (214.center);
		\draw [style=new edge style 2, in=45, out=-90, looseness=0.50] (223.center) to (216);
		\draw [style=new edge style 5, in=60, out=-90, looseness=0.75] (230.center) to (229);
		\draw [style=new edge style 5, in=150, out=-75, looseness=0.75] (226.center) to (229);
		\draw (232.center) to (233.center);
		\draw [style=new edge style 5, in=90, out=-90] (229) to (231.center);
		\draw [in=90, out=-90] (235.center) to (237);
		\draw [in=120, out=-90, looseness=0.75] (239.center) to (242);
		\draw (242) to (243.center);
		\draw [in=-90, out=0] (242) to (237);
		\draw [style=new edge style 2] (240) to (241.center);
		\draw [in=90, out=-90, looseness=1.25] (228.center) to (244.center);
		\draw [style=new edge style 2] (260) to (261.center);
		\draw [style=new edge style 5, in=90, out=-90] (268.center) to (269.center);
		\draw (267.center) to (270.center);
		\draw [in=270, out=90] (273.center) to (277.center);
	\end{pgfonlayer}
\end{tikzpicture}.\]
\end{proof}
The lens \texttt{PutFlag} is an example of a lens which satisfies \textbf{putput} as a consequence of its tag being first-last depending.
\begin{example}[PutFlag]
Since
\[\texttt{Flag}:: (i,v) (v',v'') \mapsto (1,v) \]
is first-last depending it follows that $(\texttt{PutFlag},g) = L((\putt{},g,\texttt{Flag}))$ satisfies $\textbf{putput}$.
\end{example}
Putting together both entailment results leads to an immediate corollary.
\begin{corr}
Let $(\putt{},g,T)$ be a tagged lens, if $T$ is change-depending and first-last-depending then $L((\putt{},g,T))$ is a vwb-lens.
\end{corr}
Meaning that such a tag at most introduces a constant complement:
\begin{example}[PutScaled]
Since
\[\texttt{Scale} (x,y) (v,v') = (x,v' y / v) \]
is a change-depending and first-last-depending tag, the induced lens $(\texttt{PutScaled},g)$ is a vwb lens.
\end{example}

\section{Summary}
A variety of new symmetric monoidal categories embedding into the category $\cat{L}(\mathcal{C})$ of lenses over $\mathcal{C}$ have been defined. We summarise the results via a tree of embeddings between classes. Using the following key: [$\cat{L}(\mathcal{C})$: Category of Lenses] [$\cat{pgL}(\mathcal{C})$: Category of Lenses satisfying PutGet] [$\cat{wbL}(\mathcal{C})$: Category of wb-Lenses] [$\cat{vwbL}(\mathcal{C})$: Category of vwb-lenses
   ] [$\cat{pgppL}(\mathcal{C})$: Category of Lenses satisfying PutGet and PutPut
   ] [$\cat{TL}(\mathcal{C})$: Category of Tagged Lenses
   ] [$\mathbf{cdepTL}[\mathcal{C}]$: Category of Change-Depending Tagged Lenses
   ] [$\cat{fldepTL}(\mathcal{C})$: Category of First-Last-depending Tagged Lenses ]
and using an arrow $\mathcal{C} \rightarrow \mathcal{D}$ to encode the statement ``$\mathcal{C}$ is a monoidal subcategory of $\mathcal{D}$", we obtain the following tree of embeddings between sub-classes of lenses:
\[\tikzfigscale{0.8}{tagfigs/tree}\]
This tree shows that tagged lenses provide a clean test case for compositional entailment of lens laws by laws on side-effects. 

The notions of lens and of side-effect used to develop this test case are by no means the most general, and so left open is the question of whether the story of this test case lifts well to more general notions of both lens and side-effect.
For instance, outside of the scope of this paper is the extension of these observations to richer approaches for defining lenses such as the frameworks of profunctor optics \cite{Riley2018CategoriesOO, clarke2020profunctor, lens_dynamics_doubles2021}, symmetric lenses \cite{clarke_diagrammatic} , delta lenses \cite{clarke2020profunctor}, lenses based on the Grothendieck construction \cite{spivak2020generalized}, fibred lenses \cite{braithwaite2021fibre}, update structures \cite{hefford2020categories, safari}, higher lenses \cite{higher_lenses}, and structures similar to lenses used to model features of game theory \cite{hedges_games, lens_comp_game, dilavore2020games, capucci2021translating} and cybernetics \cite{cyber_kittens}. A second direction beyond the  scope of this paper is whether these design laws and their consequences can be re-phrased within those richer frameworks for introducing side-effects, \cite{abou_monads, monadic_combinator, monads_entangled_notions, monadic_composing} to set-based lenses.

%\color{blue}
%\begin{prop}
%\label{lensfairlyequiv}
%Fwb lenses $(S,V,g,p)$ in a category $\mathcal{C}$ with finite products are weak update structures that have ignorable gets, $\mix{} := \pi_2$, and $\copyy{} := \delta_V$.
%\end{prop}
%\begin{proof}
%The repeat update axiom follows directly from the PutPut rule for Fwb lenses
%\end{proof}
%As a result, as noted in \cite{hefford2020categories}, in any Fwb lens $\putt{} \circ \get{}$ is idempotent and can be used to buid a vwb in $\cat{SPLIT(C)}$.
%\color{black}

%\begin{equation}\label{qmdatabase}
%    \tikzfigscale{1}{figs/qmdatabasefhilb}
%\end{equation}
%Since decoherence $\deco{}$ is a magma co-magma homomorphism for $\mix{}$ and $\copyy{}$ FALSE, applying $\deco{}$ to the embedding of this update structure into $\cpm{\cat{FHilb}}$ generates a new update structure with classical types, $\get{}$ becomes a quantum measurement and $\putt{}$ becomes the encoding of classical data into a quantum system inside the database. The update structure is weak because the system $S_{2}$ is not initially assumed to be in a classical (decohered) state, whereas after insertion by $\putt{}$ the new value of $S_{2}$ is guaranteed to be classical.

%\begin{equation}\label{qmdatabase}
%    \tikzfigscale{1}{figs/qmdatabasedeco}
%\end{equation}

%This update is not causal (Trace Preserving) CITE[], and in fact the put only ignores information in a particular basis. It turns out the the minor edit 
%\begin{prop}
%The tuple of morphisms $(\putt{},\get{},\mix{},\copyy{})$ in $\cpm{\cat{FHilb}}$ defined by 
\section*{Acknowledgements}
The majority of this work was completed while the author was a PhD student in the Quantum Group of the Department of Computer Science at the University of Oxford. Refinements have since been made while affiliated with University College London and, subsequently, CentraleSup\'elec, Universit\'e Paris-Saclay.

The author (MW) is grateful to V.~Wang and J.~Hefford for useful conversations. MW was funded by the Engineering and Physical Sciences Research Council under grant EP/L015242/1 while at the University of Oxford and under grant EP/W524335/1 while at University College London. This work was also funded by the French National Research Agency (ANR): by the project TaQC, ANR-22-CE47-0012; and within the framework of France 2030 by the projects EPIQ, ANR-22-PETQ-0007; OQULUS, ANR-23-PETQ-0013; HQI-Acquisition, ANR-22-PNCQ-0001; and HQI-R\&D, ANR-22-PNCQ-0002.

\bibliography{bibliography}

\appendix

\section{The (Symmetric Monoidal) Category of Tagged Lenses}
\begin{prop}[The Category $\cat{TL}(\mathcal{C})$]
For every cartesian monoidal category $\mathcal{C}$ a symmetric monoidal category of tagged lenses $\cat{TL}(\mathcal{C})$ can be constructed with
\begin{itemize}
    \item Objects given by the objects of $\mathcal{C}$
    \item Morphisms given by \[\mathbf{TL}[\mathcal{C}](S,S') := \{ \textrm{ Tagged lenses from $S$ to $S'$ }\}\]
\end{itemize}
along with the sequential and parallel composition of morphisms as defined in the main text.
\end{prop}
\begin{proof}
The composition rule applied to the puts ($\putt{}$) and gets ($g$) is the standard composition rule in the symmetric monoidal category of lenses over a category $\mathcal{C}$ with finite products. The remaining condition to check concerns the tag $T$. We show that $T$ indeed commutes with $\putt{}$ as required.
\[\input{tagfigs/provecommute_a.tikz}\]
\[\tikzfigscale{0.9}{tagfigs/provecommute_b}\]
\[\tikzfigscale{0.88}{tagfigs/provecommute_c}\]
\[\tikzfigscale{0.85}{tagfigs/provecommute_d}\]
\[\begin{tikzpicture}[scale=0.5, {every node/.style}={scale=0.5}]
	\begin{pgfonlayer}{nodelayer}
		\node [style=none] (0) at (100, 0) {$=$};
		\node [style=none] (1) at (101.25, -1.75) {};
		\node [style=WIDEmorph] (2) at (101.75, -0.25) {$T$};
		\node [style=none] (3) at (101.25, -0.5) {};
		\node [style=none] (4) at (102.25, -0.25) {};
		\node [style=none] (5) at (103.25, -1.75) {};
		\node [style=none] (6) at (102, -0.25) {};
		\node [style=none] (7) at (103, -1.75) {};
		\node [style=none] (8) at (101.25, 0) {};
		\node [style=none] (9) at (102.25, 1) {};
		\node [style=WIDEmorph] (10) at (102.75, 1.25) {$\putt{}$};
		\node [style=none] (11) at (102.25, 1.5) {};
		\node [style=none] (12) at (103.25, 1.25) {};
		\node [style=none] (13) at (102.5, -1.75) {};
		\node [style=none] (14) at (102.25, 2.25) {};
	\end{pgfonlayer}
	\begin{pgfonlayer}{edgelayer}
		\draw [in=90, out=-90] (4.center) to (5.center);
		\draw [in=-90, out=90, looseness=1.25] (1.center) to (3.center);
		\draw [in=90, out=-90] (6.center) to (7.center);
		\draw [in=270, out=90] (8.center) to (9.center);
		\draw [in=90, out=-90] (12.center) to (13.center);
		\draw [in=270, out=90] (11.center) to (14.center);
	\end{pgfonlayer}
\end{tikzpicture}\]
%Next we check the get non-effecting condition: 
%\[\tikzfig{tagfigs/provegetnon_a}\]
%\[\tikzfig{tagfigs/provegetnon_b}\]
%and finally the self-commuting condition.
Now let us prove associativity of sequential composition:

\[ \tikzfigscale{0.4}{comp_ass_proof_1} = \tikzfigscale{0.4}{comp_ass_proof_2} = \tikzfigscale{0.4}{comp_ass_proof_3} = \tikzfigscale{0.4}{comp_ass_proof_4} \]

\[ = \tikzfigscale{0.3}{comp_ass_proof_5} = \tikzfigscale{0.3}{comp_ass_proof_6} = \tikzfigscale{0.3}{comp_ass_proof_7} \]

\[ = \tikzfigscale{0.4}{comp_ass_proof_8} = \tikzfigscale{0.4}{comp_ass_proof_9} = \tikzfigscale{0.4}{comp_ass_proof_10} \]
 
\[ = \tikzfigscale{0.4}{comp_ass_proof_11} =  \tikzfigscale{0.4}{comp_ass_proof_12} = \tikzfigscale{0.4}{comp_ass_proof_13} = \tikzfigscale{0.4}{comp_ass_proof_14} \]

\end{proof}
Now we prove unitality of sequential composition. We take the identity tagged lens to be the identity vwb lens defined by $\putt{}(v,v') = v'$, $g(v) = v$, and $T(v,v',v'') = v$.
The left unit equation is
\[\tikzfigscale{0.6}{compose_id_2}\]
and the right unit equation is
\[\tikzfigscale{0.6}{compose_id}.\]
For the monoidal structure, it is easy to see that the tensor product returns a tagged lens and is associative and unital along with following the interchange law.

\section{An intuitive description of the tag composition rule}

 Here we provide an intuitive description of what at first glance is the most confusing component, the action of the composed tag $T$, with the aim of explaining why it turns out to be the correct tag for building $L((\putt{}_2,g_2,T_2)) \circ L((\putt{}_1,g_1,T_1))$. We begin by breaking down the composed tag $T$ in terms of the new view (orange) and the old view (green). The composed lens $P$ applies first $P_1$ then $P_2$. The tag $T$ which represents that tag that occurs when $P$ is applied, is composed of two parts. In the latter part of $T$:
the tag $T_1$ is applied to the global system $V''$ by
\begin{itemize}
    \item Applying the tag $T_2$ to the system $V'$ according to the old (green) and new (orange) values of $V$
    \item Lifting the action of $T_2$ on $V'$ to an action on $V''$ by using the vwb lens $(\putt{}_1,g_1)$.
\end{itemize}
As in the following diagram
\[\begin{tikzpicture}[scale=0.7, {every node/.style}={scale=0.7}]
	\begin{pgfonlayer}{nodelayer}
		\node [style=WIDEmorph] (18) at (4, 4) {$T_1$};
		\node [style=none] (21) at (3.5, 3.75) {};
		\node [style=morph] (23) at (3.5, 2.5) {$g_2$};
		\node [style=WIDEmorph] (26) at (3, 5.25) {$\putt{}_2$};
		\node [style=none] (28) at (2.5, 5) {};
		\node [style=none] (53) at (3.5, 5.25) {};
		\node [style={white_dot}] (55) at (2.75, 1.75) {};
		\node [style=none] (56) at (2.75, 0.75) {};
		\node [style=none] (60) at (2.5, 5.5) {};
		\node [style=none] (61) at (2.5, 6.25) {};
		\node [style=none] (62) at (3.5, 4.25) {};
		\node [style=none] (63) at (4.25, 4) {};
		\node [style=none] (75) at (4.5, 4) {};
		\node [style=none] (104) at (4.25, 1.5) {};
		\node [style=none] (105) at (5.25, 1.5) {};
		\node [style=none] (106) at (4.75, 0) {\texttt{Old Value of V}};
		\node [style=none] (107) at (7.5, 2.5) {\texttt{Future Value of V}};
		\node [style=none] (109) at (2.5, 6.75) {$V''$};
		\node [style=none] (110) at (2.75, 0.25) {$V''$};
		\node [style=none] (111) at (3.25, 4.75) {$V'$};
		\node [style=none] (112) at (3.25, 3) {$V'$};
		\node [style=none] (113) at (2, 3.5) {$V''$};
		\node [style=none] (114) at (-0.75, 4.75) {};
		\node [style=none] (115) at (1.75, 5.25) {};
		\node [style=none] (116) at (1.75, 2) {};
		\node [style=none] (117) at (-1.75, 5.75) {\texttt{Lifting The Tagging of V' to V''}};
		\node [style=none] (118) at (4.25, 1.25) {};
		\node [style=none] (119) at (4.75, 0.25) {};
		\node [style=none] (120) at (7.5, 2.25) {};
		\node [style=none] (121) at (5.75, 1.75) {};
	\end{pgfonlayer}
	\begin{pgfonlayer}{edgelayer}
		\draw [in=90, out=-90] (21.center) to (23);
		\draw [in=135, out=-90, looseness=0.75] (28.center) to (55);
		\draw [in=90, out=-90] (55) to (56.center);
		\draw [in=-90, out=0] (55) to (23);
		\draw (60.center) to (61.center);
		\draw [in=90, out=-90, looseness=1.25] (53.center) to (62.center);
		\draw [style=new edge style 2, in=90, out=-90, looseness=1.25] (63.center) to (104.center);
		\draw [style=new edge style 5, in=270, out=90] (105.center) to (75.center);
		\draw [style=new edge style 0, in=90, out=-180, looseness=0.75] (115.center) to (114.center);
		\draw [style=new edge style 0, in=-90, out=-180, looseness=0.75] (116.center) to (114.center);
		\draw [style=new edge style 0, in=90, out=-90, looseness=1.25] (118.center) to (119.center);
		\draw [style=new edge style 0, in=-180, out=0] (121.center) to (120.center);
	\end{pgfonlayer}
\end{tikzpicture}\]
The former part of $T$:
\[\begin{tikzpicture}[scale=0.7, {every node/.style}={scale=0.7}]
	\begin{pgfonlayer}{nodelayer}
		\node [style=none] (0) at (-2, 4.5) {};
		\node [style=none] (1) at (-2, 3.5) {};
		\node [style=WIDEmorph] (2) at (-1.5, 3.5) {$T_2$};
		\node [style=none] (3) at (-2, 3.25) {};
		\node [style=none] (4) at (-1.25, 3.5) {};
		\node [style=none] (5) at (0, 2.5) {};
		\node [style=none] (6) at (-1, 3.5) {};
		\node [style=WIDEmorph] (7) at (0.5, 2.5) {$\putt{}_1$};
		\node [style=none] (8) at (0, 2.25) {};
		\node [style=none] (9) at (1, 2.5) {};
		\node [style=none] (10) at (-0.25, 1.25) {};
		\node [style=WIDEmorph] (11) at (0.25, 1.25) {$T_1$};
		\node [style=none] (12) at (-0.25, 1) {};
		\node [style=none] (13) at (0.5, 1.25) {};
		\node [style=none] (14) at (0.75, 1.25) {};
		\node [style=morph] (15) at (-1.25, 0) {$g_2$};
		\node [style=morph] (16) at (-0.5, -2.75) {$g_2$};
		\node [style=none] (17) at (-0.5, -1.5) {};
		\node [style=WIDEmorph] (18) at (0, -1.5) {$\putt{}_1$};
		\node [style=none] (19) at (-0.5, -1.75) {};
		\node [style=none] (20) at (0.5, -1.5) {};
		\node [style={white_dot}] (21) at (-1.75, -2) {};
		\node [style={white_dot}] (22) at (-1.25, -3.5) {};
		\node [style={white_dot}] (23) at (1.25, -2.75) {};
		\node [style={white_dot}] (24) at (2.75, -2.75) {};
		\node [style=none] (25) at (-1.25, -4.5) {};
		\node [style=none] (26) at (1.25, -4.25) {};
		\node [style=none] (27) at (2.75, -4.25) {};
	\end{pgfonlayer}
	\begin{pgfonlayer}{edgelayer}
		\draw [in=90, out=-90] (0.center) to (1.center);
		\draw [in=90, out=-90] (6.center) to (5.center);
		\draw [in=90, out=-90, looseness=1.25] (8.center) to (10.center);
		\draw [style=new edge style 2, in=105, out=-75] (13.center) to (23);
		\draw [style=new edge style 2, in=135, out=-60] (20.center) to (23);
		\draw [style=new edge style 2] (23) to (26.center);
		\draw [style=new edge style 5] (24) to (27.center);
		\draw [in=135, out=-90, looseness=0.75] (3.center) to (21);
		\draw [in=-90, out=15] (21) to (15);
		\draw [in=90, out=-90, looseness=1.25] (12.center) to (15);
		\draw [in=90, out=-90] (4.center) to (17.center);
		\draw [in=150, out=-90] (21) to (22);
		\draw [in=-90, out=15, looseness=1.25] (22) to (16);
		\draw [in=270, out=90] (16) to (19.center);
		\draw (22) to (25.center);
		\draw [style=new edge style 5, in=135, out=-90, looseness=0.75] (14.center) to (24);
		\draw [style=new edge style 5, in=90, out=-90] (9.center) to (24);
	\end{pgfonlayer}
\end{tikzpicture}\]
Consists of applying the tag $T_1$ to the system $V''$, where the values inserted into the tag $T_1$ are the original value of the intermediate system $V'$:
\[\begin{tikzpicture}[scale=0.7, {every node/.style}={scale=0.7}]
	\begin{pgfonlayer}{nodelayer}
		\node [style=none] (0) at (-0.75, 2.5) {};
		\node [style=none] (1) at (-0.75, 1.5) {};
		\node [style=WIDEmorph] (2) at (-0.25, 1.5) {$T_2$};
		\node [style=none] (3) at (-0.75, 1.25) {};
		\node [style=none] (4) at (0, 1.5) {};
		\node [style=none] (5) at (2.5, -1) {};
		\node [style=none] (6) at (0.25, 1.5) {};
		\node [style=morph] (7) at (0.25, -1.75) {$g_2$};
		\node [style=none] (8) at (0.25, -0.5) {};
		\node [style=WIDEmorph] (9) at (0.75, -0.5) {$\putt{}_1$};
		\node [style=none] (10) at (0.25, -0.75) {};
		\node [style=none] (11) at (1.25, -0.5) {};
		\node [style={white_dot}] (12) at (-0.5, -2.5) {};
		\node [style=none] (13) at (2, -2.5) {};
		\node [style=none] (14) at (-0.5, -3.5) {};
		\node [style=none] (15) at (3.5, -3.5) {\texttt{Past Value of V}};
		\node [style=none] (16) at (2.25, -2.5) {};
		\node [style=none] (17) at (3.5, -3) {};
		\node [style=none] (18) at (-0.75, -3) {};
		\node [style=none] (19) at (0, 0.25) {};
		\node [style=none] (20) at (-5.25, 0) {};
		\node [style=none] (21) at (-5.25, 0.5) {$\texttt{Past Value of Intermediate System V'}$};
	\end{pgfonlayer}
	\begin{pgfonlayer}{edgelayer}
		\draw [in=90, out=-90] (0.center) to (1.center);
		\draw [in=90, out=-90] (6.center) to (5.center);
		\draw [style=new edge style 2, in=90, out=-60, looseness=1.25] (11.center) to (13.center);
		\draw [in=90, out=-90] (4.center) to (8.center);
		\draw [in=-90, out=15, looseness=1.25] (12) to (7);
		\draw [in=270, out=90] (7) to (10.center);
		\draw (12) to (14.center);
		\draw [in=150, out=-90, looseness=0.50] (3.center) to (12);
		\draw [style=new edge style 0, in=0, out=-180] (19.center) to (20.center);
		\draw [style=new edge style 0, in=-105, out=-180, looseness=0.75] (18.center) to (20.center);
		\draw [style=new edge style 0, in=90, out=0] (16.center) to (17.center);
	\end{pgfonlayer}
\end{tikzpicture}\]
and the new value of the intermediate system $V'$ after the tagged lens $P_{1}$ has been used to update $V'$:
\[\begin{tikzpicture}[scale=0.7, {every node/.style}={scale=0.7}]
	\begin{pgfonlayer}{nodelayer}
		\node [style=none] (0) at (-2, 3.5) {};
		\node [style=none] (1) at (-2, 2.5) {};
		\node [style=WIDEmorph] (2) at (-1.5, 2.5) {$T_2$};
		\node [style=none] (3) at (-2, 2.25) {};
		\node [style=none] (4) at (-1.25, 2.5) {};
		\node [style=none] (5) at (0, 1.25) {};
		\node [style=none] (6) at (-1, 2.5) {};
		\node [style=WIDEmorph] (7) at (0.5, 1.25) {$\putt{}_1$};
		\node [style=none] (8) at (0, 1) {};
		\node [style=none] (9) at (1, 1.25) {};
		\node [style=none] (10) at (-0.25, 0) {};
		\node [style=WIDEmorph] (11) at (0.25, 0) {$T_1$};
		\node [style=none] (12) at (-0.25, -0.25) {};
		\node [style=none] (13) at (0.5, 0) {};
		\node [style=none] (14) at (0.75, 0) {};
		\node [style=morph] (15) at (-0.75, -1.25) {$g_2$};
		\node [style=none] (16) at (-1.25, 1.5) {};
		\node [style={white_dot}] (17) at (-1.5, -2.25) {};
		\node [style=none] (18) at (1, -2.5) {};
		\node [style={white_dot}] (19) at (2, -1) {};
		\node [style=none] (20) at (-1.5, -3.25) {};
		\node [style=none] (21) at (2, -2.5) {};
		\node [style=none] (22) at (2, 3.25) {\texttt{Value of System V' after update by }$P_1$};
		\node [style=none] (23) at (-0.25, 2) {};
		\node [style=none] (24) at (2, 2.75) {};
	\end{pgfonlayer}
	\begin{pgfonlayer}{edgelayer}
		\draw [in=90, out=-90] (0.center) to (1.center);
		\draw [in=90, out=-90] (6.center) to (5.center);
		\draw [in=90, out=-90, looseness=1.25] (8.center) to (10.center);
		\draw [style=new edge style 2, in=90, out=-90] (13.center) to (18.center);
		\draw [style=new edge style 5] (19) to (21.center);
		\draw [in=135, out=-90, looseness=0.75] (3.center) to (17);
		\draw [in=-90, out=15] (17) to (15);
		\draw [in=90, out=-90, looseness=1.25] (12.center) to (15);
		\draw [in=90, out=-90] (4.center) to (16.center);
		\draw [style=new edge style 5, in=135, out=-90, looseness=0.75] (14.center) to (19);
		\draw [style=new edge style 5, in=90, out=-90] (9.center) to (19);
		\draw (17) to (20.center);
		\draw [style=new edge style 0, in=-180, out=0] (23.center) to (24.center);
	\end{pgfonlayer}
\end{tikzpicture}\]
With the action of the composed tag explained

\section{The Composition of Tagged Lenses is Consistent With Standard Lens Composition}
\begin{prop}
The induced lens construction defines a strict monoidal injective-on-objects functor \[\mathcal{L}:\cat{TL}[\mathcal{C}] \longrightarrow \cat{pgL}[\mathcal{C}]\] from the category $\cat{TL}[\mathcal{C}]$ of tagged lenses over $\mathcal{C}$ into the category $\cat{pgL}[\mathcal{C}]$ of putget lenses over $\mathcal{C}$.
\end{prop}
\begin{proof}
The functor is defined on objects by $\mathcal{L}(V) := V$ and on morphisms by producing the lens induced by the tagged lens $\mathcal{L}((\putt{},g,T)) := L((\putt{},g,T))$. To prove functoriality it must be shown that $P_{(\putt{},g,T)}$ indeed factors in the following way \[\begin{tikzpicture}[scale=0.7, {every node/.style}={scale=0.7}]
	\begin{pgfonlayer}{nodelayer}
		\node [style=none] (1) at (-2.75, -0.75) {};
		\node [style=WIDEmorph] (9) at (-2.25, 0.25) {$P$};
		\node [style=none] (10) at (-2.75, 0.5) {};
		\node [style=none] (11) at (-2.75, 0) {};
		\node [style=none] (13) at (-1.75, 0.25) {};
		\node [style=none] (15) at (-1.75, -0.75) {};
		\node [style=none] (16) at (-2.75, 1.25) {};
		\node [style=none] (17) at (-0.75, 0) {$=$};
		\node [style=WIDEmorph] (18) at (0.75, 0.75) {$T$};
		\node [style=none] (21) at (1, 0.75) {};
		\node [style=none] (22) at (1.25, 0.75) {};
		\node [style=morph] (23) at (1, -0.5) {$g$};
		\node [style=WIDEmorph] (26) at (1.5, 2) {$\putt{}$};
		\node [style=none] (27) at (0.25, 1) {};
		\node [style=none] (28) at (0.25, 0.5) {};
		\node [style={white_dot}] (30) at (2.25, -0.75) {};
		\node [style=none] (31) at (1, 1.75) {};
		\node [style=none] (53) at (2, 2) {};
		\node [style=none] (54) at (2.25, -2.25) {};
		\node [style={white_dot}] (55) at (0.5, -1.25) {};
		\node [style=none] (56) at (0.5, -2.25) {};
		\node [style=none] (57) at (1, 3) {};
		\node [style=none] (58) at (1, 2.25) {};
	\end{pgfonlayer}
	\begin{pgfonlayer}{edgelayer}
		\draw [in=90, out=-90] (13.center) to (15.center);
		\draw (10.center) to (16.center);
		\draw [in=-90, out=90, looseness=1.25] (1.center) to (11.center);
		\draw [in=90, out=-90] (21.center) to (23);
		\draw [in=270, out=90] (27.center) to (31.center);
		\draw [in=30, out=-90, looseness=0.50] (53.center) to (30);
		\draw [in=150, out=-75, looseness=0.75] (22.center) to (30);
		\draw [in=120, out=-90, looseness=0.75] (28.center) to (55);
		\draw (55) to (56.center);
		\draw [in=-90, out=0] (55) to (23);
		\draw (57.center) to (58.center);
		\draw (30) to (54.center);
	\end{pgfonlayer}
\end{tikzpicture}\]
and that $T$ is indeed a lens tag for $\putt{}$. Firstly by definition $P$ is
\[\begin{tikzpicture}[scale=0.5, {every node/.style}={scale=0.5}]
	\begin{pgfonlayer}{nodelayer}
		\node [style=none] (109) at (-10.25, 3.5) {};
		\node [style=WIDEmorph] (110) at (-9.75, 4.5) {$P$};
		\node [style=none] (111) at (-10.25, 4.75) {};
		\node [style=none] (112) at (-10.25, 4.25) {};
		\node [style=none] (113) at (-9.25, 4.5) {};
		\node [style=none] (114) at (-9.25, 3.5) {};
		\node [style=none] (115) at (-10.25, 5.5) {};
		\node [style=none] (116) at (-8.25, 4.25) {$=$};
		\node [style=WIDEmorph] (117) at (-5.75, 4.5) {$P_1$};
		\node [style=none] (118) at (-6.25, 4.25) {};
		\node [style=morph] (119) at (-6.25, 3.5) {$g_2$};
		\node [style=WIDEmorph] (120) at (-6.75, 5.75) {$P_2$};
		\node [style=none] (121) at (-7.25, 5.5) {};
		\node [style=none] (122) at (-6.25, 5.75) {};
		\node [style=none] (123) at (-5.25, 2.75) {};
		\node [style={white_dot}] (124) at (-7, 2.75) {};
		\node [style=none] (125) at (-7, 2) {};
		\node [style=none] (126) at (-7.25, 6) {};
		\node [style=none] (127) at (-7.25, 6.75) {};
		\node [style=none] (128) at (-6.25, 4.75) {};
		\node [style=none] (129) at (-5.25, 4.5) {};
		\node [style=none] (130) at (-4.25, 4.25) {$=$};
		\node [style=WIDEmorph] (131) at (-1.5, 3.25) {$P_1$};
		\node [style=none] (132) at (-2, 3) {};
		\node [style=morph] (133) at (-2, 2.25) {$g_2$};
		\node [style=none] (137) at (-1, 1.5) {};
		\node [style={white_dot}] (138) at (-2.75, 1.5) {};
		\node [style=none] (139) at (-2.75, 0.75) {};
		\node [style=none] (143) at (-1, 3.25) {};
		\node [style=WIDEmorph] (144) at (-2.75, 6.75) {$T_2$};
		\node [style=none] (145) at (-2.5, 6.75) {};
		\node [style=none] (146) at (-2.25, 6.75) {};
		\node [style=morph] (147) at (-2.5, 5.5) {$g_2$};
		\node [style=WIDEmorph] (148) at (-2, 8) {$\putt{}_2$};
		\node [style=none] (149) at (-3.25, 7) {};
		\node [style=none] (150) at (-3.25, 6.5) {};
		\node [style={white_dot}] (151) at (-1.25, 5.25) {};
		\node [style=none] (152) at (-2.5, 7.75) {};
		\node [style=none] (153) at (-1.5, 8) {};
		\node [style=none] (154) at (-2, 3.5) {};
		\node [style={white_dot}] (155) at (-3, 4.75) {};
		\node [style=none] (157) at (-2.5, 9) {};
		\node [style=none] (158) at (-2.5, 8.25) {};
		\node [style=none] (159) at (0, 4.25) {$=$};
		\node [style=morph] (162) at (2.25, 0.5) {$g_2$};
		\node [style={white_dot}] (164) at (1.5, -0.25) {};
		\node [style=none] (165) at (1.5, -1) {};
		\node [style=WIDEmorph] (167) at (1.25, 7.5) {$T_2$};
		\node [style=none] (168) at (1.5, 7.5) {};
		\node [style=none] (169) at (1.75, 7.5) {};
		\node [style=morph] (170) at (1.5, 6.25) {$g_2$};
		\node [style=WIDEmorph] (171) at (2, 8.75) {$\putt{}_2$};
		\node [style=none] (172) at (0.75, 7.75) {};
		\node [style=none] (173) at (0.75, 7.25) {};
		\node [style={white_dot}] (174) at (2.75, 6) {};
		\node [style=none] (175) at (1.5, 8.5) {};
		\node [style=none] (176) at (2.5, 8.75) {};
		\node [style={white_dot}] (178) at (1, 5.5) {};
		\node [style=none] (179) at (1.5, 9.75) {};
		\node [style=none] (180) at (1.5, 9) {};
		\node [style=WIDEmorph] (181) at (2.5, 3.5) {$T_1$};
		\node [style=none] (182) at (2.75, 3.5) {};
		\node [style=none] (183) at (3, 3.5) {};
		\node [style=morph] (184) at (2.75, 2.25) {$g_1$};
		\node [style=WIDEmorph] (185) at (3.25, 4.75) {$\putt{}_1$};
		\node [style=none] (186) at (2, 3.75) {};
		\node [style=none] (187) at (2, 3.25) {};
		\node [style={white_dot}] (188) at (4, 2) {};
		\node [style=none] (189) at (2.75, 4.5) {};
		\node [style=none] (190) at (3.75, 4.75) {};
		\node [style=none] (191) at (4, 0.5) {};
		\node [style={white_dot}] (192) at (2.25, 1.5) {};
		\node [style=none] (195) at (2.75, 5) {};
	\end{pgfonlayer}
	\begin{pgfonlayer}{edgelayer}
		\draw [in=90, out=-90] (113.center) to (114.center);
		\draw (111.center) to (115.center);
		\draw [in=-90, out=90, looseness=1.25] (109.center) to (112.center);
		\draw [in=90, out=-90] (118.center) to (119);
		\draw [in=135, out=-90, looseness=0.75] (121.center) to (124);
		\draw (124) to (125.center);
		\draw [in=-90, out=0] (124) to (119);
		\draw (126.center) to (127.center);
		\draw [in=90, out=-90, looseness=1.25] (122.center) to (128.center);
		\draw [in=90, out=-90] (129.center) to (123.center);
		\draw [in=90, out=-90] (132.center) to (133);
		\draw (138) to (139.center);
		\draw [in=-90, out=0] (138) to (133);
		\draw [in=90, out=-90] (143.center) to (137.center);
		\draw [in=90, out=-90] (145.center) to (147);
		\draw [in=270, out=90] (149.center) to (152.center);
		\draw [in=30, out=-90, looseness=0.50] (153.center) to (151);
		\draw [in=150, out=-75, looseness=0.75] (146.center) to (151);
		\draw [in=120, out=-90, looseness=0.75] (150.center) to (155);
		\draw [in=-90, out=0] (155) to (147);
		\draw (157.center) to (158.center);
		\draw [in=90, out=-90] (151) to (154.center);
		\draw [in=90, out=-90, looseness=0.50] (155) to (138);
		\draw (164) to (165.center);
		\draw [in=-90, out=0] (164) to (162);
		\draw [in=90, out=-90] (168.center) to (170);
		\draw [in=270, out=90] (172.center) to (175.center);
		\draw [in=30, out=-90, looseness=0.50] (176.center) to (174);
		\draw [in=150, out=-75, looseness=0.75] (169.center) to (174);
		\draw [in=120, out=-90, looseness=0.75] (173.center) to (178);
		\draw [in=-90, out=0] (178) to (170);
		\draw (179.center) to (180.center);
		\draw [in=120, out=-90, looseness=0.75] (178) to (164);
		\draw [in=90, out=-90] (182.center) to (184);
		\draw [in=270, out=90] (186.center) to (189.center);
		\draw [in=30, out=-90, looseness=0.50] (190.center) to (188);
		\draw [in=150, out=-75, looseness=0.75] (183.center) to (188);
		\draw [in=120, out=-90, looseness=0.75] (187.center) to (192);
		\draw [in=-90, out=0] (192) to (184);
		\draw [in=90, out=-90] (188) to (191.center);
		\draw [in=90, out=-90] (174) to (195.center);
		\draw [in=90, out=-90] (192) to (162);
	\end{pgfonlayer}
\end{tikzpicture}\]
we want to show the above expression is equivalent to
\[\begin{tikzpicture}[scale=0.5, {every node/.style}={scale=0.5}]
	\begin{pgfonlayer}{nodelayer}
		\node [style=none] (17) at (3.5, 0) {$=$};
		\node [style=WIDEmorph] (18) at (0.75, 0.75) {$T$};
		\node [style=none] (21) at (1, 0.75) {};
		\node [style=none] (22) at (1.25, 0.75) {};
		\node [style=morph] (23) at (1, -0.5) {$g$};
		\node [style=WIDEmorph] (26) at (1.5, 2) {$\putt{}$};
		\node [style=none] (27) at (0.25, 1) {};
		\node [style=none] (28) at (0.25, 0.5) {};
		\node [style={white_dot}] (30) at (2.25, -0.75) {};
		\node [style=none] (31) at (1, 1.75) {};
		\node [style=none] (53) at (2, 2) {};
		\node [style=none] (54) at (2.25, -2.25) {};
		\node [style={white_dot}] (55) at (0.5, -1.25) {};
		\node [style=none] (56) at (0.5, -2.25) {};
		\node [style=none] (57) at (1, 3) {};
		\node [style=none] (58) at (1, 2.25) {};
		\node [style=WIDEmorph] (59) at (7.25, 7.5) {$P_1$};
		\node [style=none] (60) at (6.75, 7.25) {};
		\node [style=morph] (61) at (6.75, 6.5) {$g_2$};
		\node [style=WIDEmorph] (62) at (6.25, 8.75) {$\putt{}_2$};
		\node [style=none] (63) at (5.75, 8.5) {};
		\node [style=none] (64) at (6.75, 8.75) {};
		\node [style={white_dot}] (65) at (11.5, -7.5) {};
		\node [style={white_dot}] (66) at (6, 5.75) {};
		\node [style=none] (67) at (6, 5) {};
		\node [style=none] (68) at (5.75, 9) {};
		\node [style=none] (69) at (5.75, 9.75) {};
		\node [style=none] (70) at (6.75, 7.75) {};
		\node [style=none] (71) at (7.75, 7.5) {};
		\node [style=WIDEmorph] (72) at (7.5, 3.5) {$T_1$};
		\node [style=none] (73) at (7, 3.25) {};
		\node [style=morph] (74) at (7, 2.5) {$g_2$};
		\node [style=WIDEmorph] (75) at (6.5, 4.75) {$\putt{}_2$};
		\node [style=none] (76) at (6, 4.5) {};
		\node [style=none] (77) at (7, 4.75) {};
		\node [style={white_dot}] (78) at (6.25, 1.75) {};
		\node [style=none] (79) at (4.75, 0.25) {};
		\node [style=none] (82) at (7, 3.75) {};
		\node [style=none] (83) at (7.75, 3.5) {};
		\node [style=none] (84) at (8, 3.5) {};
		\node [style=WIDEmorph] (85) at (5.25, 0.25) {$T_2$};
		\node [style=none] (86) at (4.75, 0) {};
		\node [style=none] (87) at (5.5, 0.25) {};
		\node [style=none] (88) at (6.75, -0.75) {};
		\node [style=none] (89) at (5.75, 0.25) {};
		\node [style=WIDEmorph] (90) at (7.25, -0.75) {$\putt{}_1$};
		\node [style=none] (91) at (6.75, -1) {};
		\node [style=none] (92) at (7.75, -0.75) {};
		\node [style=none] (93) at (6.5, -2) {};
		\node [style=WIDEmorph] (94) at (7, -2) {$T_1$};
		\node [style=none] (95) at (6.5, -2.25) {};
		\node [style=none] (96) at (7.25, -2) {};
		\node [style=none] (97) at (7.5, -2) {};
		\node [style=morph] (98) at (5.5, -3.25) {$g_2$};
		\node [style=morph] (99) at (6.25, -6) {$g_2$};
		\node [style=none] (100) at (6.25, -4.75) {};
		\node [style=WIDEmorph] (101) at (6.75, -4.75) {$\putt{}_1$};
		\node [style=none] (102) at (6.25, -5) {};
		\node [style=none] (103) at (7.25, -4.75) {};
		\node [style={white_dot}] (104) at (5, -5.25) {};
		\node [style={white_dot}] (105) at (5.5, -6.75) {};
		\node [style={white_dot}] (106) at (8, -6) {};
		\node [style={white_dot}] (107) at (9.5, -6) {};
		\node [style=none] (110) at (11.5, -8.5) {};
		\node [style=morph] (111) at (7.25, -7) {$g_1$};
		\node [style=morph] (113) at (7.25, -8) {$g_2$};
		\node [style={white_dot}] (114) at (6.25, -9.25) {};
		\node [style=none] (115) at (6.25, -10) {};
	\end{pgfonlayer}
	\begin{pgfonlayer}{edgelayer}
		\draw [in=90, out=-90] (21.center) to (23);
		\draw [in=270, out=90] (27.center) to (31.center);
		\draw [in=30, out=-90, looseness=0.50] (53.center) to (30);
		\draw [in=150, out=-75, looseness=0.75] (22.center) to (30);
		\draw [in=120, out=-90, looseness=0.75] (28.center) to (55);
		\draw (55) to (56.center);
		\draw [in=-90, out=0] (55) to (23);
		\draw (57.center) to (58.center);
		\draw (30) to (54.center);
		\draw [in=90, out=-90] (60.center) to (61);
		\draw [in=135, out=-90, looseness=0.75] (63.center) to (66);
		\draw [in=90, out=-90] (66) to (67.center);
		\draw [in=-90, out=0] (66) to (61);
		\draw (68.center) to (69.center);
		\draw [in=90, out=-90, looseness=1.25] (64.center) to (70.center);
		\draw [in=90, out=-90] (71.center) to (65);
		\draw [in=90, out=-90] (73.center) to (74);
		\draw [in=135, out=-90, looseness=0.75] (76.center) to (78);
		\draw [in=90, out=-90] (78) to (79.center);
		\draw [in=-90, out=0] (78) to (74);
		\draw [in=90, out=-90, looseness=1.25] (77.center) to (82.center);
		\draw [in=90, out=-90] (89.center) to (88.center);
		\draw [in=90, out=-90, looseness=1.25] (91.center) to (93.center);
		\draw [style=new edge style 2, in=30, out=-90, looseness=1.25] (83.center) to (106);
		\draw [style=new edge style 2, in=105, out=-75] (96.center) to (106);
		\draw [style=new edge style 2, in=135, out=-60] (103.center) to (106);
		\draw [in=135, out=-90, looseness=0.75] (86.center) to (104);
		\draw [in=-90, out=15] (104) to (98);
		\draw [in=90, out=-90, looseness=1.25] (95.center) to (98);
		\draw [in=90, out=-90] (87.center) to (100.center);
		\draw [in=150, out=-90] (104) to (105);
		\draw [in=-90, out=15, looseness=1.25] (105) to (99);
		\draw [in=270, out=90] (99) to (102.center);
		\draw [style=new edge style 5, in=135, out=-90, looseness=0.75] (97.center) to (107);
		\draw [style=new edge style 5, in=285, out=60] (107) to (84.center);
		\draw [style=new edge style 5, in=90, out=-90] (92.center) to (107);
		\draw [style=new edge style 5, in=-180, out=-90, looseness=1.25] (107) to (65);
		\draw (65) to (110.center);
		\draw (111) to (113);
		\draw [in=270, out=0, looseness=1.25] (114) to (113);
		\draw [style=new edge style 2, in=90, out=-90] (106) to (111);
		\draw [in=165, out=-75] (105) to (114);
		\draw (114) to (115.center);
	\end{pgfonlayer}
\end{tikzpicture}\]
Indeed this can be demonstrated by the following steps:
\[\tikzfigscale{0.87}{tagfigs/catproof4a}\]
\[\begin{tikzpicture}[scale=0.5, {every node/.style}={scale=0.5}]
	\begin{pgfonlayer}{nodelayer}
		\node [style=none] (0) at (-9.75, 0.25) {$=$};
		\node [style=WIDEmorph] (1) at (-6, 7.75) {$P_1$};
		\node [style=WIDEmorph] (2) at (-7, 9) {$\putt{}_2$};
		\node [style=none] (3) at (-6.5, 9) {};
		\node [style={white_dot}] (4) at (-1, -7.75) {};
		\node [style=none] (5) at (-7.5, 9.25) {};
		\node [style=none] (6) at (-7.5, 10) {};
		\node [style=none] (7) at (-6.5, 8) {};
		\node [style=none] (8) at (-5.5, 7.75) {};
		\node [style=WIDEmorph] (9) at (-5.25, 3.25) {$T_1$};
		\node [style=none] (10) at (-5.75, 3) {};
		\node [style=morph] (11) at (-5.75, 2.25) {$g_2$};
		\node [style={white_dot}] (12) at (-6.5, 4.25) {};
		\node [style={white_dot}] (13) at (-6.5, 1.5) {};
		\node [style=none] (14) at (-8, 0) {};
		\node [style=none] (15) at (-5.75, 3.5) {};
		\node [style=none] (16) at (-5, 3.25) {};
		\node [style=none] (17) at (-4.75, 3.25) {};
		\node [style=WIDEmorph] (18) at (-7.5, 0) {$T_2$};
		\node [style=none] (19) at (-8, -0.25) {};
		\node [style=none] (20) at (-7.25, 0) {};
		\node [style=none] (21) at (-6, -1) {};
		\node [style=none] (22) at (-7, 0) {};
		\node [style=WIDEmorph] (23) at (-5.5, -1) {$\putt{}_1$};
		\node [style=none] (24) at (-6, -1.25) {};
		\node [style=none] (25) at (-5, -1) {};
		\node [style=none] (26) at (-6.25, -2.25) {};
		\node [style=WIDEmorph] (27) at (-5.75, -2.25) {$T_1$};
		\node [style=none] (28) at (-6.25, -2.5) {};
		\node [style=none] (29) at (-5.5, -2.25) {};
		\node [style=none] (30) at (-5.25, -2.25) {};
		\node [style=morph] (31) at (-7.25, -3.5) {$g_2$};
		\node [style=morph] (32) at (-6.5, -6.25) {$g_2$};
		\node [style=none] (33) at (-6.5, -5) {};
		\node [style=WIDEmorph] (34) at (-6, -5) {$\putt{}_1$};
		\node [style=none] (35) at (-6.5, -5.25) {};
		\node [style=none] (36) at (-5.5, -5) {};
		\node [style={white_dot}] (37) at (-7.75, -5.5) {};
		\node [style={white_dot}] (38) at (-7.25, -7) {};
		\node [style={white_dot}] (39) at (-4.75, -6.25) {};
		\node [style={white_dot}] (40) at (-3.25, -6.25) {};
		\node [style=none] (41) at (-1, -8.75) {};
		\node [style=morph] (42) at (-5.5, -7.25) {$g_1$};
		\node [style=morph] (43) at (-5.5, -8.25) {$g_2$};
		\node [style={white_dot}] (44) at (-6.5, -9.5) {};
		\node [style=none] (45) at (-6.5, -10.25) {};
		\node [style=none] (46) at (-6.5, 7.5) {};
		\node [style={white_dot}] (47) at (-7.25, 6.5) {};
		\node [style=none] (48) at (-7.5, 8.75) {};
		\node [style=none] (49) at (-0.75, 0) {$=$};
		\node [style=WIDEmorph] (50) at (4, 6.5) {$P_1$};
		\node [style=WIDEmorph] (51) at (3, 7.75) {$\putt{}_2$};
		\node [style=none] (52) at (3.5, 7.75) {};
		\node [style={white_dot}] (53) at (7.75, -6.5) {};
		\node [style=none] (54) at (2.5, 8) {};
		\node [style=none] (55) at (2.5, 8.75) {};
		\node [style=none] (56) at (3.5, 6.75) {};
		\node [style=none] (57) at (4.5, 6.5) {};
		\node [style=WIDEmorph] (58) at (3.5, 4.5) {$T_1$};
		\node [style=none] (59) at (3, 4.25) {};
		\node [style=morph] (60) at (3, 3.5) {$g_2$};
		\node [style={white_dot}] (61) at (2.25, 2.75) {};
		\node [style=none] (62) at (0.75, 1.25) {};
		\node [style=none] (63) at (3, 4.75) {};
		\node [style=none] (64) at (3.75, 4.5) {};
		\node [style=none] (65) at (4, 4.5) {};
		\node [style=WIDEmorph] (66) at (1.25, 1.25) {$T_2$};
		\node [style=none] (67) at (0.75, 1) {};
		\node [style=none] (68) at (1.5, 1.25) {};
		\node [style=none] (69) at (2.75, 0.25) {};
		\node [style=none] (70) at (1.75, 1.25) {};
		\node [style=WIDEmorph] (71) at (3.25, 0.25) {$\putt{}_1$};
		\node [style=none] (72) at (2.75, 0) {};
		\node [style=none] (73) at (3.75, 0.25) {};
		\node [style=none] (74) at (2.5, -1) {};
		\node [style=WIDEmorph] (75) at (3, -1) {$T_1$};
		\node [style=none] (76) at (2.5, -1.25) {};
		\node [style=none] (77) at (3.25, -1) {};
		\node [style=none] (78) at (3.5, -1) {};
		\node [style=morph] (79) at (1.5, -2.25) {$g_2$};
		\node [style=morph] (80) at (2.25, -5) {$g_2$};
		\node [style=none] (81) at (2.25, -3.75) {};
		\node [style=WIDEmorph] (82) at (2.75, -3.75) {$\putt{}_1$};
		\node [style=none] (83) at (2.25, -4) {};
		\node [style=none] (84) at (3.25, -3.75) {};
		\node [style={white_dot}] (85) at (1, -4.25) {};
		\node [style={white_dot}] (86) at (1.5, -5.75) {};
		\node [style={white_dot}] (87) at (4, -5) {};
		\node [style={white_dot}] (88) at (5.5, -5) {};
		\node [style=none] (89) at (7.75, -7.5) {};
		\node [style=morph] (90) at (3.25, -6) {$g_1$};
		\node [style=morph] (91) at (3.25, -7) {$g_2$};
		\node [style={white_dot}] (92) at (2.25, -8.25) {};
		\node [style=none] (93) at (2.25, -9) {};
		\node [style=none] (94) at (3.5, 6.25) {};
		\node [style=none] (95) at (2.5, 7.5) {};
	\end{pgfonlayer}
	\begin{pgfonlayer}{edgelayer}
		\draw (5.center) to (6.center);
		\draw [in=90, out=-90, looseness=1.25] (3.center) to (7.center);
		\draw [in=90, out=-90, looseness=0.75] (8.center) to (4);
		\draw [in=90, out=-90] (10.center) to (11);
		\draw [in=90, out=-90] (13) to (14.center);
		\draw [in=-90, out=0] (13) to (11);
		\draw [in=90, out=-90, looseness=1.25] (12) to (15.center);
		\draw [in=90, out=-90] (22.center) to (21.center);
		\draw [in=90, out=-90, looseness=1.25] (24.center) to (26.center);
		\draw [style=new edge style 2, in=30, out=-90, looseness=1.25] (16.center) to (39);
		\draw [style=new edge style 2, in=105, out=-75] (29.center) to (39);
		\draw [style=new edge style 2, in=135, out=-60] (36.center) to (39);
		\draw [in=135, out=-90, looseness=0.75] (19.center) to (37);
		\draw [in=-90, out=15] (37) to (31);
		\draw [in=90, out=-90, looseness=1.25] (28.center) to (31);
		\draw [in=90, out=-90] (20.center) to (33.center);
		\draw [in=150, out=-90] (37) to (38);
		\draw [in=-90, out=15, looseness=1.25] (38) to (32);
		\draw [in=270, out=90] (32) to (35.center);
		\draw [style=new edge style 5, in=135, out=-90, looseness=0.75] (30.center) to (40);
		\draw [style=new edge style 5, in=285, out=60] (40) to (17.center);
		\draw [style=new edge style 5, in=90, out=-90] (25.center) to (40);
		\draw [style=new edge style 5, in=-180, out=-90, looseness=1.25] (40) to (4);
		\draw (4) to (41.center);
		\draw (42) to (43);
		\draw [in=270, out=0, looseness=1.25] (44) to (43);
		\draw [style=new edge style 2, in=90, out=-90] (39) to (42);
		\draw [in=165, out=-75] (38) to (44);
		\draw (44) to (45.center);
		\draw [in=150, out=-90, looseness=1.25] (47) to (12);
		\draw [in=-90, out=30] (12) to (46.center);
		\draw [in=150, out=-90, looseness=0.75] (48.center) to (13);
		\draw (54.center) to (55.center);
		\draw [in=90, out=-90, looseness=1.25] (52.center) to (56.center);
		\draw [in=90, out=-90, looseness=0.75] (57.center) to (53);
		\draw [in=90, out=-90] (59.center) to (60);
		\draw [in=90, out=-90] (61) to (62.center);
		\draw [in=-90, out=0] (61) to (60);
		\draw [in=90, out=-90] (70.center) to (69.center);
		\draw [in=90, out=-90, looseness=1.25] (72.center) to (74.center);
		\draw [style=new edge style 2, in=30, out=-90, looseness=1.25] (64.center) to (87);
		\draw [style=new edge style 2, in=105, out=-75] (77.center) to (87);
		\draw [style=new edge style 2, in=135, out=-60] (84.center) to (87);
		\draw [in=135, out=-90, looseness=0.75] (67.center) to (85);
		\draw [in=-90, out=15] (85) to (79);
		\draw [in=90, out=-90, looseness=1.25] (76.center) to (79);
		\draw [in=90, out=-90] (68.center) to (81.center);
		\draw [in=150, out=-90] (85) to (86);
		\draw [in=-90, out=15, looseness=1.25] (86) to (80);
		\draw [in=270, out=90] (80) to (83.center);
		\draw [style=new edge style 5, in=135, out=-90, looseness=0.75] (78.center) to (88);
		\draw [style=new edge style 5, in=285, out=60] (88) to (65.center);
		\draw [style=new edge style 5, in=90, out=-90] (73.center) to (88);
		\draw [style=new edge style 5, in=-180, out=-90, looseness=1.25] (88) to (53);
		\draw (53) to (89.center);
		\draw (90) to (91);
		\draw [in=270, out=0, looseness=1.25] (92) to (91);
		\draw [style=new edge style 2, in=90, out=-90] (87) to (90);
		\draw [in=165, out=-75] (86) to (92);
		\draw (92) to (93.center);
		\draw [in=150, out=-90, looseness=0.75] (95.center) to (61);
		\draw [in=90, out=-90] (94.center) to (63.center);
	\end{pgfonlayer}
\end{tikzpicture}\]
\[\begin{tikzpicture}[scale=0.5, {every node/.style}={scale=0.5}]
	\begin{pgfonlayer}{nodelayer}
		\node [style=none] (45) at (-6.5, -10.25) {};
		\node [style=none] (96) at (-8.25, 0) {$=$};
		\node [style=WIDEmorph] (97) at (-3.5, 6.5) {$P_1$};
		\node [style=WIDEmorph] (98) at (-4.5, 7.75) {$\putt{}_2$};
		\node [style=none] (99) at (-4, 7.75) {};
		\node [style={white_dot}] (100) at (0.25, -6.5) {};
		\node [style=none] (101) at (-5, 8) {};
		\node [style=none] (102) at (-5, 8.75) {};
		\node [style=none] (103) at (-4, 6.75) {};
		\node [style=none] (104) at (-3, 6.5) {};
		\node [style=WIDEmorph] (105) at (-4, 4.5) {$T_1$};
		\node [style=none] (106) at (-4.5, 4.25) {};
		\node [style=morph] (107) at (-4.5, 3.5) {$g_2$};
		\node [style=none] (108) at (-4.5, 4.75) {};
		\node [style=none] (109) at (-3.75, 4.5) {};
		\node [style=none] (110) at (-3.5, 4.5) {};
		\node [style={white_dot}] (111) at (-6.75, 1) {};
		\node [style={white_dot}] (112) at (-6, 1) {};
		\node [style=none] (113) at (-4.75, 0.25) {};
		\node [style={white_dot}] (114) at (-5, 1) {};
		\node [style=WIDEmorph] (115) at (-4.25, 0.25) {$\putt{}_1$};
		\node [style=none] (116) at (-4.75, 0) {};
		\node [style=none] (117) at (-3.75, 0.25) {};
		\node [style=none] (118) at (-5, -1) {};
		\node [style=WIDEmorph] (119) at (-4.5, -1) {$T_1$};
		\node [style=none] (120) at (-5, -1.25) {};
		\node [style=none] (121) at (-4.25, -1) {};
		\node [style=none] (122) at (-4, -1) {};
		\node [style=morph] (123) at (-6, -2.25) {$g_2$};
		\node [style=morph] (124) at (-5.25, -5) {$g_2$};
		\node [style=none] (125) at (-5.25, -3.75) {};
		\node [style=WIDEmorph] (126) at (-4.75, -3.75) {$\putt{}_1$};
		\node [style=none] (127) at (-5.25, -4) {};
		\node [style=none] (128) at (-4.25, -3.75) {};
		\node [style={white_dot}] (129) at (-6.5, -4.25) {};
		\node [style={white_dot}] (130) at (-6, -5.75) {};
		\node [style={white_dot}] (131) at (-3.5, -5) {};
		\node [style={white_dot}] (132) at (-2, -5) {};
		\node [style=none] (133) at (0.25, -7.5) {};
		\node [style=morph] (134) at (-4.25, -6) {$g_1$};
		\node [style=morph] (135) at (-4.25, -7) {$g_2$};
		\node [style={white_dot}] (136) at (-5.25, -8.25) {};
		\node [style=none] (137) at (-5.25, -9) {};
		\node [style=none] (138) at (-4, 6.25) {};
		\node [style=none] (139) at (-5, 7.5) {};
		\node [style=none] (140) at (-7, 2.5) {};
		\node [style=WIDEmorph] (141) at (-6.5, 2.5) {$T_2$};
		\node [style=none] (142) at (-7, 2.25) {};
		\node [style=none] (143) at (-6.25, 2.5) {};
		\node [style=none] (144) at (-6, 2.5) {};
		\node [style=none] (145) at (-5, 2.5) {};
		\node [style=WIDEmorph] (146) at (-4.5, 2.5) {$T_2$};
		\node [style=none] (147) at (-5, 2.25) {};
		\node [style=none] (148) at (-4.25, 2.5) {};
		\node [style=none] (149) at (-4, 2.5) {};
		\node [style=none] (150) at (1.25, 0) {$=$};
		\node [style=WIDEmorph] (151) at (6, 6.5) {$P_1$};
		\node [style=WIDEmorph] (152) at (5, 7.75) {$\putt{}_2$};
		\node [style=none] (153) at (5.5, 7.75) {};
		\node [style={white_dot}] (154) at (9.75, -6.5) {};
		\node [style=none] (155) at (4.5, 8) {};
		\node [style=none] (156) at (4.5, 8.75) {};
		\node [style=none] (157) at (5.5, 6.75) {};
		\node [style=none] (158) at (6.5, 6.5) {};
		\node [style=WIDEmorph] (159) at (5.5, 4.5) {$T_1$};
		\node [style=none] (160) at (5, 4.25) {};
		\node [style=morph] (161) at (5, 3) {$g_2$};
		\node [style=none] (162) at (5, 4.75) {};
		\node [style=none] (163) at (5.75, 4.5) {};
		\node [style=none] (164) at (6, 4.5) {};
		\node [style={white_dot}] (165) at (2.75, 1) {};
		\node [style=none] (166) at (4.75, 0.25) {};
		\node [style=WIDEmorph] (167) at (5.25, 0.25) {$\putt{}_1$};
		\node [style=none] (168) at (4.75, 0) {};
		\node [style=none] (169) at (5.75, 0.25) {};
		\node [style=none] (170) at (4.5, -1) {};
		\node [style=WIDEmorph] (171) at (5, -1) {$T_1$};
		\node [style=none] (172) at (4.5, -1.25) {};
		\node [style=none] (173) at (5.25, -1) {};
		\node [style=none] (174) at (5.5, -1) {};
		\node [style=morph] (175) at (3.5, -2.25) {$g_2$};
		\node [style=morph] (176) at (4.25, -5) {$g_2$};
		\node [style=none] (177) at (4.25, -3.75) {};
		\node [style=WIDEmorph] (178) at (4.75, -3.75) {$\putt{}_1$};
		\node [style=none] (179) at (4.25, -4) {};
		\node [style=none] (180) at (5.25, -3.75) {};
		\node [style={white_dot}] (181) at (3, -4.25) {};
		\node [style={white_dot}] (182) at (3.5, -5.75) {};
		\node [style={white_dot}] (183) at (6, -5) {};
		\node [style={white_dot}] (184) at (7.5, -5) {};
		\node [style=none] (185) at (9.75, -7.5) {};
		\node [style=morph] (186) at (5.25, -6) {$g_1$};
		\node [style=morph] (187) at (5.25, -7) {$g_2$};
		\node [style={white_dot}] (188) at (4.25, -8.25) {};
		\node [style=none] (189) at (4.25, -9) {};
		\node [style=none] (190) at (5.5, 6.25) {};
		\node [style=none] (191) at (4.5, 7.5) {};
		\node [style=none] (192) at (2.5, 2.5) {};
		\node [style=WIDEmorph] (193) at (3, 2.5) {$T_2$};
		\node [style=none] (194) at (2.5, 2.25) {};
		\node [style=none] (195) at (3.25, 2.5) {};
		\node [style=none] (196) at (3.5, 2.5) {};
	\end{pgfonlayer}
	\begin{pgfonlayer}{edgelayer}
		\draw (101.center) to (102.center);
		\draw [in=90, out=-90, looseness=1.25] (99.center) to (103.center);
		\draw [in=90, out=-90, looseness=0.75] (104.center) to (100);
		\draw [in=90, out=-90] (106.center) to (107);
		\draw [in=90, out=-90] (114) to (113.center);
		\draw [in=90, out=-90, looseness=1.25] (116.center) to (118.center);
		\draw [style=new edge style 2, in=30, out=-90, looseness=1.25] (109.center) to (131);
		\draw [style=new edge style 2, in=105, out=-75] (121.center) to (131);
		\draw [style=new edge style 2, in=135, out=-60] (128.center) to (131);
		\draw [in=135, out=-90, looseness=0.75] (111) to (129);
		\draw [in=-90, out=15] (129) to (123);
		\draw [in=90, out=-90, looseness=1.25] (120.center) to (123);
		\draw [in=90, out=-90] (112) to (125.center);
		\draw [in=150, out=-90] (129) to (130);
		\draw [in=-90, out=15, looseness=1.25] (130) to (124);
		\draw [in=270, out=90] (124) to (127.center);
		\draw [style=new edge style 5, in=135, out=-90, looseness=0.75] (122.center) to (132);
		\draw [style=new edge style 5, in=285, out=60] (132) to (110.center);
		\draw [style=new edge style 5, in=90, out=-90] (117.center) to (132);
		\draw [style=new edge style 5, in=-180, out=-90, looseness=1.25] (132) to (100);
		\draw (100) to (133.center);
		\draw (134) to (135);
		\draw [in=270, out=0, looseness=1.25] (136) to (135);
		\draw [style=new edge style 2, in=90, out=-90] (131) to (134);
		\draw [in=165, out=-75] (130) to (136);
		\draw (136) to (137.center);
		\draw [in=90, out=-90] (138.center) to (108.center);
		\draw [in=105, out=-90] (142.center) to (111);
		\draw [in=-90, out=30, looseness=0.75] (111) to (147.center);
		\draw [in=105, out=-90] (143.center) to (112);
		\draw [in=-90, out=45] (112) to (148.center);
		\draw [in=60, out=-90] (149.center) to (114);
		\draw [in=-90, out=120] (114) to (144.center);
		\draw [in=-90, out=90] (145.center) to (107);
		\draw [in=90, out=-90] (139.center) to (140.center);
		\draw (155.center) to (156.center);
		\draw [in=90, out=-90, looseness=1.25] (153.center) to (157.center);
		\draw [in=90, out=-90, looseness=0.75] (158.center) to (154);
		\draw [in=90, out=-90] (160.center) to (161);
		\draw [in=90, out=-90, looseness=1.25] (168.center) to (170.center);
		\draw [style=new edge style 2, in=30, out=-90, looseness=1.25] (163.center) to (183);
		\draw [style=new edge style 2, in=105, out=-75] (173.center) to (183);
		\draw [style=new edge style 2, in=135, out=-60] (180.center) to (183);
		\draw [in=135, out=-90, looseness=0.75] (165) to (181);
		\draw [in=-90, out=15] (181) to (175);
		\draw [in=90, out=-90, looseness=1.25] (172.center) to (175);
		\draw [in=150, out=-90] (181) to (182);
		\draw [in=-90, out=15, looseness=1.25] (182) to (176);
		\draw [in=270, out=90] (176) to (179.center);
		\draw [style=new edge style 5, in=135, out=-90, looseness=0.75] (174.center) to (184);
		\draw [style=new edge style 5, in=285, out=60] (184) to (164.center);
		\draw [style=new edge style 5, in=90, out=-90] (169.center) to (184);
		\draw [style=new edge style 5, in=-180, out=-90, looseness=1.25] (184) to (154);
		\draw (154) to (185.center);
		\draw (186) to (187);
		\draw [in=270, out=0, looseness=1.25] (188) to (187);
		\draw [style=new edge style 2, in=90, out=-90] (183) to (186);
		\draw [in=165, out=-75] (182) to (188);
		\draw (188) to (189.center);
		\draw [in=90, out=-90] (190.center) to (162.center);
		\draw [in=105, out=-90] (194.center) to (165);
		\draw [in=90, out=-90] (191.center) to (192.center);
		\draw [in=30, out=-90, looseness=0.75] (161) to (165);
		\draw [in=90, out=-90] (196.center) to (166.center);
		\draw [in=-90, out=90, looseness=0.50] (177.center) to (195.center);
	\end{pgfonlayer}
\end{tikzpicture}\]
\[\input{tagfigs/catproof4c.tikz}\]
\[\tikzfigscale{0.94}{tagfigs/catproof4d}\]
\end{proof}

\section{First-Last Dependence is preserved by composition}
\begin{prop}
First-Last Depending tagged lenses define a symmetric monoidal sub-category $\mathbf{fldepTL}[\mathcal{C}]$ of $\cat{TL}(\mathcal{C})$
\end{prop}
\begin{proof}
\[\input{tagfigs/flcompose_1.tikz}\]
\[\tikzfigscale{0.88}{tagfigs/flcompose_2}\]
\[\tikzfigscale{0.89}{tagfigs/flcompose_3}\]
\[\tikzfigscale{0.87}{tagfigs/flcompose_4}\]
\[\tikzfigscale{0.84}{tagfigs/flcompose_5}\]
\[\tikzfigscale{0.95}{tagfigs/flcompose_6}\]
\[\begin{tikzpicture}[scale=0.5, {every node/.style}={scale=0.5}]
	\begin{pgfonlayer}{nodelayer}
		\node [style=none] (1565) at (169.75, -0.25) {$=$};
		\node [style=WIDEmorph] (1566) at (175.5, 5.75) {$T_1$};
		\node [style=none] (1567) at (175, 5.5) {};
		\node [style=WIDEmorph] (1568) at (174.25, 7) {$\putt{}_2$};
		\node [style=none] (1569) at (173.75, 6.75) {};
		\node [style=none] (1570) at (174.75, 7) {};
		\node [style=none] (1571) at (173.75, 7.25) {};
		\node [style=none] (1572) at (173.75, 8) {};
		\node [style=none] (1573) at (175, 6) {};
		\node [style=none] (1574) at (176, 5.75) {};
		\node [style=WIDEmorph] (1575) at (173, 2.5) {$T_2$};
		\node [style=none] (1576) at (173.25, 2.5) {};
		\node [style=none] (1577) at (174.25, -1) {};
		\node [style=none] (1578) at (173.5, 2.5) {};
		\node [style=WIDEmorph] (1579) at (174.75, -1) {$\putt{}_1$};
		\node [style=none] (1580) at (174.25, -1.25) {};
		\node [style=none] (1581) at (175.25, -1) {};
		\node [style=none] (1582) at (174.25, -3) {};
		\node [style=WIDEmorph] (1583) at (174.75, -3) {$T_1$};
		\node [style=none] (1584) at (175.25, -3) {};
		\node [style={white_dot}] (1585) at (178.25, -5.5) {};
		\node [style=none] (1586) at (178.25, -9.75) {};
		\node [style=morph] (1589) at (173.75, -8.25) {$g_2$};
		\node [style=none] (1590) at (173.75, -7) {};
		\node [style=WIDEmorph] (1591) at (174.25, -7) {$\putt{}_1$};
		\node [style=none] (1592) at (173.75, -7.25) {};
		\node [style=none] (1593) at (174.75, -7) {};
		\node [style={white_dot}] (1594) at (172.5, -7.5) {};
		\node [style={white_dot}] (1595) at (173, -9) {};
		\node [style={white_dot}] (1596) at (175.5, -8.25) {};
		\node [style=none] (1597) at (173, -10) {};
		\node [style=none] (1598) at (175.5, -9.75) {};
		\node [style={white_dot}] (1599) at (177, -7.5) {};
		\node [style=none] (1600) at (177, -9.75) {};
		\node [style=none] (1601) at (175, 5.5) {};
		\node [style=none] (1602) at (172.5, -0.25) {};
		\node [style=none] (1603) at (175.75, 5.75) {};
		\node [style=none] (1604) at (174.25, -3.25) {};
		\node [style=none] (1605) at (175, -3) {};
		\node [style=none] (1606) at (171.5, -0.5) {};
		\node [style=none] (1607) at (173.25, 2.5) {};
		\node [style=none] (1608) at (175, 5.5) {};
		\node [style=morph] (1609) at (175, 4.75) {$g_2$};
		\node [style={white_dot}] (1610) at (173.5, 4) {};
		\node [style=none] (1611) at (172.5, 2.75) {};
		\node [style=none] (1612) at (172.5, 2.25) {};
		\node [style=WIDEmorph] (1613) at (172, -0.25) {$\putt{}_2$};
		\node [style=none] (1614) at (172.5, -0.25) {};
		\node [style=none] (1615) at (171.5, 0) {};
		\node [style=morph] (1617) at (172.5, -2.25) {$g_2$};
		\node [style=morph] (1618) at (173, -4.75) {$g_2$};
		\node [style={white_dot}] (1619) at (171.5, -3.5) {};
		\node [style=none] (1620) at (179.25, -0.25) {$=$};
		\node [style=WIDEmorph] (1621) at (184.25, 6) {$T_1$};
		\node [style=none] (1622) at (183.75, 5.75) {};
		\node [style=WIDEmorph] (1623) at (183, 7.25) {$\putt{}_2$};
		\node [style=none] (1624) at (182.5, 7) {};
		\node [style=none] (1625) at (183.5, 7.25) {};
		\node [style=none] (1626) at (182.5, 7.5) {};
		\node [style=none] (1627) at (182.5, 8.25) {};
		\node [style=none] (1628) at (183.75, 6.25) {};
		\node [style=none] (1629) at (184.75, 6) {};
		\node [style=WIDEmorph] (1630) at (181.75, 2.75) {$T_2$};
		\node [style=none] (1631) at (182, 2.75) {};
		\node [style=none] (1632) at (183, -0.75) {};
		\node [style=none] (1633) at (182.25, 2.75) {};
		\node [style=WIDEmorph] (1634) at (183.5, -0.75) {$\putt{}_1$};
		\node [style=none] (1635) at (183, -1) {};
		\node [style=none] (1636) at (184, -0.75) {};
		\node [style=none] (1637) at (183, -2.75) {};
		\node [style=WIDEmorph] (1638) at (183.5, -2.75) {$T_1$};
		\node [style=none] (1639) at (184, -2.75) {};
		\node [style={white_dot}] (1640) at (186.25, -5.5) {};
		\node [style=none] (1641) at (187, -9.5) {};
		\node [style=morph] (1642) at (182.5, -8) {$g_2$};
		\node [style=none] (1643) at (182.5, -6.75) {};
		\node [style=WIDEmorph] (1644) at (183, -6.75) {$\putt{}_1$};
		\node [style=none] (1645) at (182.5, -7) {};
		\node [style=none] (1646) at (183.5, -6.75) {};
		\node [style={white_dot}] (1647) at (181.25, -7.25) {};
		\node [style={white_dot}] (1648) at (181.75, -8.75) {};
		\node [style={white_dot}] (1649) at (184.25, -8) {};
		\node [style=none] (1650) at (181.75, -9.75) {};
		\node [style=none] (1651) at (184.25, -9.5) {};
		\node [style={white_dot}] (1652) at (185.75, -7.25) {};
		\node [style=none] (1653) at (185.75, -9.5) {};
		\node [style=none] (1654) at (183.75, 5.75) {};
		\node [style=none] (1656) at (184.5, 6) {};
		\node [style=none] (1657) at (183, -3) {};
		\node [style=none] (1658) at (183.75, -2.75) {};
		\node [style=none] (1660) at (182, 2.75) {};
		\node [style=none] (1661) at (183.75, 5.75) {};
		\node [style=morph] (1662) at (183.75, 5) {$g_2$};
		\node [style={white_dot}] (1663) at (182.25, 4.25) {};
		\node [style=none] (1664) at (181.25, 3) {};
		\node [style=none] (1665) at (181.25, 2.5) {};
		\node [style=morph] (1670) at (181.75, -4.5) {$g_2$};
		\node [style=none] (1671) at (181.25, 2.5) {};
		\node [style=none] (1685) at (188.75, -2) {};
		\node [style=none] (1688) at (189, 0.25) {};
		\node [style=none] (1689) at (189.75, 0.5) {};
		\node [style=none] (1690) at (189.25, -2) {};
		\node [style=WIDEmorph] (1691) at (189.5, 0.5) {$T$};
		\node [style=none] (1692) at (189, 0.75) {};
		\node [style=none] (1693) at (190, 0.5) {};
		\node [style=none] (1694) at (190.75, -2) {};
		\node [style=none] (1695) at (189, 1.5) {};
		\node [style={white_dot}] (1697) at (190, -1.25) {};
		\node [style=none] (1698) at (190, -2) {};
		\node [style=none] (1699) at (187.5, -0.25) {$=$};
	\end{pgfonlayer}
	\begin{pgfonlayer}{edgelayer}
		\draw (1571.center) to (1572.center);
		\draw [in=90, out=-90, looseness=1.25] (1570.center) to (1573.center);
		\draw [in=90, out=-90] (1578.center) to (1577.center);
		\draw [in=90, out=-90, looseness=1.25] (1580.center) to (1582.center);
		\draw [style=new edge style 5, in=90, out=-90] (1585) to (1586.center);
		\draw [style=new edge style 5, in=135, out=-90, looseness=0.75] (1584.center) to (1585);
		\draw [style=new edge style 5, in=285, out=60, looseness=0.25] (1585) to (1574.center);
		\draw [style=new edge style 5, in=90, out=-90] (1581.center) to (1585);
		\draw [style=new edge style 2, in=135, out=-60] (1593.center) to (1596);
		\draw [style=new edge style 2] (1596) to (1598.center);
		\draw [in=150, out=-90] (1594) to (1595);
		\draw [in=-90, out=15, looseness=1.25] (1595) to (1589);
		\draw [in=270, out=90] (1589) to (1592.center);
		\draw (1595) to (1597.center);
		\draw [style=new edge style 6] (1599) to (1600.center);
		\draw [in=90, out=-90, looseness=0.25] (1607.center) to (1590.center);
		\draw [style=new edge style 2, in=30, out=-75, looseness=0.25] (1605.center) to (1596);
		\draw [style=new edge style 2, in=30, out=-90, looseness=0.25] (1603.center) to (1596);
		\draw [in=90, out=-90] (1608.center) to (1609);
		\draw [in=90, out=-90] (1610) to (1611.center);
		\draw [in=-90, out=0] (1610) to (1609);
		\draw [in=120, out=-90, looseness=1.25] (1569.center) to (1610);
		\draw [in=-90, out=90] (1615.center) to (1612.center);
		\draw [in=-90, out=180, looseness=0.75] (1594) to (1619);
		\draw (1619) to (1606.center);
		\draw [in=-90, out=0, looseness=1.50] (1619) to (1617);
		\draw (1617) to (1614.center);
		\draw [in=-90, out=90] (1618) to (1604.center);
		\draw [in=60, out=-90] (1618) to (1594);
		\draw (1626.center) to (1627.center);
		\draw [in=90, out=-90, looseness=1.25] (1625.center) to (1628.center);
		\draw [in=90, out=-90] (1633.center) to (1632.center);
		\draw [in=90, out=-90, looseness=1.25] (1635.center) to (1637.center);
		\draw [style=new edge style 5, in=90, out=-90] (1640) to (1641.center);
		\draw [style=new edge style 5, in=135, out=-90, looseness=0.75] (1639.center) to (1640);
		\draw [style=new edge style 5, in=285, out=60, looseness=0.25] (1640) to (1629.center);
		\draw [style=new edge style 5, in=90, out=-90] (1636.center) to (1640);
		\draw [style=new edge style 2, in=135, out=-60] (1646.center) to (1649);
		\draw [style=new edge style 2] (1649) to (1651.center);
		\draw [in=150, out=-90] (1647) to (1648);
		\draw [in=-90, out=15, looseness=1.25] (1648) to (1642);
		\draw [in=270, out=90] (1642) to (1645.center);
		\draw (1648) to (1650.center);
		\draw [style=new edge style 6] (1652) to (1653.center);
		\draw [in=90, out=-90, looseness=0.25] (1660.center) to (1643.center);
		\draw [style=new edge style 2, in=30, out=-75, looseness=0.25] (1658.center) to (1649);
		\draw [style=new edge style 2, in=30, out=-90, looseness=0.25] (1656.center) to (1649);
		\draw [in=90, out=-90] (1661.center) to (1662);
		\draw [in=90, out=-90] (1663) to (1664.center);
		\draw [in=-90, out=0] (1663) to (1662);
		\draw [in=120, out=-90, looseness=1.25] (1624.center) to (1663);
		\draw [in=-90, out=120, looseness=0.50] (1647) to (1671.center);
		\draw [in=-90, out=90] (1670) to (1657.center);
		\draw [in=60, out=-90] (1670) to (1647);
		\draw [in=270, out=90] (1685.center) to (1688.center);
		\draw [in=90, out=-90] (1689.center) to (1690.center);
		\draw [in=90, out=-90] (1693.center) to (1694.center);
		\draw (1692.center) to (1695.center);
		\draw [style=new edge style 6] (1698.center) to (1697);
	\end{pgfonlayer}
\end{tikzpicture}\]
\end{proof}

\section{Change-depending implies getput}
\begin{prop}[GetPut]
The restriction of $\mathcal{L}$ to $\mathbf{cdepTL}[\mathcal{C}]$ factors through $\mathbf{wbL}[\mathcal{C}]$
\end{prop}
\begin{proof}
All that is required is to demonstrate that the lens induced by any tagged lens with a change-depending tag is a well-behaved lens.
\[\begin{tikzpicture}[scale=0.7, {every node/.style}={scale=0.7}]
	\begin{pgfonlayer}{nodelayer}
		\node [style=WIDEmorph] (9) at (-2, 6.25) {$P$};
		\node [style=none] (10) at (-2.5, 6.5) {};
		\node [style=none] (11) at (-2.5, 6) {};
		\node [style=none] (13) at (-1.5, 6.25) {};
		\node [style=none] (16) at (-2.5, 7.25) {};
		\node [style=none] (17) at (-0.5, 6) {$=$};
		\node [style=morph] (59) at (-1.5, 5.5) {$g$};
		\node [style={white_dot}] (60) at (-2, 4.5) {};
		\node [style=none] (61) at (-2, 3.5) {};
		\node [style=WIDEmorph] (218) at (0.75, 7.25) {$T$};
		\node [style=none] (219) at (1, 7.25) {};
		\node [style=none] (220) at (1.25, 7.25) {};
		\node [style=morph] (221) at (1, 6) {$g$};
		\node [style=WIDEmorph] (222) at (1.5, 8.5) {$\putt{}$};
		\node [style=none] (223) at (0.25, 7.5) {};
		\node [style=none] (224) at (0.25, 7) {};
		\node [style={white_dot}] (225) at (2.25, 5.75) {};
		\node [style=none] (226) at (1, 8.25) {};
		\node [style=none] (227) at (2, 8.5) {};
		\node [style={white_dot}] (229) at (0.5, 5.25) {};
		\node [style=none] (231) at (1, 9.5) {};
		\node [style=none] (232) at (1, 8.75) {};
		\node [style=morph] (236) at (2.25, 4.25) {$g$};
		\node [style={white_dot}] (237) at (1.5, 3.25) {};
		\node [style=none] (238) at (1.5, 2.25) {};
		\node [style=none] (239) at (3.25, 6) {$=$};
		\node [style=WIDEmorph] (240) at (5.5, 8.75) {$T$};
		\node [style=none] (241) at (5.75, 8.75) {};
		\node [style=none] (242) at (6, 8.75) {};
		\node [style=morph] (243) at (4.75, 6) {$g$};
		\node [style=WIDEmorph] (244) at (4.25, 7.5) {$\putt{}$};
		\node [style=none] (245) at (3.75, 7.75) {};
		\node [style=none] (246) at (3.75, 7.25) {};
		\node [style={white_dot}] (247) at (6, 5.75) {};
		\node [style=none] (248) at (5, 8.5) {};
		\node [style=none] (249) at (4.75, 7.5) {};
		\node [style={white_dot}] (250) at (4.25, 5.25) {};
		\node [style=none] (251) at (5, 9.75) {};
		\node [style=none] (252) at (5, 9) {};
		\node [style=morph] (253) at (6, 4.25) {$g$};
		\node [style={white_dot}] (254) at (5.25, 3.25) {};
		\node [style=none] (255) at (5.25, 2.25) {};
		\node [style=none] (274) at (7, 6) {$=$};
		\node [style=WIDEmorph] (275) at (9.25, 8.75) {$T$};
		\node [style=none] (276) at (9.5, 8.75) {};
		\node [style=none] (277) at (9.75, 8.75) {};
		\node [style=morph] (278) at (8.5, 6) {$g$};
		\node [style=WIDEmorph] (279) at (8, 7.5) {$\putt{}$};
		\node [style=none] (280) at (7.5, 7.75) {};
		\node [style=none] (281) at (7.5, 7.25) {};
		\node [style={white_dot}] (282) at (9.75, 6.25) {};
		\node [style=none] (283) at (8.75, 8.5) {};
		\node [style=none] (284) at (8.5, 7.5) {};
		\node [style=none] (286) at (8.75, 9.75) {};
		\node [style=none] (287) at (8.75, 9) {};
		\node [style=morph] (288) at (9.75, 5.25) {$g$};
		\node [style=none] (290) at (9, 2.25) {};
		\node [style={white_dot}] (292) at (9, 3.25) {};
		\node [style={white_dot}] (293) at (9, 4.5) {};
		\node [style=none] (294) at (10.75, 6) {$=$};
		\node [style=WIDEmorph] (295) at (13, 8.75) {$T$};
		\node [style=none] (296) at (12.25, 7.5) {};
		\node [style=none] (297) at (13.5, 8.75) {};
		\node [style=morph] (298) at (12.5, 6.5) {$g$};
		\node [style=WIDEmorph] (299) at (11.75, 7.5) {$\putt{}$};
		\node [style=none] (300) at (11.25, 7.75) {};
		\node [style=none] (301) at (11.25, 7.25) {};
		\node [style=none] (303) at (12.5, 8.5) {};
		\node [style=none] (304) at (13.25, 8.75) {};
		\node [style=none] (305) at (12.5, 9.75) {};
		\node [style=none] (306) at (12.5, 9) {};
		\node [style=none] (308) at (12.75, 2.25) {};
		\node [style={white_dot}] (310) at (12.75, 3.25) {};
		\node [style={white_dot}] (311) at (13.5, 5.5) {};
		\node [style=morph] (312) at (13.5, 7.5) {$g$};
		\node [style=morph] (313) at (12, 5.25) {$g$};
		\node [style={white_dot}] (315) at (13.5, 4.25) {};
	\end{pgfonlayer}
	\begin{pgfonlayer}{edgelayer}
		\draw (10.center) to (16.center);
		\draw (60) to (61.center);
		\draw [in=-90, out=0] (60) to (59);
		\draw [in=180, out=-90, looseness=0.75] (11.center) to (60);
		\draw (59) to (13.center);
		\draw [in=90, out=-90] (219.center) to (221);
		\draw [in=270, out=90] (223.center) to (226.center);
		\draw [in=30, out=-90, looseness=0.50] (227.center) to (225);
		\draw [in=150, out=-75, looseness=0.75] (220.center) to (225);
		\draw [in=120, out=-90, looseness=0.75] (224.center) to (229);
		\draw [in=-90, out=0] (229) to (221);
		\draw (231.center) to (232.center);
		\draw (237) to (238.center);
		\draw [in=-90, out=0] (237) to (236);
		\draw [in=150, out=-90] (229) to (237);
		\draw (236) to (225);
		\draw [in=90, out=-90] (241.center) to (243);
		\draw [in=270, out=90] (245.center) to (248.center);
		\draw [in=60, out=-90, looseness=1.25] (249.center) to (247);
		\draw [in=150, out=-75, looseness=0.75] (242.center) to (247);
		\draw [in=120, out=-90, looseness=0.75] (246.center) to (250);
		\draw [in=-90, out=0] (250) to (243);
		\draw (251.center) to (252.center);
		\draw (254) to (255.center);
		\draw [in=-90, out=0] (254) to (253);
		\draw [in=150, out=-90] (250) to (254);
		\draw (253) to (247);
		\draw [in=90, out=-90] (276.center) to (278);
		\draw [in=270, out=90] (280.center) to (283.center);
		\draw [in=60, out=-90, looseness=1.25] (284.center) to (282);
		\draw [in=150, out=-75, looseness=0.75] (277.center) to (282);
		\draw (286.center) to (287.center);
		\draw (288) to (282);
		\draw [in=-90, out=135] (293) to (278);
		\draw (293) to (292);
		\draw (292) to (290.center);
		\draw [in=90, out=-90] (296.center) to (298);
		\draw [in=270, out=90] (300.center) to (303.center);
		\draw (305.center) to (306.center);
		\draw [in=-90, out=135, looseness=1.50] (311) to (298);
		\draw (310) to (308.center);
		\draw (311) to (312);
		\draw [in=-90, out=165, looseness=1.50] (315) to (313);
		\draw [in=-90, out=105, looseness=1.75] (313) to (304.center);
		\draw (312) to (297.center);
		\draw [in=-90, out=180, looseness=0.75] (292) to (281.center);
		\draw [in=180, out=-90, looseness=0.75] (301.center) to (310);
		\draw [in=30, out=-90, looseness=1.25] (288) to (293);
		\draw (311) to (315);
		\draw [in=0, out=-90, looseness=1.25] (315) to (310);
	\end{pgfonlayer}
\end{tikzpicture}\]
\[\begin{tikzpicture}[scale=0.7, {every node/.style}={scale=0.7}]
	\begin{pgfonlayer}{nodelayer}
		\node [style=WIDEmorph] (178) at (0.75, 1.75) {$T$};
		\node [style=none] (179) at (1, 1.75) {};
		\node [style=none] (180) at (1.25, 1.75) {};
		\node [style=none] (184) at (0.25, 1.5) {};
		\node [style={white_dot}] (186) at (1.5, 0.25) {};
		\node [style=none] (188) at (0.25, 2.5) {};
		\node [style=none] (189) at (0.25, 1.75) {};
		\node [style=morph] (191) at (1.5, -0.75) {$g$};
		\node [style={white_dot}] (192) at (1, -2) {};
		\node [style=none] (193) at (1, -2.75) {};
		\node [style=none] (194) at (2.5, 0) {$=$};
		\node [style=none] (198) at (3.25, 2) {};
		\node [style={white_dot}] (199) at (4.5, 0.75) {};
		\node [style=morph] (202) at (4.5, -0.25) {$g$};
		\node [style={white_dot}] (203) at (4, -1.5) {};
		\node [style=none] (204) at (4, -2.25) {};
		\node [style=none] (205) at (5.5, 0) {$=$};
		\node [style=none] (206) at (6.25, 2) {};
		\node [style={white_dot}] (210) at (7.5, -0.25) {};
		\node [style={white_dot}] (211) at (7, -1.5) {};
		\node [style=none] (212) at (7, -2.25) {};
		\node [style=none] (213) at (8.5, 0) {$=$};
		\node [style=none] (214) at (9.5, 1.25) {};
		\node [style=none] (217) at (9.5, -1.25) {};
		\node [style=none] (256) at (-4.5, 0) {$=$};
		\node [style=WIDEmorph] (257) at (-2.25, 2.75) {$T$};
		\node [style=none] (258) at (-2.75, 1) {};
		\node [style=none] (259) at (-2, 2.75) {};
		\node [style=morph] (260) at (-2.75, -0.5) {$g$};
		\node [style=WIDEmorph] (261) at (-3.25, 1) {$\putt{}$};
		\node [style=none] (262) at (-3.75, 1.25) {};
		\node [style=none] (263) at (-3.75, 0.75) {};
		\node [style={white_dot}] (264) at (-1.75, -0.25) {};
		\node [style=none] (265) at (-2.75, 2.5) {};
		\node [style=none] (266) at (-1.75, 2.75) {};
		\node [style={white_dot}] (267) at (-3.25, -1.25) {};
		\node [style=none] (268) at (-2.75, 3.75) {};
		\node [style=none] (269) at (-2.75, 3) {};
		\node [style=morph] (270) at (-1.75, -1.75) {$g$};
		\node [style={white_dot}] (271) at (-2.5, -2.75) {};
		\node [style=none] (272) at (-2.5, -3.75) {};
		\node [style=none] (273) at (-0.5, 0) {$=$};
		\node [style=none] (316) at (-8.25, 0) {$=$};
		\node [style=WIDEmorph] (317) at (-5.75, 2.5) {$T$};
		\node [style=none] (318) at (-6.5, 1.25) {};
		\node [style=none] (319) at (-5.25, 2.5) {};
		\node [style=morph] (320) at (-6.75, -0.75) {$g$};
		\node [style=WIDEmorph] (321) at (-7, 1.25) {$\putt{}$};
		\node [style=none] (322) at (-7.5, 1.5) {};
		\node [style=none] (323) at (-7.5, 1) {};
		\node [style=none] (324) at (-6.25, 2.25) {};
		\node [style=none] (325) at (-5.5, 2.5) {};
		\node [style=none] (326) at (-6.25, 3.5) {};
		\node [style=none] (327) at (-6.25, 2.75) {};
		\node [style=none] (328) at (-6, -4) {};
		\node [style={white_dot}] (329) at (-6, -3) {};
		\node [style={white_dot}] (330) at (-5.25, -0.75) {};
		\node [style=morph] (331) at (-5.25, 1.25) {$g$};
		\node [style=morph] (332) at (-6, 0.25) {$g$};
		\node [style={white_dot}] (333) at (-5.25, -2) {};
	\end{pgfonlayer}
	\begin{pgfonlayer}{edgelayer}
		\draw [in=-90, out=150] (186) to (179.center);
		\draw [in=45, out=-90, looseness=1.25] (180.center) to (186);
		\draw (188.center) to (189.center);
		\draw (186) to (191);
		\draw (192) to (193.center);
		\draw [in=-90, out=0, looseness=0.75] (192) to (191);
		\draw [in=180, out=-90, looseness=0.50] (184.center) to (192);
		\draw (199) to (202);
		\draw (203) to (204.center);
		\draw [in=-90, out=0, looseness=0.75] (203) to (202);
		\draw [in=180, out=-90, looseness=0.50] (198.center) to (203);
		\draw (211) to (212.center);
		\draw [in=-90, out=0, looseness=0.75] (211) to (210);
		\draw [in=180, out=-90, looseness=0.50] (206.center) to (211);
		\draw (217.center) to (214.center);
		\draw [in=90, out=-90] (258.center) to (260);
		\draw [in=270, out=90] (262.center) to (265.center);
		\draw [in=60, out=-90, looseness=1.25] (266.center) to (264);
		\draw [in=150, out=-90, looseness=0.75] (259.center) to (264);
		\draw [in=120, out=-90, looseness=0.75] (263.center) to (267);
		\draw [in=-90, out=0] (267) to (260);
		\draw (268.center) to (269.center);
		\draw (271) to (272.center);
		\draw [in=-90, out=0] (271) to (270);
		\draw [in=150, out=-90] (267) to (271);
		\draw (270) to (264);
		\draw [in=90, out=-90] (318.center) to (320);
		\draw [in=270, out=90] (322.center) to (324.center);
		\draw (326.center) to (327.center);
		\draw (329) to (328.center);
		\draw (330) to (331);
		\draw [in=-90, out=90, looseness=1.75] (332) to (325.center);
		\draw (331) to (319.center);
		\draw [in=180, out=-90, looseness=0.75] (323.center) to (329);
		\draw (330) to (333);
		\draw [in=0, out=-90, looseness=1.25] (333) to (329);
		\draw [in=120, out=-90] (320) to (333);
		\draw [in=-90, out=150] (330) to (332);
	\end{pgfonlayer}
\end{tikzpicture}.\]
\end{proof}

\end{document}